\documentclass[12pt,a4paper]{amsart}
\usepackage[czech,english]{babel}
\usepackage[T1]{fontenc}
\usepackage{amssymb,amsxtra}
\usepackage[all,cmtip]{xy}
\usepackage{hyperref}

\calclayout

\newcommand{\+}{\protect\nobreakdash-}
\newcommand{\<}{\protect\nobreakdash--}

\renewcommand{\:}{\colon}
\newcommand{\rarrow}{\longrightarrow}
\newcommand{\ot}{\otimes}
\newcommand{\ocn}{\odot}

\DeclareMathOperator{\Hom}{Hom}
\DeclareMathOperator{\Ext}{Ext}
\DeclareMathOperator{\Tor}{Tor}
\DeclareMathOperator{\Tot}{Tot}
\DeclareMathOperator{\cone}{cone}
\DeclareMathOperator{\Fil}{\mathsf{Fil}}
\DeclareMathOperator{\Spec}{Spec}

\DeclareFontFamily{U}{mathb}{\hyphenchar\font45}
\DeclareFontShape{U}{mathb}{m}{n}{
      <5> <6> <7> <8> <9> <10> gen * mathb
      <10.95> mathb10 <12> <14.4> <17.28> <20.74> <24.88> mathb12
      }{}
\DeclareSymbolFont{mathb}{U}{mathb}{m}{n}
\DeclareFontSubstitution{U}{mathb}{m}{n}
\DeclareMathSymbol{\blackdiamond}{0}{mathb}{"0C}

\newcommand{\bu}{{\text{\smaller\smaller$\scriptstyle\bullet$}}}
\newcommand{\cu}{{\text{\smaller$\scriptstyle\blackdiamond$}}}
\newcommand{\subcu}{{\text{\smaller$\scriptscriptstyle\blackdiamond$}}}

\newcommand{\lrarrow}{\mskip.5\thinmuskip\relbar\joinrel\relbar\joinrel
 \rightarrow\mskip.5\thinmuskip\relax}
\newcommand{\llarrow}{\mskip.5\thinmuskip\leftarrow\joinrel\relbar
 \joinrel\relbar\mskip.5\thinmuskip\relax}

\newcommand{\bModl}{{\operatorname{\mathbf{--Mod}}}}
\newcommand{\bModr}{{\operatorname{\mathbf{Mod--}}}}
\newcommand{\rModl}{{\operatorname{\mathrm{--Mod}}}}
\newcommand{\rModr}{{\operatorname{\mathrm{Mod--}}}}
\newcommand{\sModl}{{\operatorname{\mathsf{--Mod}}}}
\newcommand{\sModr}{{\operatorname{\mathsf{Mod--}}}}
\newcommand{\bmodl}{{\operatorname{\mathbf{--mod}}}}
\newcommand{\bmodr}{{\operatorname{\mathbf{mod--}}}}
\newcommand{\rmodl}{{\operatorname{\mathrm{--mod}}}}
\newcommand{\rmodr}{{\operatorname{\mathrm{mod--}}}}

\newcommand{\smodr}{{\operatorname{\mathsf{mod--}}}}

\newcommand{\bComodl}{{\operatorname{\mathbf{--Comod}}}}
\newcommand{\bcomodl}{{\operatorname{\mathbf{--comod}}}}
\newcommand{\bContra}{{\operatorname{\mathbf{--Contra}}}}
\newcommand{\bComodr}{{\operatorname{\mathbf{Comod--}}}}

\newcommand{\Ab}{\mathsf{Ab}}

\newcommand{\sA}{\mathsf A}
\newcommand{\sB}{\mathsf B}
\newcommand{\sC}{\mathsf C}
\newcommand{\sD}{\mathsf D}
\newcommand{\sE}{\mathsf E}
\newcommand{\sF}{\mathsf F}
\newcommand{\sH}{\mathsf H}
\newcommand{\sK}{\mathsf K}
\newcommand{\sL}{\mathsf L}
\newcommand{\sR}{\mathsf R}
\newcommand{\sS}{\mathsf S}
\newcommand{\sT}{\mathsf T}
\newcommand{\sW}{\mathsf W}
\newcommand{\sZ}{\mathsf Z}

\newcommand{\bA}{\mathbf A}

\newcommand{\bC}{\mathbf C}

\newcommand{\cC}{\mathcal C}
\newcommand{\cL}{\mathcal L}
\newcommand{\cM}{\mathcal M}
\newcommand{\cR}{\mathcal R}
\newcommand{\cS}{\mathcal S}
\newcommand{\cW}{\mathcal W}

\newcommand{\bb}{{\mathsf{b}}}
\newcommand{\bco}{{\mathsf{bco}}}
\newcommand{\bctr}{{\mathsf{bctr}}}
\newcommand{\co}{{\mathsf{co}}}
\newcommand{\ctr}{{\mathsf{ctr}}}
\newcommand{\ac}{{\mathsf{ac}}}
\newcommand{\abs}{{\mathsf{abs}}}

\newcommand{\inj}{{\mathsf{inj}}}
\newcommand{\proj}{{\mathsf{proj}}}
\renewcommand{\flat}{{\mathsf{flat}}}
\renewcommand{\cot}{{\mathsf{cot}}}
\newcommand{\fg}{{\mathsf{fg}}}
\newcommand{\fp}{{\mathsf{fp}}}
\newcommand{\gs}{{\mathsf{gs}}}
\newcommand{\st}{{\mathsf{st}}}
\newcommand{\cmp}{{\mathsf{cmp}}}

\newcommand{\binj}{{\mathbf{inj}}}
\newcommand{\bproj}{{\mathbf{proj}}}
\newcommand{\bfree}{{\mathbf{free}}}
\newcommand{\bflat}{{\mathbf{flat}}}
\newcommand{\bcot}{{\mathbf{cot}}}
\newcommand{\bfp}{{\mathbf{fp}}}

\newcommand{\rop}{{\mathrm{op}}}
\newcommand{\sop}{{\mathsf{op}}}

\newcommand{\id}{{\mathrm{id}}}

\newcommand{\bModrinj}{{\operatorname{\mathbf{Mod_\binj--}}}}
\newcommand{\sModrinj}{{\operatorname{\mathsf{Mod_\inj--}}}}
\newcommand{\bModrfp}{{\operatorname{\mathbf{Mod_\bfp--}}}}
\newcommand{\bModrproj}{{\operatorname{\mathbf{Mod_\bproj--}}}}
\newcommand{\bmodrproj}{{\operatorname{\mathbf{mod_\bproj--}}}}

\newcommand{\boL}{\mathbb L}
\newcommand{\boR}{\mathbb R}
\newcommand{\boZ}{\mathbb Z}
\newcommand{\boQ}{\mathbb Q}

\newcommand{\Section}[1]{\bigskip\section{#1}\medskip}

\theoremstyle{plain}
\newtheorem{thm}{Theorem}[section]

\newtheorem{lem}[thm]{Lemma}
\newtheorem{prop}[thm]{Proposition}
\newtheorem{cor}[thm]{Corollary}
\theoremstyle{definition}
\newtheorem{ex}[thm]{Example}

\newtheorem{rem}[thm]{Remark}

\begin{document}

\author{Leonid Positselski}

\address{Leonid Positselski, Institute of Mathematics, Czech Academy
of Sciences \\ \v Zitn\'a~25, 115~67 Praha~1 \\ Czech Republic}

\email{positselski@math.cas.cz}

\author{Jan \v S\v tov\'\i\v cek}

\address{Jan {\v S}{\v{t}}ov{\'{\i}}{\v{c}}ek, Charles University,
Faculty of Mathematics and Physics, Department of Algebra,
Sokolovsk\'a 83, 186 75 Praha, Czech Republic}

\email{stovicek@karlin.mff.cuni.cz}

\title{Contraderived categories of CDG-modules}

\begin{abstract}
 For any CDG\+ring $B^\cu=(B^*,d,h)$, we show that the homotopy category
of graded-projective (left) CDG\+modules over $B^\cu$ is equivalent
to the quotient category of the homotopy category of graded-flat
CDG\+modules by its full triangulated subcategory of flat CDG\+modules.
 The \emph{contraderived category} (\emph{in the sense of Becker})
$\sD^\bctr(B^\cu\bModl)$ is the common name for these two
triangulated categories.
 We also prove that the classes of cotorsion and graded-cotorsion
CDG\+modules coincide, and the contraderived category of CDG\+modules
is equivalent to the homotopy category of graded-flat graded-cotorsion
CDG\+modules.
 Assuming the graded ring $B^*$ to be graded right coherent, we show
that the contraderived category $\sD^\bctr(B^\cu\bModl)$ is compactly
generated and its full subcategory of compact objects is
anti-equivalent to the full subcategory of compact objects in
the coderived category of right CDG\+modules $\sD^\bco(\bModr B^\cu)$.
 Specifically, the latter triangulated category is the idempotent
completion of the absolute derived category of finitely presented
right CDG\+modules $\sD^\abs(\bmodr B^\cu)$.
\end{abstract}

\maketitle

\tableofcontents

\section*{Introduction}
\medskip

\setcounter{subsection}{-1}
\subsection{Brief summary}\label{brief-summary-subsecn}
 One of the main results of Neeman's paper~\cite[Proposition~8.1 and 
Theorem~8.6]{Neem} can be summarized by saying that \emph{for the exact
category of flat modules over a ring, Becker's contraderived category~\cite{Bec}
agrees with the conventional derived category}.
 This result is closely related to the \emph{flat/projective
periodicity} phenomenon, originally discovered
in~\cite[Theorem~2.5]{BG}.

 Furthermore, the \emph{cotorsion periodicity} theorem, proved
by Bazzoni, Cort\'es-Izurdiaga, and Estrada in~\cite[Theorem~1.2(2),
Proposition~4.8(2), or Theorem~5.1(2)]{BCE}, implies that \emph{for
the exact category of flat modules over a ring, Becker's coderived
category agrees with the conventional derived category} as well.
 So all the three triangulated categories are the same.
 In this paper, we generalize these results to curved DG\+modules over
an arbitrary curved DG\+ring, which in particular encompasses
the settings of DG\+modules over ordinary DG\+rings and matrix
factorizations of modules over a ring.

 Another result from Neeman's
paper~\cite[Propositions~7.12 and~7.14]{Neem} says that Becker's
contraderived category $\sD^\bctr(R\rModl)=\sK(R\rModl_\proj)$ of left
modules over a right coherent ring $R$ is compactly generated, and
provides a description of the full subcategory of compact objects in
$\sD^\bctr(R\rModl)$ as the opposite category to the bounded derived
category $\sD^\bb(\rmodr R)$ of finitely presented right $R$\+modules.
 In this paper, we generalize this result to left CDG\+modules over
any CDG\+ring $B^\cu=(B^*,d,h)$ whose underlying graded ring $B^*$
is graded right coherent.


 A reader who is not familiar with the concepts or the terminology can read through the rest of this introduction about motivation, a more detailed explanation, and significance of these results.

\subsection{What is the derived category?}
 Classically, this question has a straightforward answer.
 The derived category of modules over a ring or sheaves on a space
is defined as the category of complexes up to quasi-isomorphism.
 Here by a quasi-isomorphism one means a morphism of complexes
inducing an isomorphism on the cohomology objects.
 This is the definition going back to Grothendieck and
Verdier~\cite[Exemples~I.1.2.3 and n${}^\circ$~II.1.1]{Ver}.

 Similarly one defines the derived category of differential graded
(DG) modules over a DG\+ring~\cite[Section~4.1]{Kel} or a sheaf
of DG\+rings.
 \emph{Curved differential graded} (\emph{CDG}) rings and
modules~\cite[Section~1]{Pcurv}, \cite[Section~3.1]{Pkoszul},
\cite[Section~1.1]{EP}, \cite[Sections~3.2 and~6.1]{Prel},
\cite[Section~6]{Pksurv} are natural generalizations of DG\+rings
and DG\+modules in many respects, but their differentials do \emph{not}
square to zero, so they do not have cohomology modules.
 Accordingly, the conventional notion of a quasi-isomorphism is
\emph{undefined} for curved DG\+structures.

 The concepts of derived categories naturally applicable to
CDG\+modules are the \emph{derived categories of the second
kind}~\cite[Remark~9.2]{PS4}, \cite[Section~7]{Pksurv}.
 Here by the \emph{derived category of the first kind} one means
the conventional derived category, as in~\cite{Ver,Spal,Neem-de,Kel}.
 ``Derived categories of the second kind'' are an umbrella term uniting
several constructions; the most important ones are called
the \emph{coderived}, the \emph{contraderived}, and the \emph{absolute
derived} category.

 Even so, there are two approaches to defining the coderived and
contraderived categories.
 What we call the \emph{coderived} and \emph{contraderived categories
in the sense of Positselski} were introduced in~\cite{Psemi,Pkoszul,KLN}
and perfected in~\cite[Sections~1.4\<1.6]{EP}.
 What we call the \emph{coderived} and \emph{contraderived categories
in the sense of Becker} were worked out in~\cite[Proposition~1.3.6]{Bec}
and go back to the papers~\cite{Jor,Kra,Neem}.

 The coderived and contraderived categories in the sense of
Positselski, as well as the absolute derived categories, are centered about finding a suitable replacement of quasi-isomorphisms.
 Roughly speaking, the absolute derived category is the maximal Verdier
quotient of the homotopy category of CDG\+modules making all short exact
sequences of CDG\+modules to become distinguished triangles.
 The Positselski coderived category is the maximal quotient category
which in addition preserves coproducts, and the Positselski
contraderived category is the maximal such quotient which in addition
preserves products.

The coderived and contraderived categories in the sense of
Becker are, on the other hand, centered about classes of CDG\+modules which can act as resolutions. The Becker contraderived category can be defined as the homotopy category of graded-projective CDG\+modules (which, for an ordinary ring, is none other than the homotopy category of unbounded complexes of projective modules).
Dually, the Becker coderived category is the homotopy category of graded-injective CDG\+modules. In fact, Becker showed in~\cite{Bec} that both can be again obtained as suitable Verdier quotients of the homotopy category of CDG\+modules. In the case of Becker coderived categories, the present authors described in~\cite[\S7.5]{PS5} the class of objects which one quotients out as the minimal subcategory of CDG\+modules which contains all contractible CDG\+modules and is closed under extensions and direct limits. The Becker contraderived category is more complicated and is the subject of study in this paper.

Similar ideas as above also apply to subcategories of CDG\+modules given by a certain property of the underlying graded modules. A key case for us, alluded to in~\S\ref{brief-summary-subsecn}, is the one of graded-flat CDG\+modules (i.~e.\ those CDG\+modules whose underlying graded module is flat over the underlying graded ring). As the properties which we consider are given by extension-closed subcategories of graded modules, this is also related to the theory of coderived and contraderived categories of \emph{exact} categories.
In the context of derived categories of the second kind in the sense of Positselski, this idea was explored already in~\cite[Section~1.3]{EP} and~\cite[Section~5.1]{Pedg}.
The derived categories of the second kind in the sense of Becker are less understood.
Their constructions are usually straightforward even if
technical~\cite[Section~B.7]{Pcosh}, but their properties may be unobvious and unexpected.

 It is still an open problem whether the constructions of
the first-named author of the present paper and the constructions of
Becker \emph{ever} produce different outputs within their natural
common domain of definition.
 But for all we know, the coderived and contraderived categories in
the sense of Becker may be better behaved (at least, we can prove
stronger results for them in greater generality).

\subsection{Matrix factorizations as motivating examples}
\label{introd-matrix-factorizations}
 \emph{Matrix factorizations} are an important and popular example of
CDG\+modules (or sheaves of CDG\+modules) arising in the context of algebraic geometry and
modern mathematical physics~\cite{DK,KL,Or1}.
 Classically, a matrix factorization of $w\in R$ for a regular local commutative Noetherian ring $R$, as introduced in~\cite{Eis}, is a pair of maps between free $R$\+modules $(d_0\:F^0\to F^1, d_1\:F^1\to F^0)$ such that $d_1d_0=w\cdot\id_{F^0}$ and $d_0d_1=w\cdot\id_{F^1}$. This construction generalizes naturally in that we can replace $\Spec{R}$ by a more general (not necessarily affine or regular) algebraic variety or scheme, let $w$ be a section of an invertible sheaf on $X$, and allow $F^0$ and $F^1$ to be chosen from other classes of quasi-coherent sheaves.
 However, in this generality, the task of defining appropriate derived categories of matrix factorizations becomes rather nontrivial~\cite{PL,Or2}.
 A straightforward technical attempt to define a direct analogue of
the conventional derived category of quasi-coherent sheaves for matrix
factorizations fails~\cite[Section~3.3]{EP}, as the Thomason--Trobaugh
theorem~\cite[Key Proposition~5.2.2]{TT} is not true for matrix
factorizations.

In this context, flat quasi-coherent factorizations (i.~e.\ those where the free modules $F^0, F^1$ above are replaced by flat quasi-coherent sheaves) are one of the natural objects of interest.
 Over Noetherian schemes of finite Krull dimension, the construction
of the coderived category in the sense of Positselski works well
for flat quasi-coherent factorizations; and even the absolute derived
category is sufficient, due to the fact that the exact category of
flat quasi-coherent sheaves has finite homological dimension
(by the Raynaud--Gruson theorem~\cite[Corollaire~II.3.2.7]{RG}).
 The exposition in~\cite{EP} implicitly utilizes this observation
(see~\cite[Remarks~1.4, 2.4, and~2.6]{EP}).

 But for more complicated schemes, for all we know, the outputs
produced by the first-named author's constructions of derived
categories of the second kind for flat quasi-coherent factorizations
may differ from the desired ones.
 As already mentioned, over an affine scheme, or in other words, over
a commutative ring, and for the usual unbounded complexes of modules,
it is not known whether Positselski's coderived category of flat
modules agrees with the homotopy category of
projective modules.
 One approach to work around this problem was suggested in~\cite[Remark~2.6]{EP}:
over non-Noetherian schemes, use \emph{very flat} quasi-coherent
factorizations instead of the flat ones.
 There are enough very flat quasi-coherent sheaves on a quasi-compact
semi-separated scheme, and the exact category of very flat quasi-coherent
sheaves over such a scheme has finite homological dimension;
so Positselski's coderived and absolute derived category constructions
(agree with each other and) produce the desired output.
 This became one of the motivations for the importance of
the \emph{Very Flat Conjecture}, proved in~\cite[Main Theorem~1.1 and
Sections~1.3, 2, and~8]{PSl}; see the discussion in~\cite[Section~0.2
of the Introduction]{PSl}.

 Here instead, we focus on Becker's approach to derived categories of the second kind for matrix factorizations.
 Although this is rather non-obvious, this works very well for affine schemes, 
as we prove in this paper
(see Examples~\ref{matrix-factorizations-example}
and~\ref{matrix-factorizations-example-contd}).
 As to the flat quasi-coherent factorizations over nonaffine schemes,
the usual definition of Becker's contraderived category makes no sense
for them, due to the lack of projective sheaves.
 However, the definition of what we call below the \emph{flat contraderived category} is applicable, and so is the definition of Becker's coderived category.
 We expect that the outputs of the two constructions should be related exactly as in the affine case and that
this presumably follows from what might be viewed as a version of cotorsion
periodicity theorem for quasi-coherent matrix factorizations
over schemes.
 But this topic falls outside of the scope of the present paper,
where only the affine case is worked out.

\subsection{Compact generators as motivating problem}
 Let $R$ be a ring.
 The study of the homotopy category of (unbounded complexes of)
projective $R$\+modules was initiated by J\o rgensen~\cite{Jor}.
 It was shown in~\cite{Jor} that, under certain assumptions,
the homotopy category $\sK(R\rModl_\proj)$ of complexes of projective
left $R$\+modules is compactly generated, and the full subcategory of
compact objects in $\sK(R\rModl_\proj)$ is anti-equivalent to
the bounded derived category of finitely presented right $R$\+modules
$\sD^\bb(\rmodr R)$ \,\cite[Theorem~3.2]{Jor}.

 In a subsequent paper by Neeman~\cite{Neem}, it was shown that
J\o rgensen's assumptions can be relaxed, and it suffices to assume
the ring $R$ to be
right coherent~\cite[Propositions~7.12 and~7.14]{Neem}.
 Generally speaking, however (for an arbitrary ring $R$), the homotopy
category $\sK(R\rModl_\proj)$ need not be compactly generated, as Neeman's
counterexample~\cite[Example~7.16]{Neem} illustrates.
 
 Another important result proved in Neeman's paper~\cite{Neem} is
a triangulated equivalence between the homotopy category of
projective modules and the derived category of the exact category
of flat modules.
 This result holds for an arbitrary ring~$R$
\,\cite[Proposition~8.1 and Theorem~8.6]{Neem}.
Our aim is to generalize both these results from ordinary rings to CDG\+rings (and so in particular to categories of affine matrix factorizations from the last subsection).

\smallskip

 To give the reader a better idea and also because it plays a role in
our results, we recall the story and the state of the art for
the dual-analogous situation of the homotopy category of injective
modules.
 The homotopy category of injective objects in a locally Noetherian
abelian category was studied by Krause in~\cite{Kra}.
 It was shown in~\cite{Kra} that the homotopy category $\sK(\sA_\inj)$
of injective objects in a locally Noetherian abelian category $\sA$
is compactly generated, and the full subcategory of compact objects
in $\sK(\sA_\inj)$ is equivalent to the bounded derived category
$\sD^\bb(\sA_\fg)$ of the abelian category of finitely generated
objects in~$\sA$ \,\cite[Proposition~2.3]{Kra}.

 Subsequently, the study of homotopy categories of injectives
was extended to locally coherent abelian categories in
the preprint~\cite{Sto2} by the second-named author of
the present paper.
 In the context of locally coherent abelian categories $\sA$, it was
shown in~\cite{Sto2} that the homotopy category of injective objects
$\sK(\sA_\inj)$ is compactly generated, and the full subcategory of
compact objects in $\sK(\sA_\inj)$ is equivalent to the bounded derived
category $\sD^\bb(\sA_\fp)$ of the abelian category of finitely
presented objects in~$\sA$ \,\cite[Corollary~6.13]{Sto2}.

 On the other hand, an argument, largely due to Arinkin, extending
the mentioned result of Krause to CDG\+rings with one-sided Noetherian
underlying graded rings was presented in
the memoir~\cite[Section~3.11]{Pkoszul} by the first-named author
of the present paper.
 In the preprint~\cite{Pedg}, this result was further extended to
\emph{abelian DG\+categories}, or rather more specifically to
locally Noetherian and certain locally coherent abelian
DG\+categories~\cite[Theorems~9.23 and~9.39]{Pedg}.

 Finally, in the paper~\cite{PS5} by the present authors,
a generalization of the mentioned results
from~\cite{Kra,Pkoszul,Sto2,Pedg} to arbitrary locally coherent
abelian DG\+categories was established.
 According to~\cite[Theorem~8.19]{PS5}, for any locally coherent
abelian DG\+category~$\bA$, the homotopy category $\sH^0(\bA_\binj)$
of the full DG\+subcategory $\bA_\binj$ of graded-injective objects
in $\bA$ is compactly generated, and the full subcategory of
compact objects in $\sH^0(\bA_\binj)$ is equivalent to
the idempotent completion of the absolute derived category
$\sD^\abs(\bA_\bfp)$ of the abelian DG\+subcategory $\bA_\bfp$
of finitely presented objects in~$\bA$.
 In particular, for any CDG\+ring $B^\cu=(B^*,d,h)$ with
a graded right coherent underlying graded ring $B^*$, the homotopy
category $\sH^0(\bModrinj B^\cu)$ of graded-injective right
CDG\+modules over $B^\cu$ is compactly generated, and the full
subcategory of compact objects in  $\sH^0(\bModrinj B^\cu)$ is
equivalent to the idempotent completion of the absolute derived
category $\sD^\abs(\bmodr B^\cu)=\sD^\abs(\bModrfp B^\cu)$ of finitely
presented right CDG\+modules over~$B^\cu$.

\subsection{Description of the results}
\label{introd-description-of-results-subsecn}
 As indicated, our aim is to generalize the above-mentioned
results of J\o rgensen and Neeman~\cite{Jor,Neem} about the homotopy
categories of projective modules to the homotopy categories of
graded-projective CDG\+modules over CDG\+rings.
 So, there are two main results in this paper, which we complement by examples of particular settings where they apply in~\S\ref{examples-Dbctr-subsecn} and~\S\ref{examples-compacts-subsecn}.

 The first theorem claims an equivalence between what we call
the \emph{projective contraderived category} and the \emph{flat
contraderived category} of CDG\+modules over an arbitrary
CDG\+ring $B^\cu=(B^*,d,h)$.
 We denote by $B^\cu\bModl_\bproj$ the DG\+category of graded-projective
left CDG\+modules and by $B^\cu\bModl_\bflat$ the DG\+category of
graded-flat left CDG\+modules over $B^\cu$, while
$\sH^0(B^\cu\bModl)_\flat\subset\sH^0(B^\cu\bModl_\bflat)$ denotes
the full triangulated subcategory of flat CDG\+modules.
 
\begin{thm}[Theorem~\ref{projective-and-flat-contraderived-theorem}]
\label{introd-projective-and-flat-contraderived-theorem}
 Let $B^\cu=(B^*,d,h)$ be a CDG\+ring.
 Then the composition of the fully faithful inclusion of triangulated
categories\/ $\sH^0(B^\cu\bModl_\bproj)\rarrow\sH^0(B^\cu\bModl_\bflat)$
and the triangulated Verdier quotient functor\/
$\sH^0(B^\cu\bModl_\bflat)\rarrow
\sH^0(B^\cu\bModl_\bflat)/\sH^0(B^\cu\bModl)_\flat$ is a triangulated
equivalence
$$
 \sH^0(B^\cu\bModl_\bproj)\simeq
 \sH^0(B^\cu\bModl_\bflat)/\sH^0(B^\cu\bModl)_\flat.
$$
\end{thm}


 In addition, we consider the DG\+category $B^\cu\bModl_\bflat^\bcot$
of graded-flat graded-cotorsion CDG\+modules.
 Its homotopy category $\sH^0(B^\cu\bModl_\bflat^\bcot)$ can be thought
of as \emph{Becker's coderived category of the exact DG\+category of
graded-flat CDG\+modules over~$B^\cu$}.
 The next theorem claims that it is equivalent to the contraderived
category.

\begin{thm}[Theorem~\ref{graded-flat-coderived-contraderived-theorem}]
\label{introd-graded-flat-coderived-contraderived-theorem}
 Let $B^\cu=(B^*,d,h)$ be a CDG\+ring.
 Then the composition of the fully faithful inclusion of triangulated
categories\/ $\sH^0(B^\cu\bModl_\bflat^\bcot)\rarrow
\sH^0(B^\cu\bModl_\bflat)$ and the triangulated Verdier quotient
functor\/ $\sH^0(B^\cu\bModl_\bflat)\rarrow
\sH^0(B^\cu\bModl_\bflat)/\sH^0(B^\cu\bModl)_\flat$ is a triangulated
equivalence \hbadness=1325
$$
 \sH^0(B^\cu\bModl_\bflat^\bcot)\simeq
 \sH^0(B^\cu\bModl_\bflat)/\sH^0(B^\cu\bModl)_\flat.
$$
\end{thm}

We can summarize these two theorems in the form of a recollement 
\begin{equation}\label{recollement-for-graded-flat-CDG-modules}
\xymatrix{
\sH^0(B^\cu\bModl)_\flat
\ar@{>->}[rr]|-\hole|-{i_*} &&
\sH^0(B^\cu\bModl_\bflat)
\ar@{->>}[rr]|-\hole|-{j^*}
\ar@{->>}@/_3ex/[ll]_-{i^*}
\ar@{->>}@/^3ex/[ll]^-{i^!} &&
\sD^\bctr(B^\cu\bModl),
\ar@{>->}@/_3ex/[ll]_-{j_!}
\ar@{>->}@/^3ex/[ll]^-{j_*}
}
\end{equation}
where the essential image of $j_!$ is $\sH^0(B^\cu\bModl_\bproj)$ and the essential image of $j_*$ is $\sH^0(B^\cu\bModl_\bflat^\bcot)$.
Moreover, as explained in Section~\ref{the-Quillen-equivalence-secn}, these equivalences are induced by Quillen equivalences of certain Quillen model categories, so we obtain not only triangulated equivalences, but full-fledged equivalences of the corresponding stable homotopy theories.

\smallskip

 Secondly, for any CDG\+ring $B^\cu$, we also construct a natural functor
of \emph{tensor product pairing} of right and left
CDG\+modules~\eqref{derived-tensor-product}
$$
 \ot_{B^*}^\boL\:\sD^\bco(\bModr B^\cu)\times\sD^\bctr(B^\cu\bModl)
 \lrarrow\sD(\Ab)
$$
taking values in the derived category of abelian groups $\sD(\Ab)$.
 Furthermore, assuming that the graded ring $B^*$ is graded right
coherent, we construct a contravariant triangulated
functor~\eqref{functor-Xi-on-absolute-derived-category}
$$
 \Xi_{B^\cu}\:\sD^\abs(\bmodr B^\cu)^\sop\lrarrow
 \sD^\bctr(B^\cu\bModl)
$$
from the absolute derived category of finitely presented right
CDG\+modules over $B^\cu$ to the contraderived category of left
CDG\+modules.
 Here is our second main theorem.

\begin{thm}[Theorems~\ref{image-of-Xi-compact-generators-thm}
and~\ref{Xi-fully-faithful-theorem}, and
Corollary~\ref{summary-on-compact-generators-cor}]
\label{introd-compact-generators-theorem}
 Let $B^\cu=(B^*,d,h)$ be a CDG\+ring such that the graded ring $B^*$
is graded right coherent.
 Then the contraderived category of left CDG\+modules\/
$\sD^\bctr(B^\cu\bModl)$ is compactly generated.
 The contravariant triangulated functor\/~$\Xi_{B^\cu}$ is fully
faithful, its image consists of compact objects in\/
$\sD^\bctr(B^\cu\bModl)$ and generates this triangulated category.

 Furthermore, the absolute derived category\/ $\sD^\abs(\bmodr B^\cu)$
is naturally a full triangulated subcategory of the coderived category\/
$\sD^\bco(\bModr B^\cu)$, consisting of compact objects and generating
the coderived category.
 Restricting the pairing\/ ${-}\ot_{B^*}^\boL{-}$ above to
the corresponding subcategories of compact objects and identifying
them with\/ $\sD^\abs(\bmodr B^\cu)^\sop$ and\/
$\sD^\abs(\bmodr B^\cu)$, respectively, one simply obtains
the derived Hom-functor
$$
 \boR\Hom_{B^*}^\bu(-,-)\:
 \sD^\abs(\bmodr B^\cu)^\sop\times\sD^\abs(\bmodr B^\cu)
 \lrarrow\sD(\Ab).
$$
\end{thm}


\subsection{One terminological convention}
 Throughout the rest of this paper, unless otherwise mentioned,
whenever speaking of the coderived or the contraderived category
we always presume derived categories of the second kind \emph{in
the sense of Becker} (and \emph{not} in the sense of Positselski).
 Derived categories of the second kind in the sense of Positselski
will only appear in our exposition in connection with some results
from his preceding work~\cite{Pkoszul,Prel,Pksurv} mentioned in
the examples.
 See~\cite[Remark~9.2]{PS4} and~\cite[Section~7.9]{Pksurv} for
a discussion.

\medskip\noindent
\textbf{Acknowledgement.}
 The authors are grateful to Michal Hrbek for helpful discussions.
 This research is supported by the GA\v CR project 23-05148S.
 The first-named author is also supported by
the Czech Academy of Sciences (RVO~67985840).

\Section{Preliminaries on CDG-Modules, Notation and Conventions}

 The exposition in this section, which contains preliminary material
on CDG\+mod\-ules and DG\+categories of CDG\+modules,
follows~\cite[Sections~1.1, 1.2 and~3.1]{Pkoszul},
\cite[Sections~3.2, 4.2, and~6.1]{Prel}, \cite[Sections~1, 2.2, 3.1,
3.5, and~9.2]{Pedg}, and~\cite[Sections~1, 2.6, 6.2, and~8.3]{PS5}.
 We suggest in particular the book~\cite{Prel} and
the long preprint~\cite{Pedg} as reference sources offering more
detailed discussions than the one below. 

 We start with a very brief discussion of terminology related to
direct limits.

\subsection{Direct limits}
 A nonempty poset $\Xi$ is said to be \emph{directed} if for any two
elements $\xi_1$ and $\xi_2\in\Xi$ there exists an element $\xi\in\Xi$
such that $\xi_1\le\xi$ and $\xi_2\le\xi$.
 Given an infinite regular cardinal~$\aleph$, a poset $\Xi$ is said
to be \emph{$\aleph$\+directed} if for any subset $\Upsilon\subset\Xi$
such that the cardinality of $\Upsilon$ is smaller than~$\aleph$
there exists an element $\xi\in\Xi$ such that $\upsilon\le\xi$
for every $\upsilon\in\Upsilon$.
 We refer to~\cite[Definition~1.4 and Remark~1.21]{AR} for
the definitions of a \emph{filtered category} and
an \emph{$\aleph$\+filtered category}.

 In the category-theoretic terminology of~\cite{AR}, one speaks of
colimits of diagrams indexed by directed posets as ``directed colimits''
and of the colimits of diagrams indexed by $\aleph$\+directed posets as
``$\aleph$\+directed colimits''.
 The colimits of diagrams indexed by filtered (small) categories
are called ``filtered colimits'', and the colimits of diagrams indexed
by $\aleph$\+filtered categories are referred to as ``$\aleph$\+filtered
colimits''.
 Using the construction of~\cite[Theorem~1.5]{AR}, all filtered
colimits can be reduced to directed colimits.
 Similarly, all $\aleph$\+filtered colimits can be reduced to
$\aleph$\+directed colimits~\cite[Remark~1.21]{AR}.
 So $\aleph$\+directed and $\aleph$\+filtered colimits are equivalent
in a suitable sense.

 As this is mostly a paper about modules, we stick to the (perhaps
less developed, but traditional) module-theoretic terminology.
 So we speak of colimits of diagrams indexed by directed posets as
\emph{direct limits} and of the colimits of diagrams indexed by
$\aleph$\+directed posets as \emph{$\aleph$\+direct limits}.

\subsection{Graded modules} \label{prelim-graded-modules}
 In this paper we consider graded rings and modules with the grading
by the abelian group of integers~$\boZ$
(but most results are valid also for a more general grading,
see~\cite[Remark in Section~1.1]{Pkoszul}).
Given a graded ring $B^*=\bigoplus_{i\in\boZ}B^i$, the abelian category
of graded left $B^*$\+modules $M^*=\bigoplus_{i\in\boZ}M^i$ is denoted
by $B^*\sModl$, and the abelian category of graded right $B^*$\+modules
$N^*=\bigoplus_{i\in\boZ}N^i$ is denoted by $\sModr B^*$.

 Given a graded ring $B^*$, the opposite graded ring $B^\rop{}^*$ is
defined by the rules $(B^\rop){}^n=B^n$ for all $n\in\boZ$ and
$a^\rop b^\rop=(-1)^{|a||b|}(ba)^\rop$ for all homogeneous elements
$a$ and $b$ of degrees $|a|$ and $|b|$ in~$B^*$.
 Then the abelian category of graded right $B^*$\+modules is naturally
identified with the abelian category of graded left
$B^\rop{}^*$\+modules and vice versa. 
 We will mostly use the construction of the opposite graded ring
$B^\rop{}^*$ as a notational device (as explained below).

 Given a graded left $B^*$\+module $M^*$ and an integer $n\in\boZ$,
the graded left $B^*$\+module $M^*[n]$ has the graded components
$M^*[n]^i=M^{n+i}$ for all $i\in\boZ$.
 Similarly, for any graded right $B^*$\+module $N^*$, the graded
right $B^*$\+module $N^*[n]$ has the graded components
$N^*[n]^i=N^{n+i}$.
 The action of the graded ring $B^*$ on the grading shift $M^*[n]$
of a graded left $B^*$\+module~$M^*$ has to be twisted by a sign. The
rule is that the action of a homogeneous element $b\in B$ on $m\in M[n]$
is given by $(-1)^{|b|n}bm$, where the formula involves the action of
$b$ on $M$. To the contrary, the right action of $B^*$ on $N[n]$ is 
identical to that on $N$. This convention follows the usual Koszul sign 
rules making the identity maps $\id_M\:M\to M[n]$ and $\id_N\:N\to N[n]$ 
graded homomorphisms of degree $-n$.  We also refer
to~\cite[Section~1.1]{Pkoszul} or~\cite[Section~6.1]{Prel}.

 Given two graded left $B^*$\+modules $L^*$ and $M^*$, the graded
abelian group of morphisms $\Hom_{B^*}^*(L^*,M^*)$ is defined by
the rule that, for every $n\in\boZ$, the grading component
$\Hom_{B^*}^n(L^*,M^*)$ is the group of all graded left $B^*$\+module
morphisms $L^*\rarrow M^*[n]$.
 The definition of the graded abelian group of morphisms
$\Hom_{B^\rop{}^*}^*(R^*,N^*)$ between two graded right $B^*$\+modules
$R^*$ and $N^*$ is similar.
 The opposite graded ring $B^\rop{}^*$ to $B^*$ is mentioned in
this notation.

 Given a graded right $B^*$\+module $M^*$ and a graded left
$B^*$\+module $N^*$, the tensor product of graded modules
$N^*\ot_{B^*}M^*$ is endowed with the induced grading in the usual way.
 So $N^*\ot_{B^*}M^*$ is also a graded abelian group.

 Now let $A^*$, $B^*$, and $C^*$ be three graded rings.
 Then, for any graded $A^*$\+$B^*$\+bi\-mod\-ule $N^*$ and any graded
$B^*$\+$C^*$\+bimodule $M^*$, the graded abelian group $N^*\ot_{B^*}M^*$
has a natural graded $A^*$\+$C^*$\+bimodule structure.
 For any graded $B^*$\+$A^*$\+bimodule $L^*$ and any graded
$B^*$\+$C^*$\+bimodule $M^*$, the graded abelian group
$\Hom_{B^*}^*(L^*,M^*)$ has a natural graded $A^*$\+$C^*$\+bimodule
structure.
 For any graded $A^*$\+$B^*$\+bimodule $R^*$ and any graded
$C^*$\+$B^*$\+bimodule $N^*$, the graded abelian group
$\Hom_{B^\rop{}^*}^*(R^*,N^*)$ has a natural graded
$C^*$\+$A^*$\+bimodule
structure.
 We refer to~\cite[Section~6.1]{Prel} for further details including
the sign rules.

\subsection{CDG-modules} \label{prelim-cdg-modules-subsecn}
 An \emph{odd derivation of degree\/~$1$} on a graded ring $B^*$
is a homogeneous additive map $d\:B^*\rarrow B^*$ with the components
$d_i\:B^i\rarrow B^{i+1}$, \ $i\in\boZ$, satisfying the \emph{Leibniz
rule with signs} $d(bc)=d(b)c+(-1)^{|b|}bd(c)$ for all homogeneous
elements $b$ and $c$ of degrees~$|b|$ and~$|c|$ in~$B^*$.
 An \emph{odd derivation of degree\/~$1$} on a graded left
$B^*$\+module $M^*$ \emph{compatible with} a given odd derivation
of degree~$1$ on $B^*$ is a homogeneous additive map $d_M\:M^*
\rarrow M^*$ with the components $d_{M,i}\:M^i\rarrow M^{i+1}$
satisfying the equation $d_M(bx)=d(b)x+(-1)^{|b|}bd_M(x)$ for all
homogeneous elements $b\in B^*$ and $x\in M^*$ of degrees~$|b|$
and $|x|\in\boZ$.
 Similarly, an \emph{odd derivation of degree\/~$1$} on a graded right
$B^*$\+module $N^*$ \emph{compatible with} a given odd derivation
of degree~$1$ on $B^*$ is a homogeneous additive map $d_N\:N^*
\rarrow N^*$ with the components $d_{N,i}\:N^i\rarrow N^{i+1}$
satisfying the equation $d_N(yb)=d_N(y)b+(-1)^{|y|}yd(b)$ for all
homogeneous elements $y\in N^*$ and $b\in B^*$ of degrees~$|y|$
and~$|b|\in\boZ$.
 
 A \emph{curved DG\+ring} (\emph{CDG\+ring}) $B^\cu=(B^*,d,h)$ is
a triple consisting of
\begin{itemize}
\item a graded ring $B^*=\bigoplus_{i\in\boZ}B^i$,
\item an odd derivation $d\:B^*\rarrow B^*$ of degree~$1$, and
\item an element $h\in B^2$,
\end{itemize}
satisfying two equations
\begin{enumerate}
\renewcommand{\theenumi}{\roman{enumi}}
\item $d^2(b)=hb-bh$ for all $b\in B^*$ and
\item $d(h)=0$.
\end{enumerate}

 The element $h\in B^2$ is called the \emph{curvature element}
of a CDG\+ring $B^\cu=(B^*,d,h)$.
 We refer to~\cite[Section~3.1]{Pkoszul}, \cite[Section~3.2]{Prel},
or~\cite[Section~2.4]{Pedg} for the (nontrivial) definition of
\emph{morphisms CDG\+rings} (which we will not use in this paper). 

 A \emph{left CDG\+module} $M^\cu=(M^*,d_M)$ over a CDG\+ring $B^\cu$
is a pair consisting of
\begin{itemize}
\item a graded left $B^*$\+module $M^*=\bigoplus_{i\in\boZ}M^i$ and
\item an odd derivation $d_M\:M^*\rarrow M^*$ of degree~$1$
compatible with the odd derivation~$d$ on~$B^*$,
\end{itemize}
satisfying the equation
\begin{enumerate}
\renewcommand{\theenumi}{\roman{enumi}}
\setcounter{enumi}{2}
\item $d_M^2(x)=hx$ for all $x\in M^*$.
\end{enumerate}

 A \emph{right CDG\+module} $N^\cu=(N^*,d_M)$ over a CDG\+ring $B^\cu$
is a pair consisting of
\begin{itemize}
\item a graded right $B^*$\+module $N^*=\bigoplus_{i\in\boZ}N^i$ and
\item an odd derivation $d_N\:N^*\rarrow N^*$ of degree~$1$
compatible with the odd derivation~$d$ on~$B^*$,
\end{itemize}
satisfying the equation
\begin{enumerate}
\renewcommand{\theenumi}{\roman{enumi}}
\setcounter{enumi}{3}
\item $d_N^2(y)=-yh$ for all $y\in N^*$.
\end{enumerate}

 The more familiar concept of a \emph{DG\+ring} is a particular case
of that of a CDG\+ring: a DG\+ring $A^\bu=(A^*,d)$ is the same thing
as a CDG\+ring $(A^*,d,h)$ with a vanishing curvature element $h=0$.
 A (left or right) \emph{DG\+module} over a DG\+ring $(A^*,d)$ is
the same thing as a CDG\+module over the CDG\+ring $(A^*,d,0)$.

 Notice that the CDG\+ring $B^\cu$ carries \emph{no} natural structure
of a left or right CDG\+module over itself, because of a mismatch
of the equations for the square of the differential, (i), (iii),
and~(iv) above.
 However, any CDG\+ring $B^\cu$ is naturally a \emph{CDG\+bimodule}
over itself, in the sense of the definition below.

\subsection{CDG-bimodules} \label{prelim-cdg-bimodules-subsecn}
 Let $A^*$ and $B^*$ be two graded rings, each of them endowed with
an odd derivation $d_A\:A^*\rarrow A^*$ and $d_B\:B^*\rarrow B^*$
of degree~$1$.
 Let $K^*$ be a graded $A^*$\+$B^*$\+bimodule.
 By an \emph{odd derivation of degree\/~$1$} on $K^*$ \emph{compatible
with} the odd derivations~$d_A$ on $A^*$ and~$d_B$ on $B^*$ we simply
mean a homogeneous additive map $d_K\:K^*\rarrow K^*$ of degree~$1$ that
is \emph{simultaneously} an odd derivation on the left $A^*$\+module
$K^*$ compatible with the odd derivation~$d_A$ on $A^*$ \emph{and}
an odd derivation on the right $B^*$\+module $K^*$ compatible with
the odd derivation~$d_B$ on~$B^*$.

 Let $A^\cu=(A^*,d_A,h_A)$ and $B^\cu=(B^*,d_B,h_B)$ be two CDG\+rings.
 A \emph{CDG\+bimodule} $K^\cu=(K^*,d_K)$ over $A^\cu$ and $B^\cu$ is
a pair consisting of
\begin{itemize}
\item a graded $A^*$\+$B^*$\+bimodule $K^*$ and
\item an odd derivation $d_K\:K^*\rarrow K^*$ of degree~$1$
compatible with the odd derivations~$d_A$ on $A^*$ and~$d_B$ on~$B^*$,
\end{itemize}  
satisfying the equation
\begin{enumerate}
\renewcommand{\theenumi}{\roman{enumi}}
\setcounter{enumi}{4}
\item $d_K^2(z)=h_Az-zh_B$ for all $z\in K^*$.
\end{enumerate}

 Let $A^\cu=(A^*,d_A,h_A)$, \ $B^\cu=(B^*,d_B,h_B)$, and
$C^\cu=(C^*,d_C,h_C)$ be three CDG\+rings.
 Let $N^\cu=(N^*,d_N)$ be a CDG\+bimodule over $A^\cu$ and $B^\cu$, and
let $M^\cu=(M^*,d_M)$ be a CDG\+bimodule over $B^\cu$ and~$C^\cu$.
 Then the graded $A^*$\+$C^*$\+bimodule $N^*\ot_{B^*}M^*$ has
a natural structure of CDG\+bimodule over $A^\cu$ and $C^\cu$,
with the differential on $N^*\ot_{B^*}M^*$ given by the usual formula
$d(y\ot x)=d_N(y)\ot x+(-1)^{|y|}y\ot d_M(x)$ for all $y\in N^{|y|}$
and $x\in M^{|x|}$.
 We will denote the CDG\+bimodule $(N^*\ot_{B^*}M^*,\>d)$ over
$A^\cu$ and $C^\cu$ by $N^\cu\ot_{B^*}M^\cu$.

 Let $L^\cu=(L^*,d_L)$ be a CDG\+bimodule over $B^\cu$ and $A^\cu$, and
let $M^\cu=(M^*,d_M)$ be a CDG\+bimodule over $B^\cu$ and~$C^\cu$.
 Then the graded $A^*$\+$C^*$\+bimodule $\Hom_{B^*}^*(L^*,M^*)$ has
a natural structure of CDG\+bimodule over $A^\cu$ and $C^\cu$,
with the differential on $\Hom_{B^*}^*(L^*,M^*)$ given by the usual
formula $d(f)(z)=d_M(f(z))-(-1)^{|f|}f(d_L(z))$ for all
$f\in\Hom_{B^*}^{|f|}(L^*,M^*)$ and $z\in L^{|z|}$.
 We will denote the CDG\+bimodule $(\Hom_{B^*}^*(L^*,M^*),\>d)$
over $A^\cu$ and $C^\cu$ by $\Hom_{B^*}^\cu(L^\cu,M^\cu)$.

 Let $R^\cu=(R^*,d_R)$ be a CDG\+bimodule over $A^\cu$ and $B^\cu$, and
let $N^\cu=(N^*,d_N)$ be a CDG\+bimodule over $C^\cu$ and~$B^\cu$.
 Then the graded $C^*$\+$A^*$\+bimodule
$\Hom_{B^\rop{}^*}^*(R^*,N^*)$ has a natural structure of CDG\+bimodule
over $C^\cu$ and $A^\cu$, with the differential on
$\Hom_{B^\rop{}^*}^*(R^*,N^*)$ given by the usual formula 
$d(g)(w)=d_N(g(w))-(-1)^{|g|}g(d_R(w))$ for all
$g\in\Hom_{B^\rop{}^*}^{|g|}(R^*,N^*)$ and $w\in R^{|w|}$.
 We will denote the CDG\+bimodule $(\Hom_{B^\rop{}^*}^*(R^*,N^*),\>d)$
over $C^\cu$ and $A^\cu$ by $\Hom_{B^\rop{}^*}^\cu(R^\cu,N^\cu)$.

\subsection{Three additive categories associated with
a DG-category} \label{prelim-three-additive-categories-subsecn}
 We refer to~\cite[Section~1.2]{Pkoszul}, \cite[Section~1]{Pedg},
or~\cite[Section~1]{PS5} for general introductory discussions
of DG\+categories in the context suitable for our purposes.
 In this section, we only recall the constructions of three
additive categories $\bA^0$, $\sZ^0(\bA)$, and $\sH^0(\bA)$
associated with a DG\+category~$\bA$.

 Let $\bA$ be a DG\+category.
 Then the three additive categories $\bA^0$, $\sZ^0(\bA)$,
and $\sH^0(\bA)$ are constructed as follows.
 The classes of objects of $\bA^0$, $\sZ^0(\bA)$, and $\sH^0(\bA)$
coincide with the class of objects of the DG\+category~$\bA$.
 The abelian groups of morphisms for any given pair of objects
$X$ and $Y$ are:
\begin{itemize}
\item in the additive category $\bA^0$, the group of morphisms is
the group of degree~$0$ \emph{cochains} in the complex of morphisms
in~$\bA$,
$$
 \Hom_{\bA^0}(X,Y)=\Hom_\bA^0(X,Y);
$$
\item in the additive category $\sZ^0(\bA)$, the group of morphisms
is the group of degree~$0$ \emph{cocycles} in the complex of
morphisms in~$\bA$,
$$
 \Hom_{\sZ^0(\bA)}(X,Y)=\sZ^0\Hom_\bA^\bu(X,Y);
$$
\item in the additive category $\sH^0(\bA)$, the group of morphisms
is the group of degree~$0$ \emph{cohomology} of the complex of
morphisms in~$\bA$,
$$
 \Hom_{\sH^0(\bA)}(X,Y)=\sH^0\Hom_\bA^\bu(X,Y).
$$
\end{itemize}
 The compositions of morphisms in $\bA^0$, $\sZ^0(\bA)$, and
$\sH^0(\bA)$ are induced by the composition of morphisms in~$\bA$.

 The additive category $\bA^0$ is usually \emph{not} well-behaved,
for the reason that it does not contain enough objects.
 Typically, $\bA^0$ is a full subcategory of an additive/abelian
category one is really interested in.

 The additive category $\sZ^0(\bA)$ is well-behaved and often abelian.

 The additive category $\sH^0(\bA)$ is called the \emph{homotopy
category} of a DG\+category~$\bA$.
 For any DG\+category $\bA$ with a zero object, shifts, and cones,
the additive category $\sH^0(\bA)$ has a natural triangulated category
structure.

 A morphism in $\bA^0$ is said to be \emph{closed} if it belongs
to $\sZ^0(\bA)$.
 Morphisms in $\sZ^0(\bA)$ that are equal in $\sH^0(\bA)$ are said
to be \emph{homotopic} to each other.
 Objects of $\sZ^0(\bA)$ that vanish in $\sH^0(\bA)$ are said to be
\emph{contractible}.

 In this paper, we are interested in the case when $\bA=B^\cu\bModl$
or $\bModr B^\cu$ is the DG\+category of CDG\+modules.
 We will continue the discussion in the next
Section~\ref{prelim-dg-categories-of-cdg-modules-subsecn}.

\subsection{DG-categories of CDG-modules}
\label{prelim-dg-categories-of-cdg-modules-subsecn}
 Let $B^\cu=(B^*,d_B,h_B)$ be a CDG\+ring.
 The ring of integers $\boZ$ can be viewed as a CDG\+ring
$(\boZ,0,0)$ with the trivial grading, zero differential,
and zero curvature. 
 Clearly, left CDG\+modules over $B^\cu$ are the same things as
CDG\+bimodules over $B^\cu$ and $(\boZ,0,0)$, while right
CDG\+modules over $B^\cu$ can be viewed as CDG\+bimodules over
$(\boZ,0,0)$ and~$B^\cu$.

 Thus, for any left CDG\+modules $L^\cu$ and $M^\cu$ over $B^\cu$,
the graded abelian group $\Hom_{B^*}^*(L^*,M^*)$ endowed with
the natural differential given by the formula from
Section~\ref{prelim-cdg-bimodules-subsecn} becomes a complex of
abelian groups $\Hom_{B^*}^\bu(L^\cu,M^\cu)$.
 Together with the obvious composition structure on the graded
abelian groups of morphisms of graded modules, this defines
the \emph{DG\+category of left CDG\+modules over~$B^\cu$},
which we denote by $B^\cu\bModl$.
 We refer to~\cite[Section~6.1]{Prel} and~\cite[Section~2.2]{Pedg}
for further details.

 Similarly, for any right CDG\+modules $R^\cu$ and $N^\cu$
over $B^\cu$, the graded abelian group $\Hom_{B^\rop{}^*}(R^\cu,N^\cu)$
endowed with the natural differential given by the formula from
Section~\ref{prelim-cdg-bimodules-subsecn} becomes a complex of
abelian groups $\Hom_{B^\rop{}^*}^\bu(R^\cu,N^\cu)$.
 Together with the obvious composition structure, this defines
the \emph{DG\+category of right CDG\+modules over~$B^\cu$},
which we denote by $\bModr B^\cu$.

 One also observes that, for any right CDG\+module $N^\cu$ and
left CDG\+module $M^\cu$ over $B^\cu$, the graded abelian group
$N^*\ot_{B^*}M^*$ endowed with the natural differential given by
the formula from Section~\ref{prelim-cdg-bimodules-subsecn} becomes
a complex of abelian groups $N^\cu\ot_{B^*}M^\cu$.
 Thus we obtain a DG\+functor of two arguments
$$
 \ot_{B^*}\:\bModr B^\cu\times B^\cu\bModl\lrarrow\bC(\Ab),
$$
where $\bC(\Ab)=\boZ\bModl$ denotes the DG\+category of complexes
of abelian groups.

 The constructions of
Section~\ref{prelim-three-additive-categories-subsecn}
assign three additive categories to the DG\+category $B^\cu\bModl$.
 The additive category $(B^\cu\bModl)^0$ is naturally a full subcategory
of the abelian category $B^*\sModl$ of graded $B^*$\+modules.
 The objects of $(B^\cu\bModl)^0$ are precisely those graded left
$B^*$\+modules that admit \emph{at least one} structure of CDG\+module
over~$B^\cu$ (see~\cite[Examples~3.2 and~3.3]{Pedg} for
counterexamples). 

 The additive category $\sZ^0(B^\cu\bModl)$ is abelian; it is called
the \emph{abelian category of left CDG\+modules over~$B^\cu$}.
 The additive category $\sH^0(B^\cu\bModl)$ is triangulated; it is
called the \emph{homotopy category of left CDG\+modules over~$B^\cu$}.
 The DG\+category $B^\cu\bModl$ itself is an \emph{abelian
DG\+category} in the sense of~\cite[Section~4.6]{Pedg}
and~\cite[Section~3.2]{PS5}.

 We refer to~\cite[Sections~2.2 and~3.1, and Examples~3.17
and~4.41]{Pedg} or~\cite[Section~2.6 and Example~3.14]{PS5} for
a further discussion.

\subsection{\texorpdfstring{The graded ring $B^*[\delta]$ and the functor~$G^+$}{The graded ring B*[delta] and the functor G+}}
\label{prelim-delta-and-G-subsecn}
 Let $B^\cu=(B^*,d,h)$ be a CDG\+ring.
 The graded ring $B^*[\delta]$ is obtained by adjoining to
the graded ring $B^*$ a new element~$\delta$ of degree~$1$ subject
to the relations
$$
 \delta b-(-1)^{|b|}b\delta=d(b)
 \qquad\text{for all $b\in B^{|b|}$}
$$
and
$$
 \delta^2=h.
$$

 A left CDG\+module over $B^\cu$ is the same thing as a graded
left $B^*[\delta]$\+module; so, in fact, the abelian category
of CDG\+modules $\sZ^0(B^\cu\bModl)$ is naturally equivalent to
the abelian category of graded modules $B^*[\delta]\sModl$.
 Here the action of~$\delta$ on a left CDG\+module $M^\cu=(M^*,d_M)$
over $B^\cu$ is defined by the obvious rule
$$
 \delta x=d_M(x)
 \qquad\text{for all $x\in M$.}
$$

 Similarly, a right CDG\+module over $B^\cu$ is the same thing as
a graded right $B^*[\delta]$\+module; so the abelian category
of CDG\+modules $\sZ^0(\bModr B^\cu)$ is naturally equivalent to
the abelian category of graded modules $\sModr B^*[\delta]$.
 Here the action of~$\delta$ on a right CDG\+module $N^\cu=(N^*,d_N)$
over $B^\cu$ is defined by the rule
$$
 y\delta=-(-1)^{|y|}d_N(y)
 \qquad\text{for all $y\in N^{|y|}$.}
$$

 Given a CDG\+ring $B^\cu=(B^*,d,h)$ and a (left or right) CDG\+module
$M^\cu=(M^*,d_M)$ over $B^\cu$, we denote by $B^\cu{}^\#=B^*$
the underlying graded ring of $B^\cu$ and by $M^\cu{}^\#=M^*$
the underlying graded $B^*$\+module of~$M^\cu$.

 The forgetful functor $\#\:\sZ^0(B^\cu\bModl)\rarrow B^*\sModl$
assigning to a CDG\+module $M^\cu=(M^*,d_M)$ over $B^\cu$ its
underlying graded module $M^*=M^\cu{}^\#$ can be interpreted
in terms of the graded ring $B^*[\delta]$ as the functor
$B^*[\delta]\sModl\rarrow B^*\sModl$ of restriction of scalars
with respect to the identity inclusion of graded rings
$B^*\rarrow B^*[\delta]$.
 Accordingly, the left and right adjoint functors to
the forgetful functor $\sZ^0(B^\cu\bModl)\rarrow B^*\sModl$ can be
computed as the functors
$$
 M^*\longmapsto G^+(M^*)=B^*[\delta]\ot_{B^*}M^*
$$
and
$$
 M^*\longmapsto G^-(M^*)=\Hom_{B^*}^*(B^*[\delta],M^*).
$$

 For any CDG\+ring $B^\cu=(B^*,d,h)$ and any graded left $B^*$\+module
$M^*$, there are natural short exact sequences of graded left
$B^*$\+modules
\begin{equation} \label{G-plus-short-exact-sequence}
 0\lrarrow M^*\lrarrow G^+(M^*)^\#\lrarrow M^*[-1]\lrarrow0
\end{equation}
and
\begin{equation} \label{G-minus-short-exact-sequence}
 0\lrarrow M^*[1]\lrarrow G^-(M^*)^\#\lrarrow M^*\lrarrow0,
\end{equation}
which are transformed into each other by the shift functors~$[1]$
and~$[-1]$.
 We refer to~\cite[Section~4.2]{Prel}
and~\cite[Proposition~3.1]{Pedg} for a further discussion.

\subsection{\texorpdfstring{$\Ext^1$ and homotopy Hom lemma}{Ext and homotopy Hom lemma}}
 The following lemma goes back, at least, to~\cite[Lemma~2.1]{Gil}.
 For a version applicable to exact DG\+categories,
see~\cite[Lemma~9.41]{Pedg}, and for abelian DG\+categories,
\cite[Lemma~6.1]{PS5}.

\begin{lem} \label{Ext-1-homotopy-hom-lemma}
 Let $B^\cu=(B^*,d,h)$ be a CDG\+ring, and let $M^\cu=(M^*,d_M)$ and
$L^\cu=(L^*,d_L)$ be left CDG\+modules over~$B^\cu$.
 Then the kernel of the abelian group homomorphism
$$
 \Ext^1_{\sZ^0(B^\cu\bModl)}(L^\cu,M^\cu)\lrarrow
 \Ext^1_{B^*\sModl}(L^*,M^*)
$$
induced by the exact forgetful functor\/ $\sZ^0(B^\cu\bModl)\rarrow
B^*\sModl$ between the abelian categories of CDG\+modules over
$B^\cu$ and graded modules over $B^*$ is naturally isomorphic to
the abelian group
$$
 \Hom_{\sH^0(B^\cu\bModl)}(L^\cu,M^\cu[1])
$$
of morphisms $L^\cu\rarrow M^\cu[1]$ in the triangulated homotopy
category of CDG\+modules\/ $\sH^0(B^\cu\bModl)$.
 In particular, if\/ $\Ext^1_{B^*\sModl}(L^*,M^*)=0$, then
$$
 \Ext^1_{\sZ^0(B^\cu\bModl)}(L^\cu,M^\cu)\,\simeq\,
 \Hom_{\sH^0(B^\cu\bModl)}(L^\cu,M^\cu[1]).
$$
\end{lem}

\begin{proof}
 This is a particular case of~\cite[Lemma~9.41]{Pedg}
or~\cite[Lemma~6.1]{PS5}.
 The isomorphism of abelian groups assigns to (the cochain homotopy
class of) any closed morphism of CDG\+modules $f\:L^\cu\rarrow M^\cu[1]$
(the extension class of) the natural short exact sequence $0\rarrow
M^\cu\rarrow \cone(f)[-1]\rarrow L^\cu\rarrow0$ in the abelian
category of CDG\+modules $\sZ^0(B^\cu\bModl)$.
 The latter short exact sequence splits in the category of graded
modules $B^*\sModl$; and any graded-split short exact sequence of
CDG\+modules arises in this way.
\end{proof}

\subsection{Projectivity, injectivity, and flatness}
\label{prelim-proj-inj-flatness-subsecn}
 Let $B^\cu=(B^*,d,h)$ be a CDG\+ring.

 A left CDG\+module $P^\cu=(P^*,d_P)$ over $B^\cu$ is said to be
\emph{graded-projective} if the graded left $B^*$\+module $P^*$
is projective.
 Graded-projective CDG\+modules form a full DG\+subcategory
$B^\cu\bModl_\bproj$ closed under coproducts, shifts, twists, and cones
in the DG\+category $B^\cu\bModl$ \,\cite[Lemma~5.4(a)]{Pedg}.
 Accordingly, the notation $\sZ^0(B^\cu\bModl_\bproj)$ stands for
the full subcategory of graded-projective CDG\+modules in
the abelian category of CDG\+modules $\sZ^0(B^\cu\bModl)$;
this means the full subcategory of graded $B^*$\+projective modules
in the abelian category of graded left $B^*[\delta]$\+modules
$B^*[\delta]\sModl=\sZ^0(B^\cu\bModl)$.

 In particular, a left CDG\+module $P^\cu$ over $B^\cu$ is said to be
\emph{graded-free} if $P^*$ is a free graded $B^*$\+module (i.~e.,
the graded $B^*$\+module freely generated by some set of homogeneous
generators of various degrees).
 Graded-free CDG\+modules form a full DG\+subcategory
$B^\cu\bModl_\bfree$ closed under coproducts, shifts, twists, and cones
in the DG\+category $B^\cu\bModl$.
 Accordingly, the notation $\sZ^0(B^\cu\bModl_\bfree)$ stands for
the full subcategory of graded-free CDG\+modules in
the abelian category of CDG\+modules $\sZ^0(B^\cu\bModl)$.

 A left CDG\+module $P^\cu=(P^*,d_P)$ over $B^\cu$ is said to be
\emph{projective} if it is a projective object of the abelian
category of CDG\+modules $\sZ^0(B^\cu\bModl)$.
 The full subcategory of projective CDG\+modules is denoted by
$\sZ^0(B^\cu\bModl)_\proj\subset\sZ^0(B^\cu\bModl)$; this means
the full subcategory of projective graded $B^*[\delta]$\+modules
in the abelian category of graded left $B^*[\delta]$\+modules.
 So $\sZ^0(B^\cu\bModl)_\proj=B^*[\delta]\sModl_\proj$.

 According to~\cite[Lemma~6.3(a)]{PS5}, a CDG\+module $P^\cu$ over
$B^\cu$ is projective if and only if it is graded-projective
\emph{and} contractible.
 It follows that the class of projective CDG\+modules is closed under
coproducts, shifts, and cones in the DG\+category $B^\cu\bModl$.

 A right CDG\+module $J^\cu=(J^*,d_J)$ over $B^\cu$ is said to be
\emph{graded-injective} if the graded right $B^*$\+module $J^*$
is injective.
 Graded-injective CDG\+modules form a full DG\+subcategory
$\bModrinj B^\cu$ closed under products, shifts, twists, and cones
in the DG\+category $\bModr B^\cu$ \,\cite[Lemma~5.4(b)]{Pedg}.
 So the notation $\sZ^0(\bModrinj B^\cu)$ stands for
the full subcategory of graded-injective CDG\+modules in
the abelian category of CDG\+modules $\sZ^0(\bModr B^\cu)$;
this means the full subcategory of graded $B^*$\+injective modules
in the abelian category of graded right $B^*[\delta]$\+modules
$\sModr B^*[\delta]=\sZ^0(\bModr B^\cu)$.

 A right CDG\+module $J^\cu=(J^*,d_J)$ over $B^\cu$ is said to be
\emph{injective} if it is an injective object of the abelian
category of CDG\+modules $\sZ^0(\bModr B^\cu)$.
 The full subcategory of injective CDG\+modules is denoted by
$\sZ^0(\bModr B^\cu)_\inj\subset\sZ^0(\bModr B^\cu)$; this means
the full subcategory of injective graded $B^*[\delta]$\+modules
in the abelian category of graded right $B^*[\delta]$\+modules.
 So $\sZ^0(\bModr B^\cu)_\inj=\sModrinj B^*[\delta]$.

 According to~\cite[Lemma~7.2(a)]{PS5}, a CDG\+module $J^\cu$ over
$B^\cu$ is injective if and only if it is graded-injective
\emph{and} contractible.
 It follows that the class of injective CDG\+modules is closed under
products, shifts, and cones in the DG\+category $\bModr B^\cu$.

 A left CDG\+module $F^\cu=(F^*,d_F)$ over $B^\cu$ is said to be
\emph{graded-flat} if the graded left $B^*$\+module $F^*$ is flat.
 Graded-flat CDG\+modules form a full DG\+subcategory
$B^\cu\bModl_\bflat$ closed under coproducts, shifts, twists, and cones
in the DG\+category $B^\cu\bModl$.
 Accordingly, the notation $\sZ^0(B^\cu\bModl_\bflat)$ stands for
the full subcategory of graded-flat CDG\+modules in
the abelian category of CDG\+modules $\sZ^0(B^\cu\bModl)$;
this means the full subcategory of graded $B^*$\+flat modules
in the abelian category of graded left $B^*[\delta]$\+modules
$B^*[\delta]\sModl=\sZ^0(B^\cu\bModl)$.

 A left CDG\+module $F^\cu=(F^*,d_F)$ over $B^\cu$ is said to be
\emph{flat} if it is flat as a graded left $B^*[\delta]$\+module.
 The full subcategory of flat CDG\+modules is denoted by
$\sZ^0(B^\cu\bModl)_\flat\subset\sZ^0(B^\cu\bModl)$; so
$\sZ^0(B^\cu\bModl)_\flat=B^*[\delta]\sModl_\flat$.
 Notice that the graded left (or right) $B^*$\+module $B^*[\delta]$
is free with two generators~$1$ and~$\delta$
(see~\cite[proof of Theorem~4.7]{Prel} or~\cite[Section~3.1]{Pedg});
in particular, $B^*[\delta]$ is a flat graded left and right
$B^*$\+module.
 Hence the underlying graded $B^*$\+module of any flat graded
$B^*[\delta]$\+module is flat; so any flat CDG\+module is graded-flat.

 According to~\cite[proof of Lemma~4.2]{Pedg} or~\cite[Lemma~3.1]{PS5}
(cf.\ the proof of Lemma~\ref{Ext-1-homotopy-hom-lemma} above),
any cone of a closed morphism in the DG\+category $B^\cu\bModl$ can be
viewed as an extension in the abelian category $\sZ^0(B^\cu\bModl)$.
 Since the class of all flat graded $B^*[\delta]$\+modules is closed
under extensions (as well as direct sums and grading shifts) in
$B^*[\delta]\sModl$, it follows that the class of flat CDG\+modules is
closed under coproducts, shifts, and cones in the DG\+category
$B^\cu\bModl$.

 As particular cases of the notation of
Section~\ref{prelim-three-additive-categories-subsecn} for
the homotopy category of a DG\+category, we denote
by $\sH^0(B^\cu\bModl_\bproj)\subset\sH^0(B^\cu\bModl_\bflat)$
the homotopy categories of graded-projective and graded-flat
left CDG\+modules.
 The full triangulated subcategory of flat CDG\+modules is denoted by
$\sH^0(B^\cu\bModl)_\flat\subset\sH^0(B^\cu\bModl_\bflat)$.
 Similarly, the homotopy category of graded-injective right
CDG\+modules is denoted by $\sH^0(\bModrinj B^\cu)$.

\begin{rem} \label{flat-complexes-remark}
 Let $R$ be an associative ring, viewed as a CDG\+ring concentrated
entirely in cohomological degree~$0$, with zero differential and
zero curvature element.
 Then CDG\+modules over $R$ are the same things as complexes of
$R$\+modules.
 In this case, the flat CDG\+modules over $R$ in the sense of
the definition above are precisely the acyclic complexes of flat
$R$\+modules with flat $R$\+modules of cocycles.

 Indeed, the graded version of the Govorov--Lazard characterization of
flat modules over associative rings tells that the flat graded modules
over $R[\delta]$ are precisely the direct limits of (finitely generated)
projective graded $R[\delta]$\+modules (see~\cite[Theorem~3.2]{OR}
for an even more general characterization of flat modules over rings
with many objects).
 As a particular case of the discussion of projective CDG\+modules
above, the projective graded modules over $R[\delta]$ correspond to
contractible complexes of projective $R$\+modules.
 Thus the flat graded $R[\delta]$\+modules correspond to the direct
limits of contractible complexes of projective $R$\+modules.
 According to~\cite[Theorem~8.6\,(ii)\,$\Leftrightarrow$\,(iii)]{Neem},
the latter are precisely the acyclic complexes of flat $R$\+modules
with flat $R$\+modules of cocycles.

 Alternatively, here is a direct proof of the claim that a complex of
$R$\+modules is acyclic with flat modules of cocycles if and only if
it corresponds to a flat graded $R[\delta]$\+module.
 For the ``if'' implication, it suffices to say that direct limits
of contractible complexes of projective modules are acyclic with
flat modules of cocycles.
 Conversely, for any complex of left $R$\+modules $M^\bu$ and any
complex of right $R$\+modules $N^\bu$, the graded abelian group
$N^\bu\ot_{R[\delta]}M^\bu$ is naturally isomorphic to the cokernel
of the differential on the complex of abelian groups
$N^\bu\ot_RM^\bu$.
 Now, if $F^\bu$ is an acyclic complex of flat left $R$\+modules with
flat $R$\+modules of cocycles, then the complex $N^\bu\ot_RF^\bu$ is
acyclic for any complex of right $R$\+modules $N^\bu$.
 The functor $N^\bu\longmapsto N^\bu\ot_R F^\bu$ is exact, and
the functor assigning to an acyclic complex the (co)kernel of its
differential is exact; hence the functor $N^\bu\longmapsto
N^\bu\ot_{R[\delta]}F^\bu$ is exact.
 
 Quite generally, a similar argument based on the graded version of
the Govorov--Lazard theorem tells that the flat CDG\+modules over
any CDG\+ring $B^\cu=(B^*,d,h)$ are precisely the direct limits of
contractible graded-projective CDG\+modules computed in the abelian
category of CDG\+modules $\sZ^0(B^\cu\bModl)$.
\end{rem}

\subsection{Finitely and countably generated and presented CDG-modules}
\label{prelim-finite-countable-modules-subsecn}
 Let $B^*$ be a graded ring.
 A graded $B^*$\+module is said to be \emph{finitely generated}
if it has a finite generating set (of homogeneous elements in
some degrees), and \emph{countably generated} if it has an at most
countable generating set of elements.
 A graded $B^*$\+module is said to be \emph{finitely presented}
if it is the cokernel of a morphism of finitely generated free
graded modules, and \emph{countably presented} if it is the cokernel 
of a morphism of (at most) countably generated free graded modules.
 The additive category of finitely presented graded right
$B^*$\+modules is denoted by $\smodr B^*$.

\begin{prop} \label{finiteness-countability-for-cdg-modules-prop}
 Let $B^\cu=(B^*,d,h)$ be a CDG\+ring and $M^\cu=(M^*,d_M)$ be
a CDG\+module over~$B^\cu$.  Then \par
\textup{(a)} the graded $B^*[\delta]$\+module $M^*$ is finitely
generated if and only if the graded $B^*$\+module $M^*$ is finitely
generated; \par
\textup{(b)} the graded $B^*[\delta]$\+module $M^*$ is finitely
presented if and only if the graded $B^*$\+module $M^*$ is finitely
presented; \par
\textup{(c)} the graded $B^*[\delta]$\+module $M^*$ is countably
generated if and only if the graded $B^*$\+module $M^*$ is countably
generated; \par
\textup{(d)} the graded $B^*[\delta]$\+module $M^*$ is countably
presented if and only if the graded $B^*$\+module $M^*$ is countably
presented.
\end{prop}

\begin{proof}
 Parts~(a) and~(c), as well as the ``only if'' assertions of
parts~(b) and~(d), follow immediately from the fact that
$B^*[\delta]$ is a finitely generated/presented (left or right)
graded $B^*$\+module.
 Parts~(a) and~(b) can be also obtained as a particular case
of~\cite[Lemmas~9.6 and~9.7]{Pedg}.
 Parts~(b) and~(d) can be obtained as particular cases
of~\cite[Lemma~6.7]{PS5}.
\end{proof}

 A CDG\+module $M^\cu$ over $B^\cu$ is said to be \emph{finitely
generated} if it satisfies the equivalent conditions of
Proposition~\ref{finiteness-countability-for-cdg-modules-prop}(a),
and \emph{finitely presented} if it satisfies the equivalent
conditions of
Proposition~\ref{finiteness-countability-for-cdg-modules-prop}(b).
 A CDG\+module $M^\cu$ over $B^\cu$ is said to be \emph{countably
generated} if it satisfies the equivalent conditions of
Proposition~\ref{finiteness-countability-for-cdg-modules-prop}(c),
and \emph{countably presented} if it satisfies the equivalent
conditions of
Proposition~\ref{finiteness-countability-for-cdg-modules-prop}(d).

 Finitely presented CDG\+modules form a full DG\+subcategory 
$\bmodr B^\cu$ closed under finite direct sums, shifts, twists,
and cones in the DG\+category $\bModr B^\cu$.
 Accordingly, the notation $\sZ^0(\bmodr B^\cu)$ stands for
the full subcategory of finitely presented CDG\+modules in
the abelian category of CDG\+modules $\sZ^0(\bModr B^\cu)$;
so $\sZ^0(\bmodr B^\cu)=\smodr B^*[\delta]$.
 Similarly, the homotopy category of the DG\+category of finitely
presented CDG\+modules is denoted by $\sH^0(\bmodr B^\cu)$.

\begin{cor} \label{graded-projective-generated-presented-cor}
 Let $B^\cu$ be a CDG\+ring.  Then \par
\textup{(a)} a graded-projective CDG\+module over $B^\cu$ is
finitely generated if and only if it is finitely presented; \par
\textup{(b)} a graded-projective CDG\+module over $B^\cu$ is
countably generated if and only if it is countably presented.
\end{cor}

\begin{proof}
 One observes that, for any graded ring $B^*$, a projective
graded $B^*$\+module is finitely generated if and only if it is
finitely presented (for part~(a)), and a projective graded
$B^*$\+module is countably generated if and only if it is
countably presented (for part~(b)).
 Then it remains to recall
Proposition~\ref{finiteness-countability-for-cdg-modules-prop}(a--b)
(for part~(a) of the corollary) or
Proposition~\ref{finiteness-countability-for-cdg-modules-prop}(c--d)
(for part~(b) of the corollary).
\end{proof}

\Section{Preliminaries on Cotorsion Pairs and Model Structures}

 This section contains a reminder of fairly standard material.
 We refer to a number of sources, such as the books~\cite{Hov-book,GT}
and the papers~\cite{ET,BBE,Hov,Bec,Sto-ICRA,Gil2,PR,PS4}, etc.

\subsection{Cotorsion pairs in exact categories}
 Le $\sE$ be an exact category (in Quillen's sense).
 We denote by $\Ext^*_{\sE}({-},{-})$ the Yoneda Ext functor
in the exact category~$\sE$.

 Let $\sL$ and $\sR\subset\sE$ be two classes of objects in~$\sE$.
 Then the notation $\sL^{\perp_1}\subset\sE$ stands for the class
of all objects $X\in\sE$ such that $\Ext^1_\sE(L,X)=0$ for all
$L\in\sL$.
 Dually, ${}^{\perp_1}\sR\subset\sE$ is the class of all objects
$Y\in\sE$ such that $\Ext^1_\sE(Y,R)=0$ for all $R\in\sR$.

 A pair of classes of objects $(\sL,\sR)$ in $\sE$ is said to be
a \emph{cotorsion pair} if $\sR=\sL^{\perp_1}$ and
$\sL={}^{\perp_1}\sR$.
 For any class of objects $\sS\subset\sE$, the pair of classes
$\sR=\sS^{\perp_1}$ and $\sL={}^{\perp_1}\sR$ is a cotorsion pair
in~$\sE$.
 The cotorsion pair (${}^{\perp_1}(\sS^{\perp_1})$, $\sS^{\perp_1}$)
in $\sE$ is said to be \emph{generated} by the class~$\sS$.
 Dually, for any class of objects $\sT\subset\sE$, the pair of classes
$\sL={}^{\perp_1}\sT$ and $\sR=\sL^{\perp_1}$ is a cotorsion pair
in~$\sE$.
 The cotorsion pair (${}^{\perp_1}\sT$, $({}^{\perp_1}\sT)^{\perp_1})$
in $\sE$ is said to be \emph{cogenerated} by the class~$\sT$.

 The intersection $\sL\cap\sR$ is called the \emph{core} (or in
another terminology, the \emph{kernel}) of a cotorsion pair $(\sL,\sR)$
in an exact category~$\sE$.

 A cotorsion pair $(\sL,\sR)$ is $\sE$ is said to be \emph{complete}
if for every object $E\in\sE$ there exist (admissible) short exact
sequences in~$\sE$ of the form
\begin{gather}
 0\lrarrow R'\lrarrow L\lrarrow E\lrarrow0,
 \label{special-precover-sequence} \\
 0\lrarrow E\lrarrow R\lrarrow L'\lrarrow0
 \label{special-preenvelope-sequence}
\end{gather}
with objects $L$, $L'\in\sL$ and $R$, $R'\in\sR$.

 A short exact sequence~\eqref{special-precover-sequence} is called
a \emph{special precover sequence}.
 A short exact sequence~\eqref{special-preenvelope-sequence} is called
a \emph{special preenvelope sequence}.
 The short exact sequences~(\ref{special-precover-sequence}--%
\ref{special-preenvelope-sequence}) are also collectively referred to
as \emph{approximation sequences}.

 A class of objects $\sL\subset\sE$ is said to be \emph{generating} if
for every object $E\in\sE$ there exists an object $L\in\sL$ together
with an admissible epimorphism $L\rarrow E$ in~$\sE$.
 Dually, a class of objects $\sR\subset\sE$ is said to be
\emph{cogenerating} if for every object $E\in\sE$ there exists
an object $R\in\sR$ together with an admissible monomorphism
$E\rarrow R$ in~$\sE$.
 Obviously, in any complete cotorsion pair $(\sL,\sR)$ is $\sE$,
the class $\sL$ is generating and the class $\sR$ is cogenerating.


\begin{lem} \label{hereditary-cotorsion-lemma}
 Let $(\sL,\sR)$ be a cotorsion pair in
an exact category\/~$\sE$.
 Assume that the class\/ $\sL$ is generating and the class\/ $\sR$ is
cogenerating in\/~$\sE$.
 Then the following conditions are equivalent:
\begin{enumerate}
\item the class\/ $\sL$ is closed under kernels of admissible
epimorphisms in\/~$\sE$;
\item the class\/ $\sR$ is closed under cokernels of admissible
monomorphisms in\/~$\sE$;
\item $\Ext^2_\sE(L,R)=0$ for all objects $L\in\sL$ and $R\in\sR$;
\item $\Ext^n_\sE(L,R)=0$ for all objects $L\in\sL$, \ $R\in\sR$,
and all integers $n\ge1$.
\end{enumerate}
\end{lem}

\begin{proof}
 This is~\cite[Lemma~6.17]{Sto-ICRA}.
 The statement there mentions an additional assumption on $\sE$
(the so-called weak idempotent completeness) which is, however,
not necessary in the argument.
\end{proof}

 A cotorsion pair $(\sL,\sR)$ in
an exact
category $\sE$ is said to be \emph{hereditary} if it satisfies
the assumptions and any one of the equivalent conditions~(1--4)
of Lemma~\ref{hereditary-cotorsion-lemma}.

 Let $\sA$ be an exact category and $\sE\subset\sA$ be a full
subcategory closed under extensions.
 Then the \emph{inherited exact category structure} on $\sE$ is
defined by the rule that the (admissible) short exact sequences
in $\sE$ are the short exact sequences in $\sA$ with the terms
belonging to~$\sE$.

 Let $(\sL,\sR)$ be a complete cotorsion pair in
the exact category~$\sA$.
 We will say $(\sL,\sR)$ \emph{restricts to} (\emph{a complete
cotorsion pair in}) the subcategory $\sE\subset\sA$ if the pair
of classes ($\sE\cap\sL$, $\sE\cap\sR$) is a complete cotorsion
pair in the exact category $\sE$ (where the exact structure on $\sE$
is inherited from~$\sA$).

\begin{lem} \label{restricted-exact-structure}
 Let\/ $\sA$ be an exact category, $\sE\subset\sA$ be a full subcategory
closed under extensions, and $(\sL,\sR)$ be a complete cotorsion pair
in\/~$\sA$.
 Assume that the full subcategory\/ $\sE$ is closed under kernels of
admissible epimorphisms in\/ $\sA$, and the class\/ $\sL$ is contained
in\/~$\sE$.
 Then the complete cotorsion pair $(\sL,\sR)$ in\/ $\sA$ restricts to
a complete cotorsion pair $(\sL$, $\sE\cap\sR)$ in\/~$\sE$.
 Furthermore, if the cotorsion pair $(\sL,\sR)$ is hereditary in\/
$\sA$, then the restricted cotorsion pair $(\sL$, $\sE\cap\sR)$
is hereditary in\/~$\sE$.
\end{lem}

\begin{proof}
 This is~\cite[Lemma~2.2]{Pgen}.
\end{proof}

\subsection{Filtrations and the Eklof--Trlifaj theorem}
 We continue the discussion in the context of exact categories before
passing to the special case of Grothendieck abelian categories with
enough projective objects.

 Let $\sC$ be a category and $\alpha$~be an ordinal.
 A direct system $(C_i\to C_j)_{0\le i<j<\alpha}$ of objects in $\sC$
indexed by the well-ordered set~$\alpha$ is said to be a \emph{smooth
chain} if $C_j=\varinjlim_{i<j}C_i$ for all limit ordinals $j<\alpha$.
 Assuming that the direct limit $\varinjlim_{i<\alpha}C_i$ exists,
we will denote it by $C_\alpha=\varinjlim_{i<\alpha}C_i$.

 Let $\sE$ be an exact category, and let
$(F_i\to F_j)_{0\le i<j\le\alpha}$ be a smooth chain of objects in
$\sE$ with a direct limit $F_\alpha=\varinjlim_{i<\alpha}F_i$.
 One says that the smooth chain $(F_i)_{0\le i\le\alpha}$ is
a \emph{filtration} of the object $F=F_\alpha$ in $\sE$
if the following two conditions hold:
\begin{itemize}
\item $F_0=0$;
\item for every ordinal $0\le i<i+1<\alpha$, the transition morphism
$F_i\rarrow F_{i+1}$ is an admissible monomorphism in~$\sE$.
\end{itemize}

 An object $F\in\sE$ endowed with an ordinal-indexed filtration
$(F_i)_{0\le i\le\alpha}$ is said to be \emph{filtered by}
the cokernels of the admissible monomorphisms $F_i\rarrow F_{i+1}$,
\ $0\le i<i+1<\alpha$.
 In an alternative terminology, the object $F$ is said to be
a \emph{transfinitely iterated extension} of the objects
$(F_{i+1}/F_i)_{0\le i<i+1<\alpha}$ (\emph{in the sense of
the direct limit}).

 Given a class of objects $\sS\subset\sE$, the class of all objects
in $\sE$ filtered by (objects isomorphic to) the objects from $\sS$
is denoted by $\Fil(\sS)$.
 A class of objects $\sF\subset\sE$ is said to be \emph{deconstructible}
if there exists a \emph{set} of objects $\sS\subset\sE$ such that
$\sF=\Fil(\sS)$.

 The following result is known classically as
the \emph{Eklof lemma}~\cite[Lemma~1]{ET}, \cite[Lemma~6.2]{GT}.

\begin{lem} \label{eklof-lemma}
 For any class of objects\/ $\sR$ in an exact category\/ $\sE$,
the left\/ $\Ext^1$\+or\-thog\-o\-nal class\/ ${}^{\perp_1}\sR$ is
closed under transfinitely iterated extensions (in the sense of
the direct limit) in\/~$\sE$.
 In other words, $\Fil({}^{\perp_1}\sR)={^{\perp_1}\sR}$.
\end{lem}

\begin{proof}
 This is~\cite[Lemma~7.5]{PS6}.
 The argument from~\cite[Lemma~4.5]{PR} is applicable.
\end{proof}

 Given a class of objects $\sF\subset\sE$, we denote by $\sF^\oplus
\subset\sE$ the class of all direct summands of objects from~$\sF$.
 The following theorem is due to Eklof and Trlifaj~\cite[Theorems~2
and~10]{ET}, \cite[Theorem~6.11 and Corollary~6.14]{GT}.
 We formulate it for certain abelian categories only.

\begin{thm} \label{eklof-trlifaj-theorem}
 Let\/ $\sA$ be a Grothendieck category with a projective generator
and $(\sL,\sR)$ be the cotorsion pair generated by a \emph{set} of
objects\/ $\sS\subset\sA$.  Then \par
\textup{(a)} the cotorsion pair $(\sL,\sR)$ is complete; \par
\textup{(b)} if\/ $\sS$ contains a projective generator of\/ $\sA$,
then\/ $\sL=\Fil(\sS)^\oplus$.
\end{thm}

\begin{proof}
 This result, properly stated, holds in any locally presentable abelian
category~\cite[Corollary~3.6 and Theorem~4.8]{PR},
\cite[Theorems~3.3 and~3.4]{PS4}, as well as in any efficient
exact category~\cite[Theorem~5.16]{Sto-ICRA}.
 The proofs are based on the small object
argument~\cite[Theorem~2.1.14]{Hov-book}.
\end{proof}

\begin{ex} \label{flat-cotorsion-pair-example}
 The following example of a hereditary complete cotorsion pair
is thematic.
 Let $R$ be an associative ring.
 A left $R$\+module $C$ is said to be \emph{cotorsion} if
$\Ext^1_R(F,C)=0$ for all flat left $R$\+modules~$F$.
 The pair of classes (flat left $R$\+modules, cotorsion left
$R$\+modules) is a cotorsion pair in the abelian category of left
$R$\+modules $\sA=R\rModl$; it is known as
the \emph{flat cotorsion pair}.
 The class of all flat left $R$\+modules is
deconstructible~\cite[Lemma~1 and Proposition~2]{BBE},
\cite[Lemma~6.23]{GT}; so Lemma~\ref{eklof-lemma} implies that
the flat cotorsion pair is generated by a set (of flat modules).
 By Theorem~\ref{eklof-trlifaj-theorem}, it follows that the flat
cotorsion pair is complete~\cite{BBE}.
 The class of flat $R$\+modules is also obviously closed under
kernels of epimorphisms; so the flat cotorsion pair in $R\rModl$
is hereditary by Lemma~\ref{hereditary-cotorsion-lemma}(1).
\end{ex}

\subsection{Pure-injective graded modules}
\label{pure-injective-subsecn}
 In this section we specialize the discussion of cotorsion pairs
even further and consider the abelian category $\sA=B^*\sModl$ of
graded modules over a graded ring~$B^*$.

 A short exact sequence of graded left $B^*$\+modules $0\rarrow K^*
\rarrow L^*\rarrow M^*\rarrow0$ is said to be \emph{pure} if
the functor of tensor product over $B^*$ with any graded right
$B^*$\+module takes it to a short exact sequence of graded abelian
groups.
 If this is the case, the morphism of graded $B^*$\+modules
$K^*\rarrow L^*$ is said to be a \emph{pure monomorphism} and
the morphism $L^*\rarrow M^*$ is said to be a \emph{pure epimorphism}.
 The graded module $K^*$ is also said to be a \emph{pure homogeneous
$B^*$\+submodule} of $L^*$, and the graded module $M^*$ is said to be
a \emph{pure homogeneous quotient} of~$L^*$.

 A graded left $B^*$\+module $J^*$ is said to be \emph{pure-injective}
if the graded Hom functor $\Hom^*_{B^*}({-},J^*)$ takes any pure short
exact sequence of graded left $B^*$\+modules to an exact sequence of
graded abelian groups.
 Equivalently, this means that graded $B^*$\+module morphisms into
$J^*$ can be extended from any pure homogeneous $B^*$\+submodule $K^*$
to the ambient graded $B^*$\+module~$L^*$.

 For any graded left $B^*$\+module $M^*$, any graded right
$B^*$\+module $N^*$, and any integer $n\ge0$, there is a natural
isomorphism of abelian groups
\begin{equation} \label{character-module-Ext-Tor-homological-formula}
 \Ext^n_{B^*\sModl}(M^*,\Hom_\boZ^*(N^*,\boQ/\boZ))\,\simeq\,
 \Hom_\boZ(\Tor_n^{B^*}(N^*,M^*)^0,\>\boQ/\boZ).
\end{equation}
 Here $\Tor_n^{B^*}(N^*,M^*)^0$ denotes the degree~$0$ component of
the graded abelian group $\Tor_n^{B^*}(N^*,M^*)$ (in the grading
induced by the gradings of $B^*$, $N^*$, and~$M^*$).

 In particular,
the isomorphism~\eqref{character-module-Ext-Tor-homological-formula}
for $n=0$, together with exactness of the functor
$\Hom_\boZ({-},\boQ/\boZ)$, imply the fact that the graded left
$B^*$\+module $\Hom^*_\boZ(N^*,\boQ/\boZ)$ is pure-injective for
all graded right $B^*$\+modules~$N^*$.
 Conversely, all pure-in\-jec\-tive graded left $B^*$\+modules are
direct summands of the character modules $\Hom^*_\boZ(N^*,\boQ/\boZ)$
(see, e.~g., \cite[Corollary~2.21(b) or Theorem~2.27(d)]{GT}).

\begin{thm} \label{cogenerated-by-pure-injectives}
 Let $B^*$ be a graded ring.
 For any class\/ $\sT$ of pure-injective graded left $B^*$\+modules,
the cotorsion pair $(\sL,\sR)$ cogenerated by\/ $\sT$ in
the abelian category $B^*\sModl$ is complete and generated by a set
(of graded modules).
\end{thm}

\begin{proof}
 This is~\cite[Lemmas~6.17--6.18 and Theorem~6.19]{GT}.
\end{proof}

\begin{ex}
 The cotorsion pair cogenerated by the class of \emph{all}
pure-injective graded modules in $B^*\sModl$ is the (graded version of)
the flat cotorsion pair defined in
Example~\ref{flat-cotorsion-pair-example}.
 So, in this cotorsion pair, the left class $\sL$ is the class of all
flat graded left $B^*$\+modules, while the right class $\sR$ is the class
of all \emph{cotorsion} graded left $B^*$\+modules (which will be
discussed below in Section~\ref{cotorsion-graded-modules-subsecn}).
 This assertion follows from the fact that the pure-injective graded
modules are the direct summands of the character modules and
the isomorphism~\eqref{character-module-Ext-Tor-homological-formula}
for $n=1$.
\end{ex}

\subsection{Weak factorization systems}
 Let $\sC$ be a category, and let $l\:A\rarrow B$ and $r\:X\rarrow Y$
be two morphisms in~$\sC$.
 One says that the morphism~$l$ has \emph{left lifting property}
with respect to~$r$, or equivalently, the morphism~$r$ has \emph{right
lifting property} with respect to~$l$ if any commutative square diagram
as below admits a diagonal filling making both the triangles
commutative:
$$
 \xymatrix{
 A \ar[r] \ar[d]_l &  X \ar[d]^r \\
 B \ar[r] \ar@{..>}[ru] & Y
 }
$$

 Given two classes of morphisms $\cL$ and $\cR$ in $\sC$, one denotes
by $\cL^\square$ the class of all morphisms having right lifting
property with respect to all morphisms from $\cL$, and by
${}^\square\cR$ the class of all morphisms having left lifting property
with respect to all morphisms from~$\cR$.
 A pair of classes of morphisms $(\cL,\cR)$ in $\sC$ is said to be
a \emph{weak factorization system} if the following conditions hold:
\begin{itemize}
\item $\cR=\cL^\square$ and ${}^\square\cR=\cL$;
\item every morphism~$f$ in $\sC$ can be factorized as $f=rl$
with $r\in\cR$ and $l\in\cL$.
\end{itemize}

 A weak factorization system $(\cL,\cR)$ in $\sC$ is said to be
\emph{cofibrantly generated} if there exists a \emph{set} of morphisms
$\cS$ in $\sC$ such that $\cR=\cS^\square$.

 Let $\sL$ and $\sR$ be two classes of objects in an abelian
category~$\sA$.
 A morphism in $\sA$ is said to be an \emph{$\sL$\+monomorphism} if
it is a monomorphism and its cokernel belongs to~$\sL$.
 Dually, a morphism in $\sA$ is said to be an \emph{$\sR$\+epimorphism}
if it is an epimorphism and its cokernel belongs to~$\sR$.

 A weak factorization system $(\cL,\cR)$ in an abelian category $\sA$
is said to be \emph{abelian} if there exists a pair of classes of
objects $(\sL,\sR)$ in $\sA$ such that $\cL$ is the class of all
$\sL$\+monomorphisms and $\cR$ is the class of all $\sR$\+epimorphisms
in~$\sA$.

\begin{thm}
 Let $(\sL,\sR)$ be a pair of classes of objects in an abelian
category~$\sA$.
 Denote by $\cL$ the class of all\/ $\sL$\+monomorphisms and by $\cR$
the class of all\/ $\sR$\+epimorphisms in\/~$\sA$.
 Then $(\cL,\cR)$ is a weak factorization system if and only if
$(\sL,\sR)$ is a complete cotorsion pair in\/~$\sA$.
 So abelian weak factorization systems correspond bijectively to
complete cotorsion pairs in\/~$\sA$.
\end{thm}

\begin{proof}
 This result is essentially due to Hovey~\cite[Theorem~2.2]{Hov}.
 For another proof, see~\cite[Theorem~2.4]{PS4}.
\end{proof}

 An abelian weak factorization system $(\cL,\cR)$ in $\sA$ is said
to be \emph{hereditary} if the related complete cotorsion pair
$(\sL,\sR)$ in $\sA$ is hereditary.

 An abelian weak factorization system $(\cL,\cR)$ in a locally
presentable abelian category $\sA$ is cofibrantly generated if and
only if the cotorsion pair $(\sL,\sR)$ in $\sA$ is generated by
a set of objects~\cite[Lemma~3.7]{PS4}.

\subsection{Abelian model structures}
\label{abelian-model-structures-subsecn}
 Let $\sC$ be a category.
 A \emph{model structure} on $\sC$ is a triple of classes of morphisms
$(\cL,\cW,\cR)$ in $\sC$ satisfying the following conditions:
\begin{itemize}
\item the pair of classes ($\cL\cap\cW$, $\cR$) is a weak factorization
system in~$\sC$;
\item the pair of classes ($\cL$, $\cW\cap\cR$) is a weak factorization
system in~$\sC$;
\item the class of morphisms $\cW$ is closed under retracts (in
the category of morphisms in~$\sC$);
\item the class of morphisms $\cW$ satisfies the $2$-out-of-$3$
property: if $(f,g)$ is a composable pair of morphisms in $\sC$ and
two of the three morphisms $f$, $g$, and $fg$ belong to $\cW$, then
the third morphism also belongs to~$\cW$.
\end{itemize}

 Morphisms from the class $\cL$ are called \emph{cofibrations},
morphisms from the class $\cR$ are called \emph{fibrations}, and
morphisms from the class $\cW$ are called \emph{weak equivalences}
in~$\sC$.
 Furthermore, morphisms from the class $\cL\cap\cW$ are called
\emph{trivial cofibrations}, and morphisms from the class
$\cW\cap\cR$ are called \emph{trivial fibrations}.
It is a standard fact that two of the classes in a model structure determine the third one.

 Given a model structure $(\cL,\cW,\cR)$ on a category $\sC$ with
an initial object $\varnothing\in\sC$, an object $L\in\sC$ is said
to be \emph{cofibrant} if the morphism $\varnothing\rarrow L$ is
a cofibration.
 Dually, if ${*}\in\sC$ is a terminal object, then an object $R\in\sC$
is said to be \emph{fibrant} if the morphism $R\rarrow{*}$ is
a fibration.

 Under the stronger assumption that the category $\sC$ has a zero
object $0\in\sC$ (i.~e., $\varnothing=0={*}$), an object $W\in\sC$
is said to be \emph{weakly trivial} if $0\rarrow W$ is a weak
equivalence, or equivalently, $W\rarrow0$ is a weak equivalence
in~$\sC$.

 A model structure $(\cL,\cW,\cR)$ on a category $\sC$ is said to be
\emph{cofibrantly generated} if both the weak factorization systems
($\cL\cap\cW$, $\cR$) and ($\cL$, $\cW\cap\cR$) are cofibrantly
generated.
 A model structure on an abelian category $\sA$ is said to be
\emph{abelian} if both ($\cL\cap\cW$, $\cR$) and ($\cL$, $\cW\cap\cR$)
are abelian weak factorization systems.

 A class of objects $\sW$ in an abelian category $\sA$ is said to be
\emph{thick} if it is closed under direct summands, extensions,
kernels of epimorphisms, and cokernels of monomorphisms.
 The following theorem is one the main results of the paper~\cite{Hov}.

\begin{thm} \label{abelian-model-structures-thm}
 Let\/ $\sA$ be an abelian category.
 Then abelian model structures on\/ $\sA$ correspond bijectively to
triples of classes of objects $(\sL,\sW,\sR)$ in\/ $\sA$ such that
\begin{itemize}
\item the pair of classes\/ $(\sL\cap\sW,\sR)$ is a complete
cotorsion pair in\/~$\sA$;
\item the pair of classes\/ $(\sL,\sW\cap\sR)$ is a complete
cotorsion pair in\/~$\sA$;
\item the class of objects\/ $\sW$ is thick in\/~$\sA$.
\end{itemize}
  The correspondence assigns to a triple of classes of objects
$(\sL,\sW,\sR)$ the triple of classes of morphisms $(\cL,\cW,\cR)$,
where $\cL$ is the class of all\/ $\sL$\+monomorphisms and
$\cR$ is the class of all\/ $\sR$\+epimorphisms.
 The class $\cL\cap\cW$ consists precisely of all
$(\sL\cap\sW)$\+monomorphisms, and the class $\cW\cap\cR$ consists
precisely of all $(\sW\cap\sR)$\+epimorphisms.
 The class $\cW$ consists of all morphisms~$w$ decomposable as
$w=rl$, where
$r$ is a $\sW$\+epimorphism and $l$ is a $\sW$\+monomorphism.

 Conversely, if $(\cL,\cW,\cR)$ is an abelian model structure on\/
$\sA$, then\/ $\sL$ is recovered as the class of all cofibrant objects
and\/ $\sR$ is the class of all fibrant objects in\/~$\sA$, while\/
$\sW$ is the class of all weakly trivial objects.
\end{thm}

\begin{proof}
 This is a part of~\cite[Theorem~2.2]{Hov} (see~\cite[Theorem~4.2]{PS4}
for a quick summary with specific references).
Our statement differs from the references in the description of weak equivalences in terms of weakly trivial objects, but it is equivalent by \cite[Lemma 5.8]{Hov} (see also~\cite[Lemma 6.14]{Sto-ICRA}).
\end{proof}

 With the result of Theorem~\ref{abelian-model-structures-thm} in mind,
one usually identifies abelian model structures $(\cL,\cW,\cR)$ with
the related triples of classes of objects $(\sL,\sW,\sR)$ and speaks
of the triple $(\sL,\sW,\sR)$ as ``an abelian model structure''.
 We will follow this conventional abuse of terminology.

 Let $(\sL,\sW,\sR)$ be an abelian model structure on an abelian
category~$\sA$.
 Comparing the assertions of
Theorem~\ref{abelian-model-structures-thm}
and Lemma~\ref{hereditary-cotorsion-lemma}, one can easily see
that the cotorsion pair ($\sL\cap\sW$, $\sR$) is hereditary if and
only if the cotorsion pair $(\sL$, $\sW\cap\sR$) is hereditary.
 If this is the case, one says that the abelian model structure
$(\sL,\sW,\sR)$ is hereditary.

 Given two cotorsion pairs $(\widetilde\sL,\sR)$ and
$(\sL,\widetilde\sR)$ in $\sA$, one can immediately see that
the inclusion $\widetilde\sL\subset\sL$ holds
if and only if the inclusion $\widetilde\sR\subset\sR$ holds.
 If this is the case, we say that the cotorsion pairs
$(\widetilde\sL,\sR)$ and $(\sL,\widetilde\sR)$ are \emph{nested}.
 The next theorem is the main result of the paper~\cite{Gil2}.

\begin{thm} \label{gillespie-theorem}
 Let\/ $\sA$ be an abelian category.
 Then there is a bijective correspondence between hereditary abelian
model structures on\/ $\sA$ and nested pairs of hereditary complete
cotorsion pairs with a common core, i.~e., pairs of hereditary
complete cotorsion pairs $(\widetilde\sL,\sR)$ and
$(\sL,\widetilde\sR)$ such that\/
$\widetilde\sL\subset\sL$, \ $\widetilde\sR\subset\sR$, and\/
$\widetilde\sL\cap\sR=\sL\cap\widetilde\sR$.
 To a hereditary abelian model structure $(\sL,\sW,\sR)$ on\/ $\sA$,
the pair of cotorsion pairs $(\widetilde\sL,\sR)$ and
$(\sL,\widetilde\sR)$ with\/ $\widetilde\sL=\sL\cap\sW$ and\/
$\widetilde\sR=\sW\cap\sR$ is assigned.

 Conversely, given two cotorsion pairs $(\widetilde\sL,\sR)$ and
$(\sL,\widetilde\sR)$ satisfying the conditions above, the class of
all weakly trivial objects\/ $\sW$ is recovered as the class of
all objects $W\in\sA$ admitting a short exact sequence\/
$0\lrarrow W\rarrow R\rarrow L'\rarrow0$ with
$R\in\widetilde\sR$ and $L'\in\widetilde\sL$, or equivalently,
admitting a short exact sequence\/ $0\rarrow R'\rarrow L\rarrow W
\rarrow0$ with $L\in\widetilde\sL$ and $R'\in\widetilde\sR$.
\end{thm}

\begin{proof}
 This is~\cite[Main Theorem~1.2]{Gil2}.
\end{proof}

 A \emph{model category} is a category with set-indexed limits and
colimits endowed with a model structure~\cite{Hov-book}.
 An \emph{abelian model category} is a complete, cocomplete abelian
category endowed with an abelian model structure.

 Given a model category $\sC$, the category $\sC[\cW^{-1}]$ obtained
by inverting formally all the weak equivalences in $\sC$ is called
the \emph{homotopy category} of~$\sC$.
 For any hereditary abelian model structure on an abelian category
$\sA$, the homotopy category $\sA[\cW^{-1}]$ is naturally
triangulated~\cite[Corollary~1.1.15 and preceding discussion]{Bec},
\cite[Section~5.7]{PS5}.

 The homotopy categories of many abelian model categories naturally
occurring in the context of complexes, DG\+modules, CDG\+modules etc.\
are various kinds of derived categories.
 So the terminology ``the homotopy category of a model category''
is misleadingly clashing with the terminology ``the homotopy category
of complexes'' or ``the homotopy category of (C)DG\+modules''. 

\subsection{Quillen adjunctions and equivalences}
\label{quillen-adjunctions-equivalences-subsecn}
 Let $\sC$ and $\sD$ be two categories endowed with weak factorization
systems $(\cL_\sC,\cR_\sC)$ and $(\cL_\sD,\cR_\sD)$.
 Let $F\:\sC\rarrow\sD$ and $G\:\sD\rarrow\sC$ be a pair of adjoint
functors, with the functor $G$ right adjoint to~$F$.
 Then one can easily see that the inclusion $F(\cL_\sC)\subset\cL_\sD$
holds if and only if the inclusion $G(\cR_\sD)\subset\cR_\sC$ holds.

 Let $\sC$ and $\sD$ be two categories endowed with model structures,
and let $F\:\sC\rarrow\sD$ and $G\:\sD\rarrow\sC$ be a pair of adjoint
functors as above.
 One says that $(F,G)$ is a \emph{Quillen adjunction} if the following
two conditions hold:
\begin{itemize}
\item the functor $F$ takes cofibrations to cofibrations, or 
equivalently, the functor $G$ takes trivial fibrations to trivial
fibrations;
\item the functor $F$ takes trivial cofibrations to trivial
cofibrations, or equivalently, the functor $G$ takes fibrations to
fibrations.
\end{itemize}
 If this is the case, then $F$ is called a \emph{left Quillen functor}
and $G$ is called a \emph{right Quillen functor}.
 The following result is known as \emph{Ken Brown's lemma}.

\begin{lem} \label{ken-brown-lemma}
 Let\/ $\sC$ and\/ $\sD$ be two categories with finite limits and
colimits, endowed with model structures.
 Then any left Quillen functor $F\:\sC\rarrow\sD$ takes all weak
equivalences between cofibrant objects in\/ $\sC$ to weak
equivalences in\/~$\sD$.
 Dually, any right Quillen functor $G\:\sD\rarrow\sC$ takes all weak
equivalences between fibrant objects in\/ $\sD$ to weak equivalences
in\/~$\sC$.
\end{lem}

\begin{proof}
 This is~\cite[Lemma~1.1.12]{Hov-book}.
\end{proof}

 Let $\sC$ be a category with model structure.
 Inverting formally weak equivalences in the full subcategory of
cofibrant objects (or fibrant objects, or objects that are
simultaneously fibrant and cofibrant) in~$\sC$, one obtains
a category naturally equivalent to $\sC[\cW^{-1}]$
\,\cite[Proposition~1.2.3]{Hov-book}.

 Now let $\sC$ and $\sD$ be two categories with finite limits and
colimits, endowed with model structures $(\cL_\sC,\cW_\sC,\cR_\sC)$
and $(\cL_\sD,\cW_\sD,\cR_\sD)$.
 Let $(F,G)$ be a Quillen adjunction between $\sC$ and~$\sD$.
 Then one constructs the \emph{left derived functor}
$\boL F\:\sC[\cW_\sC^{-1}]\rarrow\sD[\cW_\sD^{-1}]$ by applying
the functor $F$ to cofibrant objects in~$\sC$.
 Dually, the \emph{right derived functor} $\boR G\:\sD[\cW_\sD^{-1}]
\rarrow\sC[\cW_\sC^{-1}]$ is constructed by applying the functor $G$
to fibrant objects in~$\sD$.
 Lemma~\ref{ken-brown-lemma} implies that the derived functors
$\boL F$ and $\boR G$ are well-defined.
 The derived functor $\boL F$ is left adjoint to the derived
functor~$\boR G$ \,\cite[Lemma~1.3.10]{Hov-book}.

\begin{lem} \label{quillen-equivalence-lemma}
 Let $(F,G)$ be a Quillen adjunction between two categories\/ $\sC$
and\/ $\sD$ with finite limits and colimits endowed with model
structures.
 Then the following two conditions are equivalent:
\begin{enumerate}
\item For any cofibrant object $L$ in\/ $\sC$ and any fibrant object
$R$ in\/ $\sD$, a morphism $F(L)\rarrow R$ is a weak equivalence in\/
$\sD$ if and only if the corresponding morphism $L\rarrow G(R)$ is
a weak equivalence in\/~$\sC$.
\item The adjoint functors\/ $\boL F\:\sC[\cW_\sC^{-1}]\rarrow
\sD[\cW_\sD^{-1}]$ and\/ $\boR G\:\sD[\cW_\sD^{-1}]
\rarrow\sC[\cW_\sC^{-1}]$ are (mutually inverse) equivalences
of categories.
\end{enumerate}
\end{lem}

\begin{proof}
 This is~\cite[Proposition~1.3.13]{Hov-book}.
\end{proof}

 A Quillen adjunction $(F,G)$ is called a \emph{Quillen equivalence}
if it satisfies any one of the equivalent conditions of
Lemma~\ref{quillen-equivalence-lemma}.

\Section{All Graded-Cotorsion CDG-Modules are Cotorsion}
\label{cotorsion-cdg-modules-secn}

 In this section we prove a CDG\+module generalization of
the \emph{cotorsion periodicity theorem}~\cite[Theorem~1.2(2),
Proposition~4.8(2), or Theorem~5.1(2)]{BCE}.
 More precisely, our result is a direct generalization of
the claim that ``all complexes of cotorsion modules are dg-cotorsion''
\cite[Theorem~5.3]{BCE}.

\subsection{Cotorsion graded modules}
\label{cotorsion-graded-modules-subsecn}
 The definition of a \emph{cotorsion module} over a ring $R$ was given
in Example~\ref{flat-cotorsion-pair-example}.
 Let us spell out the graded version.

 Let $B^*$ be a graded ring.
 A graded left $B^*$\+module $C^*$ is said to be \emph{cotorsion} if
$\Ext^1_{B^*\sModl}(F^*,C^*)=0$ for all flat graded left
$B^*$\+modules~$F^*$.
 By the graded version of~\cite[Lemma~1 and Proposition~2]{BBE}
or~\cite[Lemma~6.23]{GT}, the class of flat graded $B^*$\+modules is
deconstructible in $B^*\sModl$.
 By Lemma~\ref{eklof-lemma} and Theorem~\ref{eklof-trlifaj-theorem},
it follows that the pair of classes (flat graded left $B^*$\+modules,
cotorsion graded left $B^*$\+modules) is a complete cotorsion pair
in $B^*\sModl$.
 Since the class of flat graded $B^*$\+modules is closed under
kernels of epimorphisms, this cotorsion pair is also hereditary.
 It follows that the class of all cotorsion graded modules is closed
under infinite products, extensions, and cokernels of monomorphisms
in $B^*\sModl$.

 In particular, a graded left $B^*$\+module $F^*$ is flat if and only
if $\Ext^1_{B^*\sModl}(F^*,C^*)=0$ for all cotorsion graded
$B^*$\+modules~$C^*$.
 The latter assertion admits an elementary proof avoiding any use
of the deconstructibility and small object argument machinery.
 It suffices to observe that, for any graded right $B^*$\+module $N^*$,
the graded left $B^*$\+module $\Hom_\boZ^*(N^*,\boQ/\boZ)$ is
cotorsion.
 In fact, all pure-injective graded $B^*$\+modules (as defined in
Section~\ref{pure-injective-subsecn}) are cotorsion, as one can
easily see.

It is also a general fact that the restriction functor associated with any homomorphism of graded rings and the coinduction functor associated with a left flat homomorphism of graded rings both preserve cotorsionness.

\begin{lem} \label{restriction-of-scalars-preserves-cotorsion}
 For any homomorphism of graded rings $B^*\rarrow A^*$ and any
cotorsion graded $A^*$\+module $C^*$, the underlying graded
$B^*$\+module of $C^*$ is also cotorsion.
\end{lem}

\begin{proof}
 Notice that a graded left $B^*$\+module $C^*$ is cotorsion if and only
if, for any short exact sequence of flat graded left $B^*$\+modules
$0\rarrow H^*\rarrow G^*\rarrow F^*\rarrow0$, any morphism of
graded $B^*$\+modules $H^*\rarrow C^*$ can be extended to a morphism
$G^*\rarrow C^*$.
 The assertion of the lemma now follows from the fact that
$0\rarrow A^*\ot_{B^*}H^*\rarrow A^*\ot_{B^*}G^*\rarrow A^*\ot_{B^*}F^*
\rarrow0$ is a short exact sequence of flat graded $A^*$\+modules.
\end{proof}

\begin{lem} \label{flat-extension-of-scalars-preserves-cotorsion}
 Let $B^*\rarrow A^*$ be a homomorphism of graded rings making
$A^*$ a flat left $B^*$\+module.
 Then, for any cotorsion graded left $B^*$\+module $C^*$,
the graded left $A^*$\+module\/ $\Hom_{B^*}^*(A^*,C^*)$ is cotorsion.
\end{lem}

\begin{proof}
 Argue as in the previous proof and use the observation that, for
any short exact sequence of flat left $A^*$\+modules $0\rarrow H^*
\rarrow G^*\rarrow F^*\rarrow0$, the underlying sequence of
$B^*$\+modules $0\rarrow H^*\rarrow G^*\rarrow F^*\rarrow0$ is a short
exact sequence of flat left $B^*$\+modules.
\end{proof}

\subsection{Graded-cotorsion CDG-modules}
\label{graded-cotorsion-CDG-modules-subsecn}
 Let $B^\cu=(B^*,d,h)$ be a CDG\+ring.
 A left CDG\+module $C^\cu=(C^*,d_C)$ over $B^\cu$ is said to be
\emph{graded-cotorsion} if the graded left $B^*$\+module $C^*$ is
cotorsion.
 Graded-cotorsion CDG\+modules form a full DG\+subcategory
$B^\cu\bModl^\bcot$ closed under products, shifts, twists, and
cones in the DG\+category $B^\cu\bModl$.
 Accordingly, the notation $\sZ^0(B^\cu\bModl^\bcot)$ stands for
the full subcategory of graded-cotorsion CDG\+modules in the abelian
category of CDG\+modules $\sZ^0(B^\cu\bModl)$.

 A left CDG\+module $C^\cu=(C^*,d_C)$ over $B^\cu$ is said to be
\emph{cotorsion} if it is cotorsion as a graded left
$B^*[\delta]$\+module.
 The full subcategory of cotorsion CDG\+modules is denoted by
$\sZ^0(B^\cu\bModl)^\cot\subset\sZ^0(B^\cu\bModl)$.
 However, we will see that this terminology and notation is redundant:
\emph{a CDG\+module is cotorsion if and only if it is graded-cotorsion},
so $\sZ^0(B^\cu\bModl^\bcot)=\sZ^0(B^\cu\bModl)^\cot$.
 This is the result of Theorem~\ref{cotorsion=graded-cotorsion-theorem}
below.
 The following lemma is the easier implication.

\begin{lem} \label{cotorsion-are-graded-cotorsion}
 Any cotorsion CDG\+module is graded-cotorsion.
\end{lem}

\begin{proof}
 This is Lemma~\ref{restriction-of-scalars-preserves-cotorsion}
applied to the natural injective homomorphism of graded rings
$B^*\rarrow B^*[\delta]=A^*$.
\end{proof}

 We denote by $B^\cu\bModl_\bflat^\bcot=B^\cu\bModl_\bflat\cap
B^\cu\bModl^\bcot\subset B^\cu\bModl$ the full DG\+subcategory
of graded-flat graded-cotorsion CDG\+modules (i.~e., CDG\+modules
that are simultaneously graded-flat and graded-cotorsion).
 Accordingly, the notation $\sH^0(B^\cu\bModl_\bflat^\bcot)$ stands
for the homotopy category of graded-flat graded-cotorsion CDG\+modules,
as per Section~\ref{prelim-three-additive-categories-subsecn}.

\subsection{Two-out-of-three lemma for graded-flat CDG-modules}
 The following lemma will be useful below in this
Section~\ref{cotorsion-cdg-modules-secn}.

\begin{lem} \label{two-out-of-three-orthogonal-graded-flats}
 Let $B^\cu$ be a CDG\+ring and\/ $\sT\subset\sZ^0(B^\cu\bModl^\bcot)$
be a class of graded-cotorsion left CDG\+modules over $B^\cu$ closed
under cohomological degree shifts in the DG\+category $B^\cu\bModl$. 
 Consider the left\/ $\Ext^1$\+orthogonal class\/ ${}^{\perp_1}\sT$
to\/ $\sT$ in the abelian category\/ $\sZ^0(B^\cu\bModl)$.
 In a short exact sequence of graded-flat CDG\+modules over $B^\cu$
(and closed morphisms between them), if two of the CDG\+modules belong
to the class\/ ${}^{\perp_1}\sT\subset\sZ^0(B^\cu\bModl)$, then so does
the third CDG\+module.
\end{lem}

\begin{proof}
 Let $0\rarrow F^\cu\rarrow G^\cu\rarrow H^\cu\rarrow0$ be a short
exact sequence of graded-flat left CDG\+modules over $B^\cu$, and
let $C^\cu$ be a CDG\+module from the class~$\sT$.
 Since $C^\cu$ is a graded-cotorsion CDG\+module over $B^\cu$,
we have a short exact sequence of complexes of abelian groups
$$
 0\lrarrow\Hom_{B^*}^\bu(H^\cu,C^\cu)
 \lrarrow\Hom_{B^*}^\bu(G^\cu,C^\cu)
 \lrarrow\Hom_{B^*}^\bu(F^\cu,C^\cu)\lrarrow0.
$$
 Thus if two of the three complexes of abelian groups
$\Hom_{B^*}(H^\cu,C^\cu)$, \ $\Hom_{B^*}(G^\cu,C^\cu)$, and
$\Hom_{B^*}(F^\cu,C^\cu)$ are acyclic, then so is the third one.
 It remains to take into account the natural isomorphism
$$
 \Ext^1_{\sZ^0(B^\cu\bModl)}(E^\cu,C^\cu[n-1])
 \simeq H^n\Hom_{B^*}^\bu(E^\cu,C^\cu)
 \qquad\text{for all $n\in\boZ$}
$$
provided by Lemma~\ref{Ext-1-homotopy-hom-lemma} for any
graded-flat left CDG\+module $E^\cu$ and any graded-cotorsion
left CDG\+module $C^\cu$ over~$B^\cu$,
and to note that, consequently, the complex $\Hom_{B^*}^\bu(F^\cu,C^\cu)$
is acyclic for all $C^\cu\in\sT$ if and only if $F^\cu\in{}^{\perp_1}\sT$
(and analogously for $G^\cu$ and~$H^\cu$).
\end{proof}

 The following particular case of
Lemma~\ref{two-out-of-three-orthogonal-graded-flats} will be useful
in the next Section~\ref{projective-and-flat-contraderived-secn}.

\begin{lem} \label{two-flat-out-of-three-graded-flats}
 In a short exact sequence of graded-flat CDG\+modules (and closed
morphisms between them), if two of the CDG\+modules are flat, then
so is the third.
\end{lem}

\begin{proof}
 Let $\sT=\sZ^0(B^\cu\bModl)^\cot$ be the class of all cotorsion
left CDG\+modules over~$B^\cu$.
 Then all CDG\+modules from $\sT$ are graded-cotorsion by
Lemma~\ref{cotorsion-are-graded-cotorsion}, and the class $\sT$ is
obviously closed under shifts;
so Lemma~\ref{two-out-of-three-orthogonal-graded-flats} is applicable.
 It remains to recall that a left CDG\+module $F^\cu$ over $B^\cu$ is
flat if and only if $\Ext^1_{\sZ^0(B^\cu\bModl)}(F^\cu,C^\cu)=0$ for
all cotorsion left CDG\+modules~$C^\cu$ over~$B^\cu$.
\end{proof}

\subsection{Transfinitely iterated extensions and direct limits}
 Let $\sE$ be an exact category in which all set-indexed direct limits
exist.
 We say that $\sE$ \emph{has exact direct limits} if direct limits of
admissible short exact sequences are admissible short exact sequences
in~$\sE$.
 The following result, going back to~\cite[proofs of Lemma~5.2 and
Proposition~5.3]{Sto2}, was stated for abelian categories
in~\cite[Proposition~7.16]{PS5} and for exact categories
in~\cite[Proposition~8.1]{PS6}.

\begin{prop} \label{transfinite-extensions-and-direct-limits}
 Let\/ $\sE$ be an exact category with exact direct limits, and let\/
$\sF\subset\sE$ be a class of objects closed under extensions and
the cokernels of admissible monomorphisms in\/~$\sE$.
 Then the class\/ $\sF$ is closed under transfinitely iterated
extensions if and only if it is closed under direct limits in\/~$\sE$.
\end{prop}

\begin{proof}
 This is~\cite[Proposition~8.1]{PS6}.
\end{proof}

\subsection{Graded-cotorsion is equivalent to cotorsion}
 The following theorem is the main result of
Section~\ref{cotorsion-cdg-modules-secn}.

\begin{thm} \label{cotorsion=graded-cotorsion-theorem}
 Over any CDG\+ring, the classes of cotorsion and graded-cotorsion
CDG\+modules coincide.
 In other words, $\sZ^0(B^\cu\bModl^\bcot)=\sZ^0(B^\cu\bModl)^\cot$.
\end{thm}

We will in fact present two proofs of this result.
The first of them is direct in our setting, while the second one
shows how to reduce the problem to a previously known cotorsion
periodicity theorem from~\cite{BCE} for cotorsion modules
over ordinary rings.

\begin{proof}[First proof of Theorem~\ref{cotorsion=graded-cotorsion-theorem}]
 Lemma~\ref{cotorsion-are-graded-cotorsion} tells us that all
cotorsion CDG\+modules are graded-cotorsion.
 To prove the converse implication, let $C^\cu$ be a graded-cotorsion
CDG\+module.

 Denote by $\sE=\sZ^0(B^\cu\bModl_\bflat)$ the category of
graded-flat CDG\+modules over $B^\cu$, endowed with the exact category
structure inherited from (the abelian exact structure on) the abelian
category of CDG\+modules $\sA=\sZ^0(B^\cu\bModl)$.
 Then $\sE$ is an exact category with exact direct limits.

 Denote by $\sT=\{C^\cu[n]\mid n\in\boZ\}$ the set of all cohomological
degree shifts of the CDG\+module~$C^\cu$.
 Let $\sF\subset\sE$ be the class of all graded-flat CDG\+modules
$F^\cu$ such that the Ext group
$\Ext^1_{\sZ^0(B^\cu\bModl)}(F^\cu,C^\cu[n])$ vanishes for all
$n\in\boZ$.
 So $\sF=\sE\cap{}^{\perp_1}\sT$.

 By Lemma~\ref{two-out-of-three-orthogonal-graded-flats}, the class
$\sF$ is closed under the cokernels of admissible monomorphisms
in~$\sE$.
 By Lemma~\ref{eklof-lemma}, the class $\sF$ is also closed under
transfinitely iterated extensions in\/~$\sE$.
 Applying Proposition~\ref{transfinite-extensions-and-direct-limits},
we conclude that the class $\sF$ is closed under direct limits
in~$\sE$.
 As projective CDG\+modules obviously belong to $\sF$, it follows
that all flat left CDG\+modules over $B^\cu$ belong to~$\sF$.

 Thus $\Ext^1_{\sZ^0(B^\cu\bModl)}(F^\cu,C^\cu)=0$ for all flat
left CDG\+modules $F^\cu$ over~$B^\cu$.
 So $C^\cu$ is a cotorsion CDG\+module, as desired.
\end{proof}

\begin{proof}[Second proof of Theorem~\ref{cotorsion=graded-cotorsion-theorem}]
 Let $C^\cu=(C^*,d_C)$ be a graded-cotorsion CDG\+module over~$B^\cu$.
 It was mentioned in Section~\ref{prelim-proj-inj-flatness-subsecn}
that the graded ring $A^*=B^*[\delta]$, viewed as a left (or right)
graded $B^*$\+module, is a free graded module with two generators.
 Following the discussion in Section~\ref{prelim-delta-and-G-subsecn}
and Lemma~\ref{flat-extension-of-scalars-preserves-cotorsion}, we see
that the graded left $B^*[\delta]$\+module
$\Hom^*_{B^*}(B^*[\delta],C^*)$ is cotorsion.
 So $G^-(C^*)=\Hom^*_{B^*}(B^*[\delta],C^*)$ is a cotorsion
CDG\+module over~$B^\cu$.

 Similarly to the short exact
sequences~(\ref{G-plus-short-exact-sequence}--%
\ref{G-minus-short-exact-sequence}) from
Section~\ref{prelim-delta-and-G-subsecn}, for any CDG\+module $M^\cu$
over a CDG\+ring $B^\cu$ there are natural short exact sequences
of CDG\+modules (and closed morphisms of degree~$0$)
\begin{equation} \label{G-plus-CDG-module-sequence}
 0\lrarrow M^\cu[-1]\lrarrow G^+(M^*)\lrarrow M^\cu\lrarrow0
\end{equation}
and
\begin{equation} \label{G-minus-CDG-module-sequence}
 0\lrarrow M^\cu\lrarrow G^-(M^*)\lrarrow M^\cu[1]\lrarrow0,
\end{equation}
which are transformed into each other by the shift functors~$[1]$
and~$[-1]$.

 In particular, in the situation at hand, splicing up a doubly
unbounded sequence of shifts of the short exact
sequences~\eqref{G-minus-CDG-module-sequence}, we obtain an acyclic
complex
\begin{equation} \label{shift-periodic-acyclic-complex-of-cotorsion}
 \dotsb\lrarrow G^-(C^*)[-1]\lrarrow G^-(C^*)\lrarrow G^-(C^*)[1]
 \lrarrow\dotsb
\end{equation}
in the abelian category $\sZ^0(B^\cu\bModl)\simeq B^*[\delta]\sModl$.
 The CDG\+module $C^\cu$ is one of the objects of cocycles of
the acyclic complex~\eqref{shift-periodic-acyclic-complex-of-cotorsion}.
 As $G^-(C^*)[n]$ are cotorsion $B^*[\delta]$\+modules for all
$n\in\boZ$, the graded module version of the cotorsion periodicity
theorem~\cite[Theorem~5.1(2)]{BCE} implies that $C^\cu$ is a cotorsion
$B^*[\delta]$\+module.
\end{proof}

\begin{ex}
 Let $R$ be an associative ring, viewed as a CDG\+ring concentrated
in cohomological degree~$0$ (with zero differential and zero
curvature element), as in Remark~\ref{flat-complexes-remark}.
 Following the terminology of~\cite[Definition~3.3(4)]{Gil},
complexes of left $R$\+modules right $\Ext^1$\+orthogonal to
acyclic complexes of flat left $R$\+modules with flat modules
of cocycles are called ``dg-cotorsion''.

 As the acyclic complexes of flat $R$\+modules with flat modules
of cocycles are precisely the flat $R[\delta]$\+modules (according
to Remark~\ref{flat-complexes-remark}), the dg-cotorsion complexes
of $R$\+modules are precisely the cotorsion graded
$R[\delta]$\+modules.
 So what is called ``dg-cotorsion complexes of modules'' in
the terminology of~\cite{Gil,BCE} become the cotorsion CDG\+modules
(as defined in Section~\ref{graded-cotorsion-CDG-modules-subsecn})
in our CDG\+module context.

 On the other hand, what we call graded-cotorsion CDG\+modules become,
in the case of the CDG\+ring $(R,0,0)$, just graded
$R[\delta]$\+modules that are cotorsion as graded $R$\+modules;
this means arbitrary complexes of cotorsion $R$\+modules.
 Thus the result of~\cite[Theorem~5.3]{BCE}, which claims that all
complexes of cotorsion modules are dg-cotorsion, is a particular
case of our Theorem~\ref{cotorsion=graded-cotorsion-theorem}.
 Specifically, Theorem~\ref{cotorsion=graded-cotorsion-theorem}
is a generalization of~\cite[Theorem~5.3]{BCE} from the case of
ungraded rings viewed as CDG\+rings concentrated in degree~$0$,
with zero differential and zero curvature, to arbitrary CDG\+rings.
\end{ex}

\Section{Projective and Flat Contraderived Categories}
\label{projective-and-flat-contraderived-secn}

\subsection{Posing the problem}
\label{three-contraderived-posing-the-problem-subsecn}
 Let $B^\cu=(B^*,d,h)$ be a CDG\+ring.
 Consider the following three triangulated categories of left
CDG\+modules over~$B^\cu$:
\begin{enumerate}
\renewcommand{\theenumi}{\roman{enumi}}
\item the homotopy category of graded-projective CDG\+modules
$\sH^0(B^\cu\bModl_\bproj)$, which we will call the \emph{projective contraderived category} of $B^\cu$;
\item the triangulated Verdier quotient category
$$
 \sH^0(B^\cu\bModl_\bflat)/\sH^0(B^\cu\bModl)_\flat
$$
of the homotopy category of graded-flat CDG\+modules
$\sH^0(B^\cu\bModl_\bflat)$ by its full triangulated subcategory
of flat CDG\+modules $\sH^0(B^\cu\bModl)_\flat$, which we call the \emph{flat contraderived category} of $B^\cu$;
\item the homotopy category of graded-flat graded-cotorsion CDG\+modules
$$
 \sH^0(B^\cu\bModl_\bflat^\bcot)=
 \sH^0(B^\cu\bModl_\bflat\cap B^\cu\bModl^\bcot).
$$
\end{enumerate}
 See Sections~\ref{prelim-proj-inj-flatness-subsecn}
and~\ref{graded-cotorsion-CDG-modules-subsecn} for the notation and terminology.

 The main aim of this
Section~\ref{projective-and-flat-contraderived-secn}
is to show that the natural triangulated functor from the category~(i)
to the category~(ii) is a triangulated equivalence.
 In fact, we will see that the full triangulated subcategories
$\sH^0(B^\cu\bModl)_\flat$ and $\sH^0(B^\cu\bModl_\bproj)$ form
a semiorthogonal decomposition of the homotopy category of graded-flat
CDG\+modules $\sH^0(B^\cu\bModl_\bflat)$.
 Thus we will obtain a generalization of~\cite[Theorem~8.6]{Neem}
(which is a theorem about complexes of flat modules) to
the CDG\+module case.

 Furthermore, we will prove that the natural triangulated functor
from the category~(iii) to the category~(ii) is also a triangulated
equivalence.
 In fact, we will see that the full triangulated subcategories
$\sH^0(B^\cu\bModl_\bflat^\bcot)$ and $\sH^0(B^\cu\bModl)_\flat$
form another semiorthogonal decomposition of the homotopy category
$\sH^0(B^\cu\bModl_\bflat)$,
which generalizes~\cite[Corollary 5.8]{Sto2} (a finitely accessible additive
category mentioned there is just another name for a category of flat modules
over a ring with several objects).
 Thus, all the three triangulated categories~(i--iii)
are naturally equivalent, and we obtain
the recollement~\eqref{recollement-for-graded-flat-CDG-modules}
from Section~\ref{introd-description-of-results-subsecn}.

\subsection{Graded-projective and graded-free CDG-modules}
 The aim of this section is to prove the following lemma.

\begin{lem} \label{graded-projective-direct-summand-of-graded-free}
 Let $B^\cu=(B^*,d,h)$ be a CDG\+ring.
 Then, in the abelian category of CDG\+modules\/ $\sZ^0(B^\cu\bModl)$,
any graded-projective CDG\+module over $B^\cu$ is a direct summand
of a graded-free one.
\end{lem}

\begin{proof}
 Let $P^\cu=(P^*,d_P)$ be a graded-projective left CDG\+module
over~$B^\cu$.
 Consider the free graded left $B^*$\+module
$F^*=\bigoplus_{n\in\boZ}B^*[n]$.
 Then there exists a cardinal~$\mu$ such that $P^*$ is a direct
summand of the direct sum $F^*{}^{(\mu)}$ of $\mu$~copies of~$F^*$.

 Put $H^*=F^*{}^{(\mu)}$ and $E^*=H^*{}^{(\aleph_0)}=
F^*{}^{(\mu\times\aleph_0)}$ (where $\aleph_0$~is
the countable cardinal).
 As mentioned in Section~\ref{prelim-delta-and-G-subsecn}, there is
a short exact sequence of graded $B^*$\+modules $0\rarrow E^*
\rarrow G^+(E^*)^\#\rarrow E^*[-1]\rarrow0$
\,\eqref{G-plus-short-exact-sequence}.
 Since $E^*[-1]$ is a free graded $B^*$\+module, this short exact
sequence splits; so $G^+(E^*)^\#\simeq E^*\oplus E^*[-1]$.

 Now $P^\cu\oplus G^+(E^*)$ is a graded-free CDG\+module over~$B^\cu$.
 Indeed, the graded $B^*$\+module $P^*\oplus H^*{}^{(\aleph_0)}$ is
isomorphic to $H^*{}^{(\aleph_0)}$ by the cancellation trick (since
$P^*$ is a direct summand of~$H^*$).
 Hence $(P^\cu\oplus G^+(E^*))^\#\simeq P^*\oplus E^*\oplus E^*[-1]
\simeq E^*\oplus E^*[-1]$ is a free graded $B^*$\+module
(cf.~\cite[beginning of second proof of Proposition~6.15]{PS5}).
\end{proof}

\subsection{Deconstruction of graded-free CDG-modules}
 The aim of this section is to prove the next lemma.

\begin{lem} \label{deconstruction-of-graded-free}
 Let $B^\cu=(B^*,d,h)$ be a CDG\+ring.
 Then, in the abelian category of CDG\+modules\/ $\sZ^0(B^\cu\bModl)$,
any graded-free CDG\+module over $B^\cu$ is filtered by countably
generated graded-free ones.
\end{lem}

\begin{proof}
 Let $P^\cu=(P^*,d_P)$ be a graded-free left CDG\+module over~$B^\cu$,
and let $X$ be a set of homogeneous generators of the free graded
$B^*$\+module~$P^*$; so $P^*=\bigoplus_{x\in X}B^*x$, where
$x\in P^{|x|}$ for every $x\in X$.
 Proceeding by transfinite induction, we construct a filtration
$X=\bigcup_{i=0}^\alpha X_i$ of the set $X$, indexed by some
ordinal~$\alpha$, with the following properties:
\begin{itemize}
\item $X_0=\varnothing$ and $X_\alpha=X$;
\item $X_i\subset X_j$ for all $0\le i\le j\le\alpha$;
\item $X_j=\bigcup_{i<j}X_i$ for all limit ordinals $j\le\alpha$;
\item the set $X_{i+1}\setminus X_i$ is (at most) countable for
every ordinal $0\le i<\alpha$;
\item one has $d_P(x)\subset\bigoplus_{y\in X_i}B^*y$ for every
$x\in X_i$ and $0\le i\le\alpha$.
\end{itemize}

 Let $0\le j\le\alpha$ be an ordinal.
 Assuming that the subsets $X_i\subset X$ have been already constructed
for all $i<j$, the subset $X_j\subset X$ is constructed as follows.
 For $j=0$ we put $X_0=\varnothing$; and if~$j$ is a limit ordinal,
we simply put $X_j=\bigcup_{i<j}X_i$.

 For a successor ordinal $j=i+1$, we consider two cases.
 If $X_i=X$, then the construction is over; put $\alpha=i$.
 The nontrivial case occurs when $X_i\ne X$.

 Then we pick an element $x\in X\setminus X_i$ and proceed by induction
on nonnegative integers $n\ge0$, constructing finite subsets $Z_0$, 
$Z_1$, $Z_2$,~\dots~$\subset X$.
 Put $Z_0=\{x\}$.
 Assuming that a finite subset $Z_n\subset X$, \,$n\ge0$, has been
constructed already, let $z\in Z_n$ be an element.
 Then there exists a finite subset $W_z\subset X$ such that
$d_P(z)\in\bigoplus_{w\in W_z}B^*w\subset P^*$.
 Put $Z_{n+1}=\bigcup_{z\in Z_n} W_z$.

 After the subset $Z_n\subset X$ has been constructed for every
integer $n\ge0$, put $X_{i+1}=X_i\cup\bigcup_{n\ge0}Z_n$.
 This finishes the construction of the ordinal-indexed filtration
$X=\bigcup_{i=0}^\alpha X_\alpha$ on the set~$X$.

 Finally, the rules $P^*_i=\bigoplus_{x\in X_i}B^*x$ and
$P^\cu_i=(P^*_i,d_P|_{P^*_i})$ for all $0\le i\le\alpha$ define
an $\alpha$\+indexed filtration by CDG\+submodules on the graded-free
CDG\+module $P^\cu$ over $B^\cu$ with countably generated graded-free
quotient CDG\+modules, as desired.
 (Cf.~\cite[main argument in the second proof of
Proposition~6.15]{PS5}).
\end{proof}

\subsection{Countably presented graded-free CDG-modules}
 In this section we prove the following proposition.

\begin{prop} \label{countable-graded-free-orthogonal-to-flats}
 Let $B^\cu=(B^*,d,h)$ be a CDG\+ring, $Q^\cu=(Q^*,d_Q)$ be
a countably generated graded-free left CDG\+module over $B^\cu$,
and $F^\cu=(F^*,d_F)$ be a flat left CDG\+module over~$B^\cu$.
 Then any closed morphism of CDG\+modules $Q^\cu\rarrow F^\cu$ is
homotopic to zero.
 In other words, the complex of abelian groups\/
$\Hom^\bu_{B^*}(Q^\cu,F^\cu)$ is acyclic.
\end{prop}

\begin{proof}
 Let us first consider the trivial case when $Q^\cu$ is actually
a finitely generated graded-free left CDG\+module over~$B^\cu$.
 In this case, by
Corollary~\ref{graded-projective-generated-presented-cor}(a),
the graded $B^*[\delta]$\+module $Q^*$ is finitely presented.
 By the Govorov--Lazard theorem, it follows that any morphism
of $B^*[\delta]$\+modules from $Q^*$ to the flat $B^*[\delta]$\+module
$F^*$ factorizes through a (finitely generated) projective
$B^*[\delta]$\+module.
 Following the discussion in
Section~\ref{prelim-proj-inj-flatness-subsecn}, all projective
$B^*[\delta]$\+modules are contractible as CDG\+modules over~$B^\cu$.
 So any morphism $Q^\cu\rarrow F^\cu$ in $\sZ^0(B^\cu\bModl)\simeq
B^*[\delta]\sModl$ factorizes through a contractible CDG\+module;
consequently, any such morphism is homotopic to zero.

 Now we deal with the interesting case when the $B^*$\+module $Q^*$ has
a countably infinite set of homogeneous free generators.
 Let $x_0$, $x_1$, $x_2$,~\dots~$\in Q^*$ be a sequence of such
generators.
 For every integer $n\ge0$, we define a finitely presented graded left
$B^*[\delta]$\+module $Q_n^*$ by generators and relations as follows.

 The $(n+1)$\+element set $\{x_0,x_1,\dotsc,x_n\}$ is the set of
generators of the graded $B^*[\delta]$\+module~$Q_n^*$.
 Here $x_i\in Q_n^*$, \ $0\le i\le n$, are homogeneous elements of
the same degrees as the similarly denoted free generators of
the graded $B^*$\+module~$Q^*$.

 To describe the relations imposed on these generators in $Q_n^*$,
consider a nonnegative integer $i\in\boZ_{\ge0}$, and denote by
$Z_i\subset\boZ_{\ge0}$ the (say, minimal) finite set of indices
such that $d_Q(x_i)\in\bigoplus_{z\in Z_i}B^*x_z\subset Q^*$.
 Let us write $d_Q(x_i)=\sum_{z\in Z_i}b_{i,z}x_z\in Q^*$, where
$b_{i,z}\in B^*$ are some homogeneous elements.
 For every index $0\le i\le n$ such that $Z_i\subset\{0,\dotsc,n\}$,
the relation $\delta x_i=\sum_{z\in Z_i}b_{i,z}x_z$ is imposed
in the graded $B^*[\delta]$\+module~$Q_n^*$.
 If $Z_i\not\subset\{0,\dotsc,n\}$, then no such relation is imposed.

 Clearly, the graded $B^*[\delta]$\+module $Q^*$ is the direct limit
of the sequence of finitely presented graded $B^*[\delta]$\+modules
$Q_n^*$, that is $Q^*=\varinjlim_{n\ge0}Q_n^*$ in $B^*[\delta]\sModl$.
 Hence we have the telescope short exact sequence
\begin{equation} \label{telescope-sequence}
 0\lrarrow\bigoplus\nolimits_{n\ge0}Q_n^*\lrarrow
 \bigoplus\nolimits_{n\ge0}Q_n^*\lrarrow
 \varinjlim\nolimits_{n\ge0}Q_n^*=Q\lrarrow0
\end{equation}
in $B^*[\delta]\sModl$.

 Moreover, the short exact sequence~\eqref{telescope-sequence} splits
in the additive category of graded $B^*$\+modules $B^*\sModl$, as
the graded $B^*$\+module $Q^*$ is free.
 Consequently, denoting by $Q_n^\cu\in B^\cu\bModl$ the CDG\+module
corresponding to the graded $B^*[\delta]$\+module $Q_n^*$, we have
a distinguished triangle in the homotopy category $\sH^0(B^\cu\bModl)$
\begin{equation} \label{telescope-triangle}
 \bigoplus\nolimits_{n\ge0}Q_n^\cu\lrarrow
 \bigoplus\nolimits_{n\ge0}Q_n^\cu\lrarrow
 Q^\cu\lrarrow\bigoplus\nolimits_{n\ge0}Q_n^\cu[1].
\end{equation}

 Following the argument in the first paragraph of this proof, any
closed morphism of CDG\+modules $Q_n^\cu\rarrow F^\cu$ is homotopic
to zero (since the $B^*[\delta]$\+module $Q_n^*$ is finitely
presented, while the $B^*[\delta]$\+module $F^*$ is flat).
 The same holds for all shifts of the CDG\+module~$F^\cu$; so
the complex of morphisms $\Hom_{B^*}^\bu(Q_n^\cu,F^\cu)$ is acyclic
for every $n\ge0$.
 Applying the cohomological functor $\Hom({-},F^\cu)$ in the homotopy
category $\sH^0(B^\cu\bModl)$ to the distinguished
triangle~\eqref{telescope-triangle}, we conclude that all closed
morphisms of CDG\+modules $Q^\cu\rarrow F^\cu$ are homotopic to zero.
\end{proof}

\subsection{The contraderived category and contraderived cotorsion pair}
\label{contraderived-subsecn}
 Let $B^\cu=(B^*,d,h)$ be a CDG\+ring.
 A left CDG\+module $X^\cu$ over $B^\cu$ is said to be
\emph{contraacyclic} (\emph{in the sense of
Becker}~\cite[Proposition~1.3.8(1)]{Bec}) if, for any graded-projective
left CDG\+module $P^\cu$ over $B^\cu$, any closed morphism of
CDG\+modules $P^\cu\rarrow X^\cu$ is homotopic to zero.
 Equivalently (in view of Lemma~\ref{Ext-1-homotopy-hom-lemma}),
this means that $\Ext^1_{\sZ^0(B^\cu\bModl)}(P^\cu,X^\cu)=0$ for all
graded-projective CDG\+modules~$P^\cu$.

 The full subcategory of contraacyclic CDG\+modules is denoted by
$\sH^0(B^\cu\bModl)_\ac^\bctr\allowbreak\subset\sH^0(B^\cu\bModl)$ or
$\sZ^0(B^\cu\bModl)_\ac^\bctr\subset\sZ^0(B^\cu\bModl)$.
 The class of contraacyclic CDG\+modules is closed under extensions
and infinite products in the abelian category $\sZ^0(B^\cu\bModl)$
\,\cite[Lemma~6.13]{PS5}, and under shifts, cones, and direct
summands in the triangulated category $\sH^0(B^\cu\bModl)$.

 The triangulated Verdier quotient category
$$
 \sD^\bctr(B^\cu\bModl)=\sH^0(B^\cu\bModl)/\sH^0(B^\cu\bModl)_\ac^\bctr
$$
is called the \emph{contraderived category} (\emph{in the sense of
Becker}) of left CDG\+modules over~$B^\cu$.
 The following theorem is due to
Becker~\cite[Proposition~1.3.6(1)]{Bec}.

\begin{thm} \label{contraderived-cotorsion-pair-category-theorem}
 Let $B^\cu=(B^*,d,h)$ be a CDG\+ring.  Then \par
\textup{(a)} the pair of classes of objects\/
$\big(\sZ^0(B^\cu\bModl_\bproj)$, $\sZ^0(B^\cu\bModl)_\ac^\bctr\big)$ is
a hereditary complete cotorsion pair in the abelian category\/
$\sZ^0(B^\cu\bModl)$; \par
\textup{(b)} the composition of the fully faithful inclusion of
triangulated categories\/ $\sH^0(B^\cu\bModl_\bproj)\rarrow
\sH^0(B^\cu\bModl)$ and the triangulated Verdier quotient functor\/
$\sH^0(B^\cu\bModl)\rarrow\sD^\bctr(B^\cu\bModl)$ is a triangulated
equivalence
$$
 \sH^0(B^\cu\bModl_\bproj)\simeq\sD^\bctr(B^\cu\bModl).
$$
\end{thm}

\begin{proof}
 The assertions of parts~(a) and~(b) are closely related to one
another; in fact, part~(b) is a corollary of part~(a).
 Both the assertions are implied by the formulation in~\cite{Bec};
they can be obtained from the main result of~\cite[Section~6]{PS5}
by specializing it from locally presentable abelian DG\+categories
with enough projective objects to the DG\+categories of CDG\+modules.
 Part~(a) is a particular case of~\cite[Theorem~6.16]{PS5},
and part~(b) is a particular case of~\cite[Corollary~6.14]{PS5}.

 Notice that the proofs in both~\cite{Bec} and~\cite{PS5} are based on
a deconstructibility result for graded-projective
CDG\+modules~\cite[Proposition~A.10]{Bec}, \cite[Proposition~6.15]{PS5}.
 The argument in~\cite[second proof of Proposition~6.15]{PS5} is
covered by Lemmas~\ref{graded-projective-direct-summand-of-graded-free}
and~\ref{deconstruction-of-graded-free} above.
 So part~(a) can be deduced from
Lemmas~\ref{graded-projective-direct-summand-of-graded-free}--%
\ref{deconstruction-of-graded-free},
Lemma~\ref{Ext-1-homotopy-hom-lemma}, and
Theorem~\ref{eklof-trlifaj-theorem};
while part~(b) follows from part~(a) as explained
in~\cite[proof of Corollary~6.14 in Section~6.5]{PS5}.
\end{proof}

The hereditary complete cotorsion pair $\big(\sZ^0(B^\cu\bModl_\bproj)$,
$\sZ^0(B^\cu\bModl)_\ac^\bctr\big)$
formed by the classes of graded-projective CDG\+modules and of
contraacyclic CDG\+modules provided by
Theorem~\ref{contraderived-cotorsion-pair-category-theorem}(a) is called
the \emph{contraderived cotorsion pair} in the abelian category of
CDG\+modules $\sZ^0(B^\cu\bModl)$.

\subsection{Restricted contraderived cotorsion pair in graded-flat
CDG-modules} \label{restricted-contraderived-cotorsion-pair-subsecn}
 Denote by $\sA$ the abelian category of CDG\+modules
$\sA=\sZ^0(B^\cu\bModl)$, by $\sE\subset\sA$ the full subcategory of
graded-flat CDG\+modules $\sE=\sZ^0(B^\cu\bModl_\bflat)$, by
$\sL\subset\sA$ the class of all graded-projective CDG\+modules
$\sL=\sZ^0(B^\cu\bModl_\bproj)$, and by $\sR\subset\sA$ the class of
all contraacyclic CDG\+modules $\sR=\sZ^0(B^\cu\bModl)_\ac^\bctr$.
 Then the assumptions of Lemma~\ref{restricted-exact-structure} are
satisfied for the hereditary complete cotorsion pair $(\sL,\sR)$ in
the abelian category $\sA$ and the full subcategory $\sE\subset\sA$.

 Therefore, the cotorsion pair $(\sL,\sR)$ restricts to a hereditary
complete cotorsion pair ($\sL$, $\sE\cap\sR$) in the exact
category~$\sE$.
 In other words, the pair of classes $\big(\sZ^0(B^\cu\bModl_\bproj)$,
$\sZ^0(B^\cu\bModl_\bflat)_\ac^\bctr\big)$
formed by the classes of graded-projective CDG\+modules and graded-flat
contraacyclic CDG\+modules is a hereditary complete
cotorsion pair in the exact category of graded-flat CDG\+modules
$\sZ^0(B^\cu\bModl_\bflat)$.

 What are the graded-flat contraacyclic CDG\+modules, is there a simpler
description of this class?
 Using Proposition~\ref{countable-graded-free-orthogonal-to-flats},
we can already prove the following assertion.

\begin{cor} \label{flats-are-contraacyclic-cor}
 Over any CDG\+ring $B^\cu$, all flat CDG\+modules are contraacyclic.
\end{cor}

\begin{proof}
 Let $F^\cu$ be a flat left CDG\+module over~$B^\cu$.
 We have to show that the complex of abelian groups
$\Hom_{B^*}^\bu(P^\cu,F^\cu)$ is acyclic for all graded-projective
left CDG\+modules $P^\cu$ over~$B^\cu$.
 In view of Lemma~\ref{Ext-1-homotopy-hom-lemma},
this equivalent to showing that the Ext group
$\Ext^1_{\sZ^0(B^\cu\bModl)}(P^\cu,F^\cu[-1])$ vanishes.
 By Lemmas~\ref{graded-projective-direct-summand-of-graded-free}
and~\ref{deconstruction-of-graded-free}, the object $P^\cu\in
\sZ^0(B^\cu\bModl)$ is a direct summand of a CDG\+module filtered
by countably generated graded-free CDG\+modules.
 Using the Eklof lemma (Lemma~\ref{eklof-lemma}), the question reduces
to showing that $\Ext^1_{\sZ^0(B^\cu\bModl)}(Q^\cu,F^\cu[-1])=0$ for
all countably generated graded-free left CDG\+modules $Q^\cu$
over~$B^\cu$.
 Applying Lemma~\ref{Ext-1-homotopy-hom-lemma} again, we finally reduce
the question to
Proposition~\ref{countable-graded-free-orthogonal-to-flats}.
\end{proof}

 One of the main aims of the rest of
Section~\ref{projective-and-flat-contraderived-secn} is to prove
the following converse assertion to
Corollary~\ref{flats-are-contraacyclic-cor}: \emph{all contraacyclic
graded-flat CDG\+modules are flat}.
 This is the result of
Corollary~\ref{flat=graded-flat-and-contraacyclic-cor} below.

\subsection{\texorpdfstring{Graded-flat CDG-modules as $\aleph_1$-direct limits}{Graded-flat CDG-modules as aleph1-direct limits}}
 We start with a lemma which is certainly not new; see, e.~g.,
\cite[proof of Theorem~2.11\,(iv)\,$\Rightarrow$\,(i)]{AR}.

\begin{lem} \label{direct-limit=aleph-1-direct-of-countables}
 Let $B^*$ be a graded ring.
 Then any flat graded $B^*$\+module is an $\aleph_1$\+direct limit of
countably presented flat graded $B^*$\+modules.
\end{lem}

\begin{proof}
 Let $F^*$ be a flat graded $B^*$\+module.
 By the graded version of the Govorov--Lazard theorem, $F^*$ is
a direct limit of finitely generated projective (or even free)
graded $B^*$\+modules.
 Clearly, any countable direct limit of finitely generated projective
graded modules is a countably presented flat graded module.
 To prove the lemma, we need to transform a direct limit into
an $\aleph_1$\+direct limit of countable direct limits.

 So, quite generally, let $\sC$ be a category with direct limits,
let $\Xi$ be a directed poset, and let $(C_\xi)_{\xi\in\Xi}$ be
a $\Xi$\+indexed diagram in~$\sC$.
 Denote by $\Theta$ the set of all countable directed subposets
in~$\Xi$.
 To any element $T\in\Theta$, we assign the subset $Y_T\subset\Xi$
consisting of all elements $\upsilon\in\Xi$ for which there exists
$\theta\in T$ with $\upsilon\le\theta$ in~$\Xi$.
 Clearly, $Y_T$ is also a directed subposet in $\Xi$, and
$\varinjlim_{\upsilon\in Y_T}C_\upsilon=
\varinjlim_{\theta\in T}C_\theta$.
 Introduce a partial preorder on $\Theta$ by the rule that
$T'\le T''$ in $\Theta$ if $Y_{T'}\subset Y_{T''}$.
 Then (the poset of equivalence classes in) the partially
preordered set $\Theta$ is $\aleph_1$\+directed and
$\varinjlim_{\xi\in\Xi}C_\xi=\varinjlim_{T\in\Theta}
(\varinjlim_{\theta\in T}C_\theta)$.
\end{proof}

 The aim of this section is to prove the following proposition.

\begin{prop} \label{graded-flats-as-aleph-1-direct-limits}
 Let $B^\cu=(B^*,d,h)$ be a CDG\+ring.
 Then any graded-flat CDG\+module over $B^\cu$ can be obtained as
an\/ $\aleph_1$\+direct limit of countably presented graded-flat
CDG\+modules over $B^\cu$ in the abelian category of CDG\+modules\/
$\sZ^0(B^\cu\bModl)$.
\end{prop}

\begin{proof}
 In view of~\cite[Proposition~6.1]{Pgen} for $\kappa=\aleph_1$
(which is the countably presented version of
the classical~\cite[Proposition~2.1]{Len} or~\cite[Section~4.1]{CB}),
it suffices to show that any closed morphism from a countably presented
CDG\+module to a graded-flat one factorizes through a countably
presented graded-flat one in $\sZ^0(B^\cu\bModl)$.

 Indeed, let $M^\cu=(M^*,d_M)$ be a countably presented left CDG\+module
over $B^\cu$ and $F^\cu=(F^*,d_F)$ be a graded-flat left CDG\+module
over~$B^\cu$.
 Suppose given a closed morphism of CDG\+modules $M^\cu\rarrow F^\cu$.
 Put $M^\cu_0=M^\cu$.
 Then, by Lemma~\ref{direct-limit=aleph-1-direct-of-countables},
the underlying morphism of graded $B^*$\+modules $M^*_0\rarrow F^*$
factorizes through a countably presented flat graded left
$B^*$\+module~$K^*_0$,
\begin{equation} \label{graded-factorization-through-countable-flat}
 M_0^*\lrarrow K_0^*\lrarrow F^*.
\end{equation}

 Applying the functor $G^+$ to the morphism of graded $B^*$\+modules
$M^*_0\rarrow K^*_0$, we obtain a morphism of CDG\+modules
$G^+(M^*_0)\rarrow G^+(K^*_0)$ over~$B^\cu$.
 We also have the adjunction morphism of CDG\+modules
$G^+(M^*_0)\rarrow M^\cu_0$.
 Consider the pushout $M_1^\cu$ of these two morphisms in the abelian
category $\sZ^0(B^\cu\bModl)\simeq B^*[\delta]\sModl$,
\begin{equation} \label{pushout-factorization-diagram}
\begin{gathered}
 \xymatrix{
  G^+(M_0^*) \ar[r] \ar[d] & G^+(K_0^*) \ar[r] \ar[d]
  & G^+(F^*) \ar[d] \\
  M_0^\cu \ar[r] & M_1^\cu \ar[r] & F^\cu
 }
\end{gathered}
\end{equation}
 Here the upper line of
the diagram~\eqref{pushout-factorization-diagram} is obtained by
applying the functor $G^+$ to the morphisms $M_0^*\rarrow K_0^*
\rarrow F^*$ \,\eqref{graded-factorization-through-countable-flat}.
 The leftmost and rightmost vertical arrows are the adjunction morphisms.
 The morphism $M_1^\cu\rarrow F^\cu$ is defined by the conditions
that the composition $M_0^\cu\rarrow M_1^\cu\rarrow F^\cu$ is
the original morphism $M_0^\cu\rarrow F^\cu$ and the rightmost square
of the diagram is commutative.
 The CDG\+module $M_1^\cu$ is countably presented, since
the CDG\+modules $G^+(M_0^*)$, \ $G^+(K_0^*)$, and $M_0^\cu$ are
countably presented and pushouts preserve countable presentability.
 (Notice that the functor $G^+$ preserves countable presentability
in view of the short exact
sequence~\eqref{G-plus-short-exact-sequence} from
Section~\ref{prelim-delta-and-G-subsecn}.)

 Applying the forgetful functor $N\longmapsto N^\#$ to
the diagram~\eqref{pushout-factorization-diagram}, we obtain
a commutative diagram in the category of graded $B^*$\+modules,
which can be complemented by the adjunction morphism of graded
$B^*$\+modules $M_0^*\rarrow G^+(M_0^*)^\#$ as follows:
\begin{equation} \label{pushout-graded-factorization-diagram}
\begin{gathered}
 \xymatrix{
  M_0^* \ar[d] \\
  G^+(M_0^*)^\# \ar[r] \ar[d] & G^+(K_0^*)^\# \ar[r] \ar[d]
  & G^+(F^*)^\# \ar[d] \\
  M_0^* \ar[r] & M_1^* \ar[r] & F^*
 }
\end{gathered}
\end{equation}
 By a general property of adjoint functors, the vertical composition
$M_0^*\rarrow G^+(M_0^*)\rarrow M_0^*$ is the identity map.

 Thus the morphism of graded $B^*$\+modules $M_0^*\rarrow M_1^*$
underlying the morphism of CDG\+modules $M_0^\cu\rarrow M_1^\cu$
factorizes through the underlying graded module $G^+(K_0^*)^\#$
of the CDG\+module $G^+(K_0^*)$ over~$B^\cu$.
 Furthermore, the graded $B^*$\+module $G^+(K_0^*)^\#$ is countably
presented and flat in view of the short exact sequence
$0\rarrow K_0^*\rarrow G^+(K_0)^\#\rarrow K_0^*[-1]\rarrow0$
\,\eqref{G-plus-short-exact-sequence}
from Section~\ref{prelim-delta-and-G-subsecn}.

 Applying the same construction to the morphism of CDG\+modules
$M_1^\cu\rarrow F^\cu$ and proceeding by induction over nonnegative
integers, we construct a sequence of closed morphisms of CDG\+modules
over~$B^\cu$
\begin{equation} \label{sequence-of-factorizations-cdg}
 M_0^\cu\lrarrow M_1^\cu\lrarrow M_2^\cu\lrarrow\dotsb\lrarrow F^\cu
\end{equation}
together with a sequence of factorizations in the category of
graded $B^*$\+modules
\begin{equation} \label{sequence-of-factorizations-graded}
\begin{gathered}
 \xymatrixcolsep{1.2em}
 \xymatrix{
  M_0^* \ar[rr] \ar[rd] && M_1^* \ar[rr] \ar[rd]
  &&  M_2^* \ar[rr] \ar[rd] && \dotsb \ar[rr] && F^* \\
  & G^+(K_0^*)^\# \ar[ru] && G^+(K_1^*)^\# \ar[ru] && 
  G^+(K_2^*)^\# \ar[ru] \\
 }
\end{gathered}
\end{equation}

 Finally, taking the direct limits
in~\eqref{sequence-of-factorizations-cdg}
and~\eqref{sequence-of-factorizations-graded} and setting
$H^\cu=\varinjlim_{n\ge0}M_n^\cu$, we obtain a factorization
in the category of CDG\+modules $\sZ^0(B^\cu\bModl)$
for the original morphism $M^\cu\rarrow F^\cu$,
$$
 M^\cu=M_0^\cu\lrarrow H^\cu\lrarrow F^\cu,
$$
and an isomorphism in the category of graded $B^*$\+modules
$$
 H^*\,\simeq\,\varinjlim\nolimits_{\ge0}G^+(K_n^*)^\#.
$$
 The graded $B^*$\+module $H^*$ is countably presented and flat
as a countable direct limit of countably presented flat graded
$B^*$\+modules $G^+(K_n^*)^\#$, as desired. 
\end{proof}

\begin{rem}\label{remark-graded-flats-via-pseudopullback}
In more fancy terms, the conclusion of Proposition~\ref{graded-flats-as-aleph-1-direct-limits} is a special case of the Pseudopullback Theorem from~\cite[Section 2.2]{RR} applied to the pseudopullback of categories
$$
\xymatrix{
\sZ^0(B^\cu\bModl_\bflat) \ar[r] \ar[d] & \sZ^0(B^\cu\bModl) \ar[d] \\
B^*\sModl_\flat \ar[r] & B^*\sModl.
}
$$
Here, the horizontal functors are the obvious full embeddings, while the 
vertical ones are the forgetful functors of taking the underlying graded 
$B^*$\+module. All the functors preserve direct limits and in each of the three 
categories $\sZ^0(B^\cu\bModl)\simeq B[\delta]\sModl$, $B^*\sModl_\flat$,
and $B^*\sModl$ every  object is an $\aleph_1$\+direct limit of countably
presented ones (using  variants of 
Lemma~\ref{direct-limit=aleph-1-direct-of-countables}). The 
Pseudopullback Theorem implies that then also every object of $
\sZ^0(B^\cu\bModl_\bflat)$ is an $\aleph_1$\+direct limit of countably 
presented ones.
 For a historical account and a considerable generalization of
the Pseudopullback Theorem, see also~\cite[Introduction and
Corollary~5.1]{Plimacc}.
\end{rem}

\subsection{Resolution of a countably presented graded-flat CDG-module}
 The aim of this section is to prove the following lemma.

\begin{lem} \label{resolution-of-countably-presented-graded-flat}
 Let $B^\cu=(B^*,d,h)$ be a CDG\+ring. \par
\textup{(a)} For any CDG\+module $G^\cu=(G^*,d_G)$ over $B^\cu$
whose underlying graded $B^*$\+module $G^*$ is flat of projective
dimension~$\le1$, there is a short exact sequence in the abelian
category\/ $\sZ^0(B^\cu\bModl)$
$$
 0\lrarrow Q^\cu\lrarrow P^\cu\lrarrow G^\cu\lrarrow0,
$$
depending functorially on $G^\cu$, with a graded-projective CDG\+module
$P^\cu$ and a projective CDG\+module $Q^\cu$ over~$B^\cu$.
 In particular, this applies to all countably presented
graded-flat CDG\+modules $G^\cu$ over~$B^\cu$. \par
\textup{(b)} For any countably presented graded-flat CDG\+module
$G^\cu=(G^*,d_G)$ over $B^\cu$, there is a (nonfunctorial) short
exact sequence in the abelian category\/ $\sZ^0(B^\cu\bModl)$
$$
 0\lrarrow Q^\cu\lrarrow P^\cu\lrarrow G^\cu\lrarrow0
$$
with a countably generated (or equivalently, countably presented)
graded-projective CDG\+module
$P^\cu$ and a countably generated projective CDG\+module $Q^\cu$
over~$B^\cu$.
\end{lem}

\begin{proof}
 Any countably presented flat graded module over a graded ring $B^*$
has projective dimension~$\le1$ in $B^*\sModl$ by the graded version
of~\cite[Corollary~2.23]{GT}.
 To prove part~(a), we start with some functorial construction of
an epimorphism $F^\cu\rarrow G^\cu$ onto an arbitrary graded
$B^*[\delta]$\+module $G^\cu$ from a projective graded
$B^*[\delta]$\+module~$F^\cu$.
 For example, one can take $F^\cu$ to be the free graded
$B^*[\delta]$\+module spanned by the set of all homogeneous
elements of~$G^\cu$.
 In the case of part~(b), we similarly start with choosing
an epimorphism $F^\cu\rarrow G^\cu$ onto the countably presented
graded $B^*[\delta]$\+module $G^\cu$ from a countably generated
projective graded $B^*[\delta]$\+module~$F^\cu$.

 In both cases, we have a short exact sequence of CDG\+modules and
closed morphisms
$$
 0\lrarrow H^\cu\overset h\lrarrow F^\cu\lrarrow G^\cu\lrarrow0
$$
with a graded-projective CDG\+module $H^\cu$ and a projective
CDG\+module $F^\cu$ over~$B^\cu$.
 In the case~(b), the CDG\+module $H^\cu$ is countably generated,
hence it is actually countably presented by
Corollary~\ref{graded-projective-generated-presented-cor}(b).
 It remains to apply the construction of the cone of a closed morphism
in the DG\+category of CDG\+modules in order to pass to the short
exact sequence
$$
 0\lrarrow\cone(\id_{H^\cu})\lrarrow\cone(h)\lrarrow G^\cu\lrarrow0
$$
and set $Q^\cu=\cone(\id_{H^\cu})$ and $P^\cu=\cone(h)\in
\sZ^0(B^\cu\bModl)$.
 Then $Q^\cu$ is a contractible graded-projective (hence projective, as
per the discussion in Section~\ref{prelim-proj-inj-flatness-subsecn})
CDG\+module and $P^\cu$ is graded-projective CDG\+module over~$B^\cu$.
\end{proof}

\subsection{Imperfect resolution of a graded-flat CDG-module}
 The aim of this section is to prove the following corollary.

\begin{cor} \label{imperfect-resolution-of-graded-flat}
 Let $B^\cu=(B^*,d,h)$ be a CDG\+ring and $F^\cu$ be a graded-flat
CDG\+mod\-ule over~$B^\cu$.
 Then there exists a short exact sequence in the abelian category
of CDG\+modules\/ $\sZ^0(B^\cu\bModl)$
$$
 0\lrarrow H^\cu\lrarrow L^\cu\lrarrow F^\cu\lrarrow0
$$
where $H^\cu$ is a flat CDG\+module and $L^\cu$ is
an\/ $\aleph_1$\+direct limit of graded-projective CDG\+modules
over~$B^\cu$.
\end{cor}

\begin{proof}
 By Proposition~\ref{graded-flats-as-aleph-1-direct-limits},
any graded-flat CDG\+module $F^\cu$ is an $\aleph_1$\+direct limit
of countably presented graded-flat CDG\+modules.
 So we have $F^\cu=\varinjlim_{\xi\in\Xi}F_\xi^\cu$, where
$\Xi$ is an $\aleph_1$\+directed poset and $F_\xi^\cu$ are
countably presented graded-flat CDG\+modules over~$B^\cu$.
 By Lemma~\ref{resolution-of-countably-presented-graded-flat}(a),
the direct system $(F_\xi^\cu)_{\xi\in\Xi}$ can be extended to
a direct system of short exact sequences $0\rarrow Q_\xi^\cu
\rarrow P_\xi^\cu\rarrow F_\xi^\cu\rarrow0$ in $\sZ^0(B^\cu\bModl)$
with projective CDG\+modules $Q_\xi^\cu$ and graded-projective
CDG\+modules $P_\xi^\cu$ over~$B^\cu$.
 It remains to consider the induced short exact sequence of direct
limits
$$
 0\lrarrow\varinjlim\nolimits_{\xi\in\Xi}Q_\xi^\cu\lrarrow
 \varinjlim\nolimits_{\xi\in\Xi}P_\xi^\cu\lrarrow
 \varinjlim\nolimits_{\xi\in\Xi}F_\xi^\cu=F^\cu\lrarrow0
$$
and put $H^\cu=\varinjlim_{\xi\in\Xi}Q_\xi^\cu$ and
$L^\cu=\varinjlim_{\xi\in\Xi}P_\xi^\cu$.
\end{proof}

%

\subsection{\texorpdfstring{Contraacyclic $\aleph_1$-direct limits of graded-projective CDG-modules}{Contraacyclic aleph1-direct limits of graded-projective CDG-modules}}
 The aim of this section is to prove the following lemma.

\begin{lem} \label{aleph-1-graded-flat-contraacyclic}
 Let $B^\cu=(B^*,d,h)$ be a CDG\+ring and $L^\cu$ be a CDG\+module that
can be obtained as an\/ $\aleph_1$\+direct limit of graded-projective
CDG\+modules over~$B^\cu$.
 Assume that the CDG\+module $L^\cu$ over $B^\cu$ is contraacyclic,
i.~e., any closed morphism into $L^\cu$ from a graded-projective
CDG\+module is homotopic to zero.
 Then $L^\cu$ is a flat CDG\+module over~$B^\cu$.
\end{lem}

\begin{proof}
 In view of~\cite[Proposition~6.1]{Pgen}, it suffices to show that
any closed morphism into $L^\cu$ from a countably presented
CDG\+module $T^\cu$ factorizes through a countably generated projective
CDG\+module in $\sZ^0(B^\cu\bModl)$.
 Indeed, by the $\aleph_1$\+direct limit assumption on $L^\cu$,
the closed morphism $T^\cu\rarrow L^\cu$ factorizes through
a graded-projective CDG\+module $P^\cu$ over~$B^\cu$.
 By Lemma~\ref{graded-projective-direct-summand-of-graded-free},
$P^\cu$ is a direct summand of a graded-free CDG\+module $F^\cu$
in $\sZ^0(B^\cu\bModl)$; so the morphism $T^\cu\rarrow L^\cu$
factorizes as $T^\cu\rarrow F^\cu\rarrow L^\cu$.
 Finally, following the proof of
Lemma~\ref{deconstruction-of-graded-free}, one can pick a countably
generated graded-free CDG\+submodule $Q^\cu\subset F^\cu$ such that
the image of the morphism $T^\cu\rarrow F^\cu$ is contained in~$Q^\cu$.

 Alternatively, if one prefers, one could refer
to the assertion of Lemma~\ref{deconstruction-of-graded-free} together
with (the graded version of) the Hill lemma~\cite[Theorem~7.10]{GT}
instead of the concrete construction of graded-free CDG\+submodules
in the proof of Lemma~\ref{deconstruction-of-graded-free}.
 It is helpful to keep in mind that any countably generated
graded-projective CDG\+module is countably presented by
Corollary~\ref{graded-projective-generated-presented-cor}(b).

 We have factorized our closed morphism $T^\cu\rarrow L^\cu$ as
$T^\cu\rarrow Q^\cu\rarrow L^\cu$, where $Q^\cu$ is a countably
presented graded-projective CDG\+module.
 By the contraacyclicity assumption on $L^\cu$, the closed morphism
$Q^\cu\rarrow L^\cu$ is homotopic to zero.
 So it factorizes as $Q^\cu\rarrow\cone(\id_{Q^\cu})\rarrow L^\cu$,
where $\cone(\id_{Q^\cu})\rarrow L^\cu$ is some closed morphism
of CDG\+modules.
 Now the CDG\+module $\cone(\id_{Q^\cu})$ is graded-projective and
contractible; so it is a projective CDG\+module as explained in
Section~\ref{prelim-proj-inj-flatness-subsecn}.
 The CDG\+module $\cone(\id_{Q^\cu})$ is also countably
generated/presented, since so is the CDG\+module~$Q^\cu$.
\end{proof}

\subsection{Projective and flat contraderived categories identified}
 Now we can prove the main results of
Section~\ref{projective-and-flat-contraderived-secn}.
 The following corollary provides the description of graded-flat
contraacyclic CDG\+modules promised in
Section~\ref{restricted-contraderived-cotorsion-pair-subsecn}.

\begin{cor} \label{flat=graded-flat-and-contraacyclic-cor}
 A CDG\+module is flat if and only if it is graded-flat and
contraacyclic.
\end{cor}

\begin{proof}
 It was explained in Section~\ref{prelim-proj-inj-flatness-subsecn}
that all flat CDG\+modules are graded-flat.
 By Corollary~\ref{flats-are-contraacyclic-cor}, all flat
CDG\+modules are contraacyclic.

 Conversely, let $F^\cu$ be a graded-flat contraacyclic CDG\+module
over a CDG\+ring~$B^\cu$.
 By Corollary~\ref{imperfect-resolution-of-graded-flat}, there
exists a short exact sequence of CDG\+modules $0\rarrow H^\cu
\rarrow L^\cu\rarrow F^\cu\rarrow0$ such that $H^\cu$ is a flat
CDG\+module and $L^\cu$ is an $\aleph_1$\+direct limit of
graded-projective CDG\+modules over~$B^\cu$.

 Now the CDG\+module $H^\cu$ is contraacyclic (by
Corollary~\ref{flats-are-contraacyclic-cor}) and the CDG\+module
$F^\cu$ is contraacyclic (by assumption).
 Since the class of contraacyclic CDG\+modules is closed under
extensions (see Section~\ref{contraderived-subsecn}), it follows
that $L^\cu$ is also a contraacyclic CDG\+module.
 Using Lemma~\ref{aleph-1-graded-flat-contraacyclic}, we conclude
that $L^\cu$ is a flat CDG\+module over~$B^\cu$.

 Finally, both the CDG\+modules $H^\cu$ and $L^\cu$ are flat, while
the CDG\+module $F^\cu$ is graded-flat.
 Therefore, Lemma~\ref{two-flat-out-of-three-graded-flats} tells us
that the CDG\+module $F^\cu$ is also flat.
\end{proof}

 The contraderived category $\sD^\bctr(B^\cu\bModl)$ of left
CDG\+modules over $B^\cu$, as defined
in Section~\ref{contraderived-subsecn}, is identified
by Theorem~\ref{contraderived-cotorsion-pair-category-theorem}(b)
with the projective contraderived category from
Section~\ref{three-contraderived-posing-the-problem-subsecn}.
 We remind the reader that we wish to compare this
category with the triangulated Verdier quotient category
of the homotopy category of graded-flat CDG\+modules by the homotopy
category of flat CDG\+modules,
$$
 \sH^0(B^\cu\bModl_\bflat)/\sH^0(B^\cu\bModl)_\flat,
$$
as in item~(ii) of
Section~\ref{three-contraderived-posing-the-problem-subsecn}.
We called the quotient the flat contraderived category
of~$B^\cu$.

\begin{thm} \label{projective-and-flat-contraderived-theorem}
 Let $B^\cu=(B^*,d,h)$ be a CDG\+ring.
 Then the composition of the fully faithful inclusion of triangulated
categories\/ $\sH^0(B^\cu\bModl_\bproj)\rarrow\sH^0(B^\cu\bModl_\bflat)$
and the triangulated Verdier quotient functor\/
$\sH^0(B^\cu\bModl_\bflat)\rarrow
\sH^0(B^\cu\bModl_\bflat)/\sH^0(B^\cu\bModl)_\flat$ is a triangulated
equivalence
$$
 \sH^0(B^\cu\bModl_\bproj)\simeq
 \sH^0(B^\cu\bModl_\bflat)/\sH^0(B^\cu\bModl)_\flat.
$$
 In other words, the projective and flat contraderived categories of
CDG\+modules over any CDG\+ring are naturally equivalent.
\end{thm}

\begin{proof}
 We will show that the full subcategories $\sH^0(B^\cu\bModl)_\flat$
and $\sH^0(B^\cu\bModl_\bproj)$ form a semiorthogonal decomposition
of the triangulated category $\sH^0(B^\cu\bModl_\bflat)$.
 Indeed, Corollary~\ref{flats-are-contraacyclic-cor} provides
the semiorthogonality claim: for any graded-projective left 
CDG\+module $P^\cu$ and any flat left CDG\+module $F^\cu$ over $B^\cu$,
all closed morphisms of CDG\+modules $P^\cu\rarrow F^\cu$ are
homotopic to zero.
 It remains to construct, for any graded-flat left CDG\+module $G^\cu$
over $B^\cu$, a distinguished triangle $P^\cu\rarrow G^\cu\rarrow F^\cu
\rarrow P^\cu[1]$ in the homotopy category $\sH^0(B^\cu\bModl)$ with
a graded-projective CDG\+module $P^\cu$ and a flat CDG\+module~$F^\cu$.

 For this purpose, we recall from the discussion in
Section~\ref{restricted-contraderived-cotorsion-pair-subsecn} that the pair
$\big(\sZ^0(B^\cu\bModl_\bproj)$, $\sZ^0(B^\cu\bModl_\bflat)_\ac^\bctr\big)$
of the classes of graded-projective CDG\+mo\-du\-les and of
graded-flat contraacyclic CDG\+modules
is a hereditary complete cotorsion pair in the exact category of
graded-flat CDG\+mod\-ules $\sZ^0(B^\cu\bModl_\bflat)$.
 By Corollary~\ref{flat=graded-flat-and-contraacyclic-cor}, the class
of graded-flat contraacyclic CDG\+modules coincides with the class of
flat CDG\+modules, so the cotorsion pair in fact takes the form
$\big(\sZ^0(B^\cu\bModl_\bproj)$, $\sZ^0(B^\cu\bModl)_\flat\big)$.

 Given a graded-flat CDG\+module $G^\cu$, consider a special preenvelope
exact sequence $0\rarrow G^\cu\rarrow F^\cu\rarrow Q^\cu\rarrow0$
\,\eqref{special-preenvelope-sequence} in $\sZ^0(B^\cu\bModl_\bflat)$
with a flat CDG\+module $F^\cu$ and a graded-projective
CDG\+module~$Q^\cu$.
 Then the underlying short exact sequence of graded $B^*$\+modules
$0\rarrow G^*\rarrow F^*\rarrow Q^*\rarrow0$ splits (as $Q^*$ is
a projective graded $B^*$\+module), so we have a distinguished triangle
$G^\cu\rarrow F^\cu\rarrow Q^\cu\rarrow G^\cu[1]$ in the homotopy
category (cf.\ the discussion in the proof of
Lemma~\ref{Ext-1-homotopy-hom-lemma}).

 It remains to put $P^\cu=Q^\cu[-1]$ and rotate the distinguished
triangle $G^\cu\rarrow F^\cu\rarrow Q^\cu\rarrow G^\cu[1]$ that
we have constructed in order to obtain the desired distinguished
triangle $P^\cu\rarrow G^\cu\rarrow F^\cu\rarrow P^\cu[1]$.
\end{proof}

\subsection{The coderived category and coderived cotorsion pair}
\label{coderived-subsecn}
 Let $B^\cu=(B^*,d,h)$ be a CDG\+ring.
 A right CDG\+module $Y^\cu$ over $B^\cu$ is said to be \emph{coacyclic}
(\emph{in the sense of Becker}~\cite[Proposition~1.3.8(2)]{Bec}) if,
for any graded-injective right CDG\+module $J^\cu$ over $B^\cu$,
any closed morphism of CDG\+modules $Y^\cu\rarrow J^\cu$ is homotopic
to zero.
 Equivalently (in view of Lemma~\ref{Ext-1-homotopy-hom-lemma}),
this means that $\Ext^1_{\sZ^0(\bModr B^\cu)}(Y^\cu,J^\cu)=0$ for all
graded-injective CDG\+modules~$J^\cu$.

 The full subcategory of coacyclic CDG\+modules is denoted by
$\sH^0(\bModr B^\cu)_\ac^\bco\allowbreak\subset\sH^0(\bModr B^\cu)$ or
$\sZ^0(\bModr B^\cu)_\ac^\bco\subset\sZ^0(\bModr B^\cu)$.
 The class of coacyclic CDG\+modules is closed under extensions and
infinite direct sums in the abelian category $\sZ^0(\bModr B^\cu)$
\,\cite[Lemma~7.9]{PS5}, and under shifts, cones, and direct
summands in the triangulated category $\sH^0(\bModr B^\cu)$.
 Moreover, the full subcategory $\sZ^0(\bModr B^\cu)_\ac^\bco$ is
precisely the closure of the full subcategory of contractible
CDG\+modules under extensions and direct limits in
$\sZ^0(\bModr B^\cu)$ \,\cite[Corollary~7.17]{PS5}.

 The triangulated Verdier quotient category
$$
 \sD^\bco(\bModr B^\cu)=\sH^0(\bModr B^\cu)/\sH^0(\bModr B^\cu)_\ac^\bco
$$
is called the \emph{coderived category} (\emph{in the sense of Becker})
of right CDG\+modules over~$B^\cu$.
 The following theorem is due to
Becker~\cite[Proposition~1.3.6(2)]{Bec}.

\begin{thm} \label{coderived-cotorsion-pair-category-theorem}
 Let $B^\cu=(B^*,d,h)$ be a CDG\+ring.  Then \par
\textup{(a)} the pair of classes of objects\/
$\big(\sZ^0(\bModr B^\cu)_\ac^\bco$, $\sZ^0(\bModrinj B^\cu)\big)$ is
a hereditary complete cotorsion pair in the abelian category\/
$\sZ^0(\bModr B^\cu)$; \par
\textup{(b)} the composition of the fully faithful inclusion of
triangulated categories\/ $\sH^0(\bModrinj B^\cu)\rarrow
\sH^0(\bModr B^\cu)$ and the triangulated Verdier quotient functor\/
$\sH^0(\bModr B^\cu)\rarrow\sD^\bco(\bModr B^\cu)$ is a triangulated
equivalence
$$
 \sH^0(\bModrinj B^\cu)\simeq\sD^\bco(\bModr B^\cu).
$$
\end{thm}

\begin{proof}
 The assertions of part~(a) and~(b) are closely related; in fact,
part~(b) is a corollary of part~(a).
 Both the assertions are implied by the formulation in~\cite{Bec};
they can be also obtained by specializing the main result
of~\cite[Sections~7.1--7.4]{PS5} from Grothendieck
abelian DG\+categories to CDG\+modules.
 Part~(a) is a particular case of~\cite[Theorem~7.11]{PS5}, and
part~(b) is a particular case of~\cite[Corollary~7.10]{PS5}.
\end{proof}

 The hereditary complete cotorsion pair 
$\big(\sZ^0(\bModr B^\cu)_\ac^\bco$, $\sZ^0(\bModrinj B^\cu)\big)$
of the classes of coacyclic CDG\+modules and graded-injective
CDG\+modules provided by
Theorem~\ref{coderived-cotorsion-pair-category-theorem} is called
the \emph{coderived cotorsion pair} in the abelian category of
CDG\+modules $\sZ^0(\bModr B^\cu)$.

\subsection{Coderived category of graded-flat CDG-modules}
 Let $B^\cu$ be a CDG\+ring.
 By a \emph{graded-flat graded-cotorsion CDG\+module} over $B^\cu$
we mean a CDG\+module that is simultaneously graded-flat and
graded-cotorsion.
 The graded-flat graded-cotorsion CDG\+modules form a full
DG\+subcategory {\emergencystretch=1em \hfuzz=3pt
$$
 B^\cu\bModl_\bflat^\bcot=
 B^\cu\bModl_\bflat\cap B^\cu\bModl^\bcot
$$
closed under finite direct sums, shifts, twists, and cones in
the DG\+category of CDG\+modules $B^\cu\bModl$.
 Accordingly, the notation $\sZ^0(B^\cu\bModl_\bflat^\bcot)$ stands for
the full subcategory of graded-flat graded-cotorsion CDG\+modules in
the abelian category of CDG\+modules $\sZ^0(B^\cu\bModl)$, while
$\sH^0(B^\cu\bModl_\bflat^\bcot)$ is the full triangulated subcategory
of graded-flat graded-cotorsion CDG\+modules in the homotopy category
of CDG\+modules $\sH^0(B^\cu\bModl)$. }

 Notice that the flat cotorsion graded $B^*$\+modules are the injective
objects of the exact category of flat graded $B^*$\+modules (and there
are enough such injective objects in this exact category).
 For this reason, the homotopy category of graded-flat graded-cotorsion
CDG\+modules
$$
 \sH^0(B^\cu\bModl_\bflat^\bcot)
$$
can be thought of as the \emph{coderived category} (\emph{in the sense
of Becker}~\cite[Section~1.3]{Bec}) \emph{of the exact DG\+category of
graded-flat CDG\+modules} $B^\cu\bModl_\bflat$.
 From this point of view, the following theorem (together with
Theorem~\ref{projective-and-flat-contraderived-theorem}) tells us that,
for graded-flat CDG\+modules, the coderived and the contraderived
categories agree.

\begin{thm} \label{graded-flat-coderived-contraderived-theorem}
 Let $B^\cu=(B^*,d,h)$ be a CDG\+ring.
 Then the composition of the fully faithful inclusion of triangulated
categories\/ $\sH^0(B^\cu\bModl_\bflat^\bcot)\rarrow
\sH^0(B^\cu\bModl_\bflat)$ and the triangulated Verdier quotient
functor\/ $\sH^0(B^\cu\bModl_\bflat)\rarrow
\sH^0(B^\cu\bModl_\bflat)/\sH^0(B^\cu\bModl)_\flat$ is a triangulated
equivalence \hbadness=1325
$$
 \sH^0(B^\cu\bModl_\bflat^\bcot)\simeq
 \sH^0(B^\cu\bModl_\bflat)/\sH^0(B^\cu\bModl)_\flat.
$$
\end{thm}

\begin{proof}
 We will show that the full subcategories
$\sH^0(B^\cu\bModl_\bflat^\bcot)$ and $\sH^0(B^\cu\bModl)_\flat$ form
a semiorthogonal decomposition of the triangulated category
$\sH^0(B^\cu\bModl_\bflat)$.
 The argument is based on
Theorem~\ref{cotorsion=graded-cotorsion-theorem}, which tells us that
the classes of cotorsion and graded-cotorsion CDG\+modules coincide.

 To establish the semiorthogonality, let $F^\cu$ be a flat left
CDG\+module and $C^\cu$ be a graded-flat graded-cotorsion left
CDG\+module over~$B^\cu$.
 We need to show that any closed morphism of CDG\+modules $F^\cu
\rarrow C^\cu$ is homotopic to zero.
 In view of Lemma~\ref{Ext-1-homotopy-hom-lemma}, it suffices to
check that $\Ext^1_{\sZ^0(B^\cu\bModl)}(F^\cu[1],C^\cu)=0$.
 Indeed, $F^\cu[1]$ is a flat CDG\+module and $C^\cu$ is a cotorsion
CDG\+module over $B^\cu$, so the $\Ext^1$ group in question vanishes.

 It remains to construct, for any graded-flat CDG\+module $G^\cu$
over $B^\cu$, a distinguished triangle $F^\cu\rarrow G^\cu \rarrow
C^\cu\rarrow F^\cu[1]$ in the homotopy category $\sH^0(B^\cu\bModl)$
with a flat CDG\+module $F^\cu$ and a graded-flat graded-cotorsion
CDG\+module~$C^\cu$.
 With this purpose in mind, we start with another discussion of
a restricted cotorsion pair.

 According to Section~\ref{cotorsion-graded-modules-subsecn},
the classes of flat graded $B^*[\delta]$\+modules and
cotorsion graded $B^*[\delta]$\+modules form a hereditary complete
cotorsion pair in the abelian category $B^*[\delta]\sModl\simeq
\sZ^0(B^\cu\bModl)$.
 In other words, denoting by $\sL=\sZ^0(B^\cu\bModl)_\flat$ the class
of all flat left CDG\+modules and by $\sR=\sZ^0(B^\cu\bModl)^\cot$
the class of all cotorsion left CDG\+modules over $B^\cu$, we have
a hereditary complete cotorsion pair $(\sL,\sR)$ in the abelian
category $\sA=\sZ^0(B^\cu\bModl)$.
 Let $\sE\subset\sA$ be the full subcategory of graded-flat
CDG\+modules, $\sE=\sZ^0(B^\cu\bModl_\bflat)$.

 Then the assumptions of Lemma~\ref{restricted-exact-structure} are
satisfied for the hereditary complete cotorsion pair $(\sL,\sR)$ in
the abelian category $\sA$ and the full subcategory $\sE\subset\sA$.
 Therefore, the cotorsion pair $(\sL,\sR)$ restricts to a hereditary
complete cotorsion pair ($\sL$, $\sE\cap\sR$) in the exact
category~$\sE$.
 In other words, the pair
$\big(\sZ^0(B^\cu\bModl)_\flat$, $\sZ^0(B^\cu\bModl_\bflat)^\cot\big)$
of the classes of flat CDG\+modules and of graded-flat cotorsion
CDG\+modules is a hereditary complete
cotorsion pair in the exact category of graded-flat CDG\+modules
$\sZ^0(B^\cu\bModl_\bflat)$.
 By Lemma~\ref{cotorsion-are-graded-cotorsion}, all cotorsion
CDG\+modules are graded-cotorsion.

 Given a graded-flat CDG\+module $G^\cu$, consider a special precover
exact sequence $0\rarrow D^\cu\rarrow F^\cu\rarrow G^\cu\rarrow0$
\,\eqref{special-precover-sequence} in $\sZ^0(B^\cu\bModl_\bflat)$
with a flat CDG\+module $F^\cu$ and a graded-flat graded-cotorsion
CDG\+module~$D^\cu$.
 Then the underlying short exact sequence of graded $B^*$\+modules
$0\rarrow D^*\rarrow F^*\rarrow G^*\rarrow0$ splits (as
$\Ext^1_{B^*\sModl}(G^*,D^*)=0$), so we have a distinguished triangle
$D^\cu\rarrow F^\cu\rarrow G^\cu\rarrow D^\cu[1]$ in the homotopy
category $\sH^0(B^\cu\bModl)$ (cf.\ the proof of
Lemma~\ref{Ext-1-homotopy-hom-lemma}).

 It remains to put $C^\cu=D^\cu[1]$ and rotate the distinguished
triangle $D^\cu\rarrow F^\cu\rarrow G^\cu\rarrow D^\cu[1]$ that we
have constructed in order to obtain the desired distinguished triangle
$F^\cu\rarrow G^\cu\rarrow C^\cu\rarrow F^\cu[1]$.
\end{proof}

\subsection{Examples}
\label{examples-Dbctr-subsecn}

Now we will present examples of classes of CDG\+rings $B^\cu$ and their (Becker) coderived and contraderived categories.
We will continue most of these examples in Section~\ref{examples-compacts-subsecn}, where we will use results from Section~\ref{compact-generators-of-contraderived-secn} to describe compact generators of $\sD^\bctr(B^\cu\bModl)$.

\begin{ex}[Finite-dimensional CDG\+rings] \label{finite-dimensional-algebra-example}
 Let $k$ be a field and $B^*$ be a graded algebra over~$k$.
 Let $B^\cu=(B^*,d,h)$ be a CDG\+ring structure on $B^*$ with
a $k$\+linear differential~$d$.
 In this case, we will say that $B^\cu$ is a \emph{CDG\+algebra
over~$k$}.

 A graded $k$\+algebra $B^*$ is said to be (totally) finite-dimensional
if $B^i=0$ for all but a finite set of degrees~$i$ and $B^i$ is
a finite-dimensional $k$\+vector space for all degrees $i\in\boZ$.
 In this case, the CDG\+algebra $B^\cu$ is also said to be
finite-dimensional.

 For a finite-dimensional CDG\+algebra $B^\cu$, much of the theory
developed above in this paper becomes trivial.
 All flat graded $B^*$\+modules are projective, and accordingly, all
graded $B^*$\+modules are cotorsion; hence all graded-flat CDG\+modules
over $B^\cu$ are graded-projective, and all flat CDG\+modules over
$B^\cu$ are projective (hence contractible).
 So the projective and flat contraderived categories of CDG\+modules
over $B^\cu$ coincide by the definition.

 On the other hand, there are nontrivial results specific to this more
restrictive context.
 For a finite-dimensional graded $k$\+algebra $B^*$, the graded dual
$k$\+vector space $\cC^*=B^*{}\spcheck=\Hom_k^*(B^*,k)$
is naturally a finite-dimensional graded coalgebra over~$k$.
 For a finite-dimensional CDG\+algebra $B^\cu$, there is natural
structure of CDG\+coalgebra $\cC^\cu=(\cC^*,d,h)=B^\cu{}\spcheck$
on the graded coalgebra $\cC^*=B^*{}\spcheck$ (in the sense of
the definitions in~\cite[Section~4.1]{Pkoszul}
and~\cite[Section~6.3]{Pksurv}).

 More generally, for any CDG\+coalgebra $\cC^\cu=(\cC^*,d,h)$
over~$k$, the graded dual $k$\+vector space $B^*=\cC^*{}\spcheck
=\Hom_k^*(\cC^*,k)$ has a natural structure of CDG\+algebra
$B^\cu=(B^*,d,h)$.
 As mentioned in~\cite[Section~6.3]{Pksurv}, we prefer the convention
in which left $\cC^*$\+comodules become left $B^*$\+modules and
right $\cC^*$\+comodules become right $B^*$\+modules.
 Then left $\cC^*$\+contramodules (in the sense
of~\cite[Section~2.2]{Pkoszul} and~\cite[Sections~8\+-9]{Pksurv})
also become left $B^*$\+modules.

 For a finite-dimensional graded algebra $B^*$, these are
equivalences (in fact, isomorphisms) of categories: the graded
left $\cC^*$\+comodules are the same things as the graded left
$B^*$\+modules and as the graded left $\cC^*$\+contramodules,
while the graded right $\cC^*$\+comodules are the same things
as the graded right $B^*$\+modules.
 For a finite-dimensional CDG\+algebra $B^\cu$, there are isomorphisms
of DG\+categories $\cC^\cu\bComodl\simeq B^\cu\bModl\simeq
\cC^\cu\bContra$ and $\bComodr\cC^\cu\simeq\bModr B^\cu$.

 For any CDG\+coalgebra $\cC^\cu$, the \emph{derived
comodule-contramodule correspondence
theorem}~\cite[Theorem~5.2]{Pkoszul}, \cite[Theorem~9.5]{Pksurv}
claims a triangulated equivalence between the coderived category
of left CDG\+comodules and the contraderived category of left
CDG\+contramodules over~$\cC^\cu$,
\begin{equation} \label{coalgebra-co-contra-corresp}
 \boR\Hom^\cu_{\cC^*}(\cC^\cu,{-})\:
 \sD^\co(\cC^\cu\bComodl)\simeq\sD^\ctr(\cC^\cu\bContra)
 :\!\cC^\cu\ocn_{\cC^*}^\boL{-}.
\end{equation}
 Here $\sD^\co$ and $\sD^\ctr$ is the notation for the coderived
and the contraderived category in the sense of Positselski (as
defined in~\cite[Section~4.2]{Pkoszul}, \cite[Sections~7.7
and~9.3]{Pksurv}, \cite[Section~5.1]{Pedg}).
 The derived comodule-contramodule correspondence theorem was formally
stated in~\cite{Pkoszul} for derived categories of the second kind
in the sense of Positselski.
 However, by~\cite[Theorem~4.4]{Pkoszul}, \cite[Theorems~7.13
and~9.4]{Pksurv}, Positselski's and Becker's derived categories of
the second kind agree for CDG\+coalgebras over a field, so
$\sD^\co(\cC^\cu\bComodl)=\sD^\bco(\cC^\cu\bComodl)$ and
$\sD^\ctr(\cC^\cu\bContra)=\sD^\bctr(\cC^\cu\bContra)$.
 The notation $\boR\Hom^\cu_{\cC^*}$ and $\ocn_{\cC^*}^\boL$ stands for
the right derived functor of the graded left $\cC^*$\+comodule $\Hom$
and the left derived functor of the \emph{contratensor product}
functor~$\ocn_{\cC^*}$ (as defined in~\cite[Section~2.2]{Pkoszul}
and~\cite[Section~8.6]{Pksurv}).

 In particular, for a finite-dimensional CDG\+algebra $B^\cu$,
the triangulated equivalence~\eqref{coalgebra-co-contra-corresp}
specializes to the triangulated equivalence between the coderived
and contraderived categories of one and the same DG\+category of
left CDG\+modules,
\begin{equation} \label{fin-dim-algebra-co-contra-corresp}
 \boR\Hom^\cu_{B^*}(B^\cu{}\spcheck,{-})\:
 \sD^\bco(B^\cu\bModl)\simeq\sD^\bctr(B^\cu\bModl)
 :\!B^\cu{}\spcheck\ot_{B^*}^\boL{-}.
\end{equation}
 Let us emphasize that $B^\cu{}\spcheck$ is the graded dual
$k$\+vector space to the CDG\+bimodule $B^\cu$ over $B^\cu$ here;
so $B^\cu{}\spcheck$ is a left and right graded-injective
CDG\+bimodule over~$B^\cu$.
\end{ex}

\begin{ex}[The Covariant Serre--Grothendieck Duality] \label{covariant-Serre-Grothendieck-example}
As a preparation for the next example, we recall an equivalence between the Becker coderived and contraderived categories observed by Iyengar and Krause~\cite{IK}.
It is sometimes called the covariant Serre--Grothendieck duality as when one restricts it to the compact objects on both sides, one recovers the classical Serre--Grothendieck duality. This is one of the most important sources of inspiration for our Section~\ref{compact-generators-of-contraderived-secn}.

Let $R$ be a commutative Noetherian ring, which we can view as a CDG\+ring $B^\cu=(R,0,0)$ with trivial grading, differential and curvature.
Then
\begin{equation} \label{forms-of-Dbctr-for-rings}
\sD^\bctr(B^\cu\bModl)\simeq\sK(R\rModl_\proj)\simeq\sK(R\rModl_\flat^\cot)
\end{equation}
by Theorems~\ref{contraderived-cotorsion-pair-category-theorem}, \ref{projective-and-flat-contraderived-theorem} and~\ref{graded-flat-coderived-contraderived-theorem}. In fact, the second equivalence is known much more generally for any (possibly non-commutative) ring $R$ by~\cite[Corollary~5.8]{Sto2} applied to the category $R\rModl_\flat$ of flat left $R$\+modules. Likewise, we have an equivalence
\begin{equation} \label{forms-of-Dbco-for-rings}
\sD^\bco(B^\cu\bModl)\simeq\sK(R\rModl_\inj)
\end{equation}
by Theorem~\ref{coderived-cotorsion-pair-category-theorem}.
Suppose further that $R$ has a dualizing complex, i.~e.\ a bounded complex $D^\bu$ of injective $R$\+modules with all cohomologies finitely generated and such that the canonical map $R\to\Hom^\bu_R(D^\bu,D^\bu)$ is a quasi-isomorphism (see e.g.~\cite[Section~3]{IK}).
Then \cite[Theorem 4.2]{IK} in view of~\eqref{forms-of-Dbctr-for-rings} and~\eqref{forms-of-Dbco-for-rings} shows that $\sD^\bco(B^\cu\bModl)$ and $\sD^\bctr(B^\cu\bModl)$ are triangle equivalent. The recipe in~\cite{IK} is not entirely straightforward, but in view of the last equivalence in~\eqref{forms-of-Dbctr-for-rings}, we can present it in a very easy form:
\begin{equation} \label{dualizing-cpx-co-contra-corresp}
 \Hom^\bu_{R}(D^\bu,{-})\:
 \sK(R\rModl_\inj)\simeq\sK(R\rModl_\flat^\cot)
 :\!D^\bu\ot_{R}{-}.
\end{equation}
Notice the similarity with~\eqref{fin-dim-algebra-co-contra-corresp}. This is no coincidence as when $R$ is a finite-dimensional commutative algebra over a field $k$ (which is the intersection of the setups in Examples~\ref{finite-dimensional-algebra-example} and~\ref{covariant-Serre-Grothendieck-example}), then $R\spcheck=\Hom_R(R,k)$ is a dualizing complex.

 The equivalence~\eqref{dualizing-cpx-co-contra-corresp} is an immediate
consequence of the argument for~\cite[Theorem~2.4]{IK} and our
Theorems~\ref{projective-and-flat-contraderived-theorem}
and~\ref{graded-flat-coderived-contraderived-theorem}.
 The only extra thing to observe is that the functor from left to right
is well defined, or in other words that $\Hom_R(D^\bu,X^\bu)$ is
a complex of flat cotorsion $R$\+modules whenever $X^\bu$ is a complex
of injective modules.
 However, this follows from~\cite[Lemma~2.3]{enochs-flat-cot}.

We conclude this discussion by pointing out that~\eqref{dualizing-cpx-co-contra-corresp} is not the only significance of the flat cotorsion incarnation of the contraderived category. A structure theorem for flat cotorsion modules over commutative Noetherian rings, dual-analogous to the structure of injectives, is the main result of~\cite{enochs-flat-cot}, and this has found further applications, for instance in~\cite{BT-flat-cot-MF,Thomp-FCmin}, and also in~\cite[Chapter~6]{Pcosh}.
\end{ex}

\begin{ex}[Algebraic de~Rham complexes] \label{de-Rham-complex-example}
 Let $X$ be a smooth affine algebraic variety of dimension $n$ over
a field~$k$ of characteristic zero.
 Then the de~Rham complex of algebraic differential forms
$\Omega^\bu(X/k)$ is a graded commutative DG\+algebra over~$k$
\,\cite[Section~16.6]{EGAIV}, \cite[Section Tag~0FKF]{SP},
\cite[Construction~10.5]{Prel}.
 The de~Rham complex is \emph{not} a finite-dimensional DG\+algebra
over~$k$ (even though it is concentrated in a finite set of
cohomological degrees).
 So the discussion in Example~\ref{finite-dimensional-algebra-example}
is \emph{not} applicable to $\Omega^\bu(X/k)$.

 The classes of graded-projective and graded-flat DG\+modules over
$\Omega^\bu(X/k)$ do \emph{not} coincide, and \emph{not} every
DG\+module over $\Omega^\bu(X/k)$ is graded-cotorsion.
 For example, the free graded module $\Omega^*(X/k)$ over the graded
algebra $\Omega^*(X/k)$ is usually \emph{not} cotorsion, because
it is not even cotorsion as a graded module over the ring of functions
on~$X$ (cf.\ Lemma~\ref{restriction-of-scalars-preserves-cotorsion}).
 For example, in the case of the affine line $X=\mathbb A^1_k=
\Spec k[x]$, one can easily see that the ring of polynomials $k[x]$
in one variable~$x$ over a field~$k$ is \emph{not} a cotorsion module
over itself, while the flat $k[x]$\+module $k[x,x^{-1}]$ is not
projective.
 In fact, one has $\Ext^1_{k[x]}(k[x,x^{-1}],k[x])\ne0$.
 Furthermore, denoting by $Y=\Spec k[x,x^{-1}]$ the complement to
a closed point in $\mathbb A^1_k$, the DG\+module $\Omega^\bu(Y/k)$
over the DG\+algebra $\Omega^\bu(X/k)$ is graded-flat but \emph{not}
graded-projective.
 So the results of
Theorems~\ref{projective-and-flat-contraderived-theorem}
and~\ref{graded-flat-coderived-contraderived-theorem} are nontrivial
in this case already.

 Nevertheless, for any smooth affine algebraic variety $X$ over~$k$,
Positselski's and Becker's derived categories of the second kind
agree for the DG\+ring $\Omega^\bu(X/k)$.
 Indeed, the graded algebra of differential forms $\Omega^*(X/k)$ is
left and right Noetherian, so one has
$\sD^\co(\Omega^\bu(X/k)\bModl)\simeq\sD^\bco(\Omega^\bu(X/k)\bModl)$
by~\cite[Theorem~3.7]{Pkoszul}, \cite[Theorem~7.9(a)]{Pksurv}.
 Furthermore, the ring of functions $O(X)$ is commutative Noetherian
of Krull dimension $n$, while the graded ring $\Omega^*(X/k)$ is
a finitely generated projective graded $O(X)$\+module.
 By~\cite[Theorem~3.8]{Pkoszul}, \cite[Theorem~7.9(b)]{Pksurv}
(see also~\cite[Corollaire~II.3.2.7]{RG}, \cite[Proposition~4.3]{Pfp})
together with~\cite[beginning of the proof of Theorem~3.10]{Pkoszul},
one has $\sD^\ctr(\Omega^\bu(X/k)\bModl)\simeq
\sD^\bctr(\Omega^\bu(X/k)\bModl)$.

 Finally, even though formula~\eqref{fin-dim-algebra-co-contra-corresp}
is not applicable, the coderived and the contraderived category of
DG\+modules over $\Omega^\bu(X/k)$ are still naturally equivalent to
each other.
 Essentially, this holds because the graded ring $\Omega^*(X/k)$ is
Iwanaga--Gorenstein (in the sense of~\cite[Section~9.1]{EJ}).
 More in detail, the module of the highest degree forms $L:=\Omega^n(X/k)$ is invertible, i.~e.\ the module $L^{\ot-1}:=\Hom_{O(X)}(L,O(X))$ satisfies $L\ot_{O(X)}L^{\ot-1}\simeq O(X)$. Then
$$ \Omega^*(X/k)[n]\cong \Hom_{O(X)}(\Omega^*(X/k),L) $$
as a graded bimodule over itself.
 If $O(X)\to D^\bu$ is a minimal injective $O(X)$\+module resolution of 
$O(X)$ (so $D^\bu$ is finite of length~$n$), then
\begin{multline}\label{dualizing-cpx-for-deRham}
\Omega^*(X/k)\ot_{O(X)}D^\bu \simeq
\Hom_{O(X)}(\Omega^*(X/k),L)\ot_{O(X)}D^\bu[-n] \\
\simeq \Hom_{O(X)}(\Omega^*(X/k),L\ot_{O(X)}D^\bu)[-n],
\end{multline}
where the second isomorphism holds since $\Omega^*(X/k)$ is a finitely generated projective graded $O(X)$\+module. In particular, $\Omega^*(X/k)\ot_{O(X)}D^\bu$ is a finite injective resolution of $\Omega^*(X/k)$ by graded modules. Therefore, $\Omega^*(X/k)$ is also Gorenstein
 in the sense of~\cite[Section~3.9]{Pkoszul}, so the result
of~\cite[Theorem~3.9]{Pkoszul} applies.

 In fact, there is again (in view of
Theorems~\ref{contraderived-cotorsion-pair-category-theorem}, \ref{projective-and-flat-contraderived-theorem},
\ref{coderived-cotorsion-pair-category-theorem}
and~\ref{graded-flat-coderived-contraderived-theorem}) a more direct
version of this equivalence somewhat similar 
to~\eqref{dualizing-cpx-co-contra-corresp}.
 Consider the abelian DG\+category of DG\+bimodules over
$\Omega^\bu(X/k)$, where we presume that left and right actions of
the field~$k$ in our bimodules $E^\bu$ agree (while the left
and right actions of $O(X)$ can be completely different).
 Let us take a partial graded-injective resolution of
$\Omega^\bu(X/k)$ in the abelian category of DG\+bimodules:
$$
 0\lrarrow \Omega^\bu(X/k)\lrarrow E^{0,\bu}\lrarrow E^{1,\bu}
 \lrarrow\dotsb\lrarrow E^{n-1,\bu}\lrarrow E^{n,\bu}\lrarrow0.
$$
 Here the graded $\Omega^*(X/k)$\+$\Omega^*(X/k)$\+bimodules $E^{i,*}$
are injective objects in the category of graded bimodules for all
$0\le i\le n-1$, but $E^{n,\bu}$ is just the $n$\+th cosyzygy.
 Now, since $\Omega^*(X/k)$ is a flat graded $k$\+module, it follows 
that $E^{i,*}$ are also injective one-sided graded
$\Omega^*(X/k)$\+modules from either side.
 Since the graded ring $\Omega^*(X/k)$ is $n$\+Iwanaga--Gorenstein,
the one-sided graded $\Omega^*(X/k)$\+module $E^{n,*}$ is also
injective from either side.

 Now let $D^\bu=\Tot(E^{0,\bu}\to E^{1,\bu}\to\dotsb\to E^{n,\bu})$
be the totalization of the complex of DG\+bimodules $E^{\bu,\bu}$.
 Then the cone of the natural morphism of DG\+bimodules
$\Omega^\bu(X/k)\rarrow D^\bu$ is absolutely acyclic as
a DG\+bimodule over $\Omega^\bu(X/k)$ (see
Section~\ref{compact-generators-of-coderived-subsecn} below for
the definition).
 Then, similarly to the discussion in
Example~\ref{covariant-Serre-Grothendieck-example}, we have a pair
of adjoint functors between the homotopy categories of graded-injective
and graded-flat graded-cotorsion DG\+modules
\begin{equation}\label{dualizing-cpx-co-contra-deRham}
 \Hom^\bu_{\Omega^*(X/k)}(D^\bu,{-})\:
 \sH^0(\Omega^\bu(X/k)\bModl_\binj)\leftrightarrows
 \sH^0(\Omega^\bu(X/k)\bModl_\bflat^\bcot)
 :\!D^\bu\ot_{\Omega^*(X/k)}{-}.
\end{equation}

 Let us show that the functors~\eqref{dualizing-cpx-co-contra-deRham}
are mutually inverse equivalences.
 First of all, it suffices to construct, for every graded-injective
DG\+module $J^\bu$ over $\Omega^\bu(X/k)$, \emph{some} functorial
isomorphism between $J^\bu$ and $D^\bu\ot_{\Omega^*(X/k)}
\Hom_{\Omega^*(X/k)}(D^\bu,J^\bu)$ in
$\sH^0(\Omega^\bu(X/k)\bModl_\binj)$ and, for every graded-flat
graded-cotorsion DG\+module $F^\bu$ over $\Omega^\bu(X/k)$, \emph{some}
functorial isomorphism between $F^\bu$ and
$\Hom_{\Omega^*(X/k)}(D^\bu,\>D^\bu\ot_{\Omega^*(X/k)}F^\bu)$ in
$\sH^0(\Omega^\bu(X/k)\bModl_\bflat^\bcot)$.
 Then it would follow that both the adjoint functors
in~\eqref{dualizing-cpx-co-contra-deRham} are category equivalences,
and therefore, their adjunction morphisms are isomorphisms in
the homotopy categories.

 Furthermore, it follows from
Theorem~\ref{coderived-cotorsion-pair-category-theorem} that
any isomorphism in the coderived category
$\sD^\bco(\Omega^\bu(X/k)\bModl)$ between two objects from
the homotopy category $\sH^0(\Omega^\bu(X/k)\bModl_\binj)$ comes
from a unique isomorphism in $\sH^0(\Omega^\bu(X/k)\bModl_\binj)$.
 Similarly, it follows from
Theorems~\ref{contraderived-cotorsion-pair-category-theorem}, \ref{projective-and-flat-contraderived-theorem},
and~\ref{graded-flat-coderived-contraderived-theorem} that any
isomorphism in the contraderived category
$\sD^\bctr(\Omega^\bu(X/k)\bModl)$ between two objects from the homotopy
category $\sH^0(\Omega^\bu(X/k)\bModl_\bflat^\bcot)$ comes from
a unique isomorphism in $\sH^0(\Omega^\bu(X/k)\bModl_\bflat^\bcot)$.
 Furthermore, any isomorphism in the absolute derived category
$\sD^\abs(\Omega^\bu(X/k)\bModl)$ is an isomorphism in the coderived
category as well as in the contraderived category.

 It remains to notice that the DG\+bimodule morphism
$\Omega^\bu(X/k)\rarrow D^\bu$ with an absolutely acyclic cone induces
DG\+module morphisms with absolutely acyclic cones
\begin{multline*}
 D^\bu\ot_{\Omega^*(X/k)}\Hom_{\Omega^*(X/k)}(D^\bu,J^\bu)\llarrow
 \Omega^\bu(X/k)\ot_{\Omega^*(X/k)}\Hom_{\Omega^*(X/k)}(D^\bu,J^\bu)
 \\ \lrarrow
 \Omega^\bu(X/k)\ot_{\Omega^*(X/k)}
 \Hom_{\Omega^*(X/k)}(\Omega^\bu(X/k),J^\bu)=J^\bu,
\end{multline*}
where we are using the facts that the DG\+module
$\Hom_{\Omega^*(X/k)}(D^\bu,J^\bu)$ is graded-flat and the DG\+module
$J^\bu$ is graded-injective.
 Similarly, the same DG\+bimodule morphism with an absolutely acyclic
cone induces DG\+module morphisms with absolutely acyclic cones
\begin{multline*}
 \Hom_{\Omega^*(X/k)}(D^\bu,\>D^\bu\ot_{\Omega^*(X/k)}F^\bu)\lrarrow
 \Hom_{\Omega^*(X/k)}(\Omega^\bu(X/k),\>D^\bu\ot_{\Omega^*(X/k)}F^\bu)
 \\ \llarrow
 \Hom_{\Omega^*(X/k)}(\Omega^\bu(X/k),\>
 \Omega^\bu(X/k)\ot_{\Omega^*(X/k)}F^\bu)=F^\bu,
\end{multline*}
where we are using the facts that the DG\+module
$D^\bu\ot_{\Omega^*(X/k)}F^\bu$ is graded-injective and
the DG\+module $F^\bu$ is graded-flat.
 Notice that both
$\Hom_{\Omega^*(X/k)}(D^\bu,\>D^\bu\ot_{\Omega^*(X/k)}F^\bu)$ and
$F^\bu$ are graded-flat graded-cotorsion DG\+modules, while
the intermediate DG\+module
$\Hom_{\Omega^*(X/k)}(\Omega^\bu(X/k),\>D^\bu\ot_{\Omega^*(X/k)}F^\bu)$
in the latter formula is graded-injective.
 Nevertheless, the ``roof'' (fraction) formed by the two morphisms
defines a natural isomorphism in the absolute derived category
$\sD^\abs(\Omega^\bu(X/k)\bModl)$, which implies a natural isomorphism
in the homotopy category $\sH^0(\Omega^\bu(X/k)\bModl_\bflat^\bcot)$
as explained in the previous paragraph.


 We conclude by pointing out that there are more equivalences connected to the algebraic de~Rham complex. Most notably, there are natural triangulated equivalences of
$D$\+$\Omega$ duality
\begin{equation} \label{D-Omega-duality}
 \sD^\bco(\Omega^\bu(X/k)\bModl)\simeq
 \sD(D(X/k)\rModl)\simeq
 \sD^\bctr(\Omega^\bu(X/k)\bModl),
\end{equation}
where $\sD(D(X/k)\rModl)$ denotes the derived category of the abelian
category of modules over the ring $D(X/k)$ of $k$\+linear algebraic
differential operators $O(X)\rarrow O(X)$.
 See~\cite[Theorems~B.2 and~B.3]{Pkoszul} or~\cite[Theorem~10.7]{Prel}.
 (The definition of the ring of algebraic differential operators can be
found in~\cite[Section~16.8]{EGAIV}, \cite[Section Tag~09CH]{SP},
or~\cite[Construction~10.2 or~10.8]{Prel}.)
\end{ex}

\begin{ex}[Matrix factorizations] \label{matrix-factorizations-example}
 We will now discuss the class of examples advertized
in Section~\ref{introd-matrix-factorizations} of the Introduction.
We present a non-commutative version which encompasses examples such as~\cite[\S2.3]{DK} or~\cite[\S2.4]{IT} here, as this generalization comes only at a little additional cost. The classical setting (originating in~\cite{Eis}) restricts to commutative rings and invertible modules rather than bimodules.

 Let $R$ be a ring and $L$ be an invertible $R$\+$R$\+bimodule,
i.~e., an $R$\+$R$\+bimodule for which there is an $R$\+$R$\+bimodule
$L^{\ot-1}$ such that there are $R$\+$R$\+bimodule isomorphisms
$L\ot_R L^{\ot-1}\simeq R\simeq L^{\ot-1}\ot_R L$.
 We presume such isomorphisms to have been chosen so that they are
compatible with each other in the sense that the two induced
isomorphisms $L\simeq L\ot_R L^{\ot-1}\ot_R L\simeq L$ coincide.
 Since then $L\ot_R-$ and $-\ot_RL$ induce autoequivalences of $R\rModl$ and $\rModr{R}$, respectively, $L$ must be finitely generated projective both as a left and a right module, and so must be $L^{\ot-1}$ by the classical Morita theory.
 Then the $R$\+$R$\+bimodule $L^{\ot n}$ is well-defined for all $n\in\boZ$.
 Consider the graded ring $B^*$ with the grading components
$B^{2n}=L^{\ot n}$ and $B^{2n+1}=0$ for all $n\in\boZ$.
 In particular, the degree-zero subring in $B$ is $B^0=R$.
 The multiplication in $B^*$ is given by the obvious maps
$B^{2n}\times B^{2m}\rarrow B^{2n}\ot_R B^{2m}\simeq B^{2n+2m}$,
\ $n$, $m\in\boZ$.

 Let $w\in L$ be a \emph{central} element, meaning that $wr=rw\in L$ for all $r\in R$ and $w\ot l=l\ot w\in L^{\ot2}$ for all $l\in L$. This is always satisfied when $R$ is commutative and $L$ is an $R$-module (i.~e.\ the left and the right actions of $R$ coincide), or when $L=R$ and $w$ is from the centre of the ring $R$.
The element $w$ is called the \emph{potential}.

 Then the graded ring $B^*$ with the zero differential $d=0$ and
the curvature element $h=w\in L=B^2$ is a CDG\+ring $B^\cu=(B^*,0,w)$ 
(because $h$~is a central element in~$B^*$).
 Left CDG\+modules $M^\cu=(M^*,d_M)$ over $B^\cu$ are called
(\emph{$R$\+module}) \emph{matrix factorizations} of
the potential $w\in L$.
 Explicitly, an $R$\+module matrix factorization $M^\cu$ is given
by a pair of left $R$\+modules $M^0$ and $M^1$ together with two
$R$\+linear maps $d_{M,0}\:M^0\rarrow M^1$ and $d_{M,1}\:M^1\rarrow
L\ot_R M^0=M^2$ such that both the compositions $M^0\rarrow M^1
\rarrow L\ot_R M^0$ and $M^1\rarrow L\ot_RM^0\rarrow L\ot_RM^1$
are equal to the map of multiplication with~$w$.
 All the other grading components of $M^\cu$ are then recovered
as $M^{2n}=L^{\ot n}\ot_R M^0$ and $M^{2n+1}=L^{\ot n}\ot_R M^1$;
and the maps $d_{M,2n}\:M^{2n}\rarrow M^{2n+1}$ and
$d_{M,2n+1}\:M^{2n+1}\rarrow M^{2n+2}$ are obtained by applying
the functor $L^{\ot n}\ot_R{-}$ to the maps $d_{M,0}$ and~$d_{M,1}$,
respectively.

 Following our terminology in this paper, a matrix factorization
$P^\cu$ is said to be \emph{graded-projective} if the $R$\+modules
$P^0$ and $P^1$ are projective.
 Similarly, a matrix factorization $F^\cu$ is said to be
\emph{graded-flat} if the $R$\+modules $F^0$ and $F^1$ are flat.
 (In the more traditional terminology of matrix factorization theory,
as in~\cite{EP}, these would be called simply ``projective''
or ``flat'' matrix factorizations.)
 A matrix factorization $C^\cu$ is said to be \emph{graded-cotorsion}
if the $R$\+modules $C^0$ and $C^1$ are cotorsion,
and a matrix factorization $E^\cu$ is said to be \emph{graded-injective} if $E^0$ and $E^1$ are injective.

 Now Theorem~\ref{projective-and-flat-contraderived-theorem} tells us
that the homotopy category of graded-projective matrix
factorizations of a potential $w\in L$ is equivalent to the
triangulated quotient category of the homotopy category of graded-flat
matrix factorizations by the thick subcategory of flat matrix
factorizations. Furthermore, by
Theorem~\ref{graded-flat-coderived-contraderived-theorem}, the same
triangulated category is equivalent to the homotopy category of
graded-flat graded-cotorsion matrix factorizations.

 Assume that the ring $R$ is left and right Noetherian and $D^\bu$ is
a dualizing complex in the sense of~\cite[\S3.3.1]{IK} (i.~e.\ $D^\bu$
is a bounded complex of $R$-$R$-bimodules which is a complex of
injective modules with finitely generated cohomologies from either
side and both $R\to\Hom_{R}(D^\bu,D^\bu)$ and
$R\to\Hom_{R^\rop}(D^\bu,D^\bu)$ are quasi-isomorphisms).
 Assume further that an isomorphism of complexes of $R$\+$R$\+bimodules
$L\ot_RD^\bu\simeq D^\bu\ot_RL$ has been chosen, and it forms
a commutative triangular diagram with the maps $w\ot{-}\,\:D^\bu\rarrow
L\ot_RD^\bu$ and ${-}\ot w\:D^\bu\rarrow D^\bu\ot_RL$.
 Then the Becker coderived and contraderived categories are again
equivalent via
\[
 \Hom^\cu_{R}(D^\bu,{-})\:
 \sH^0(B^\cu\bModl_\binj)\simeq\sH^0(B^\cu\bModl_\bflat^\bcot)
 :\!D^\bu\ot_{R}{-}
\]
(compare to~\cite[Theorem~2.5]{EP}). Here, the functors are explained in detail in~\cite[p.~1199--1200]{EP}, they are well defined thanks to~\cite[Remark~4.1]{IK} and~\cite[Lemma~2.3]{enochs-flat-cot}, and they are mutually inverse equivalences since the functors in~\eqref{dualizing-cpx-co-contra-corresp} are such even for non-commutative $R$ (see~\cite[Theorem 4.2]{IT}).

 To show a concrete instance of matrix factorizations, the one from~\cite[\S2.3.5]{DK} involves, translated to our language, the hereditary order
\[
R = \begin{pmatrix}
S & \mathfrak{m} & \mathfrak{m} & \cdots & \mathfrak{m} \\
S & S & \mathfrak{m} && \vdots \\
\vdots && \ddots & \ddots & \vdots \\
\vdots &&& \ddots & \mathfrak{m} \\
S & S & S & \cdots & S\\
\end{pmatrix},
\]
where $S=k[x]$ is the polynomial ring over a field $k$ and $\mathfrak{m}=(x)$. The invertible bimodule is the trivial one, $L=R$ (and so $B^*=R[T,T^{-1}]$ for a symbol $T$ of degree $2$), and the potential is $w=x\cdot I_n$, where $I_n$ is the identity matrix. As explained in~\cite[Theorem~2.3.6]{DK}, the full subcategory of $\sZ^0(B^\cu\bModl_\bproj)$ formed by CDG\+modules over $B^\cu$ with finitely generated graded projective components is equivalent to Happel's root category~\cite[\S5]{Hap} of Dynkin quivers of type $A_n$. In fact, this subcategory is precisely the category of compact objects of $\sZ^0(B^\cu\bModl_\bproj)$ (compact matrix factorizations will be discussed more in detail in Example~\ref{matrix-factorizations-example-contd}).

 For $R$ commutative regular local, $L=R$ and $w\ne 0$, the equivalent graded-projective and graded-flat graded-cotorsion incarnations of $\sD^\bctr(B^\cu\bModl)$ were proved in \cite[Theorems~4.2, 5.4, and Corollaries~4.3, 5.5]{BT-flat-cot-MF} to be equivalent to the stable categories of Gorenstein projective and Gorenstein flat-cotorsion $R/(w)$\+modules. So these stable categories are consequently equivalent as well.

 In general, if $R$ is commutative, we can view these matrix factorizations in the context of the affine scheme $X=\Spec R$ and quasi-coherent sheaves.
 The results here are also of interest to matrix factorization theory on
schemes that are not necessarily Noetherian of finite Krull
dimension (see the discussion in~\cite[Remarks~1.4, 2.4, and~2.6]{EP}).
 It would be even more interesting to extend these assertions to
nonaffine schemes (e.~g., using very flat quasi-coherent
sheaves~\cite[Sections~0.2\+-0.5]{PSl} as a substitute for
nonexistent projective ones; cf.~\cite[Corollary~6.2]{ES}).
\end{ex}

\Section{Quillen Equivalence of Contraderived Abelian Model Structures}
\label{the-Quillen-equivalence-secn}

 The aim of this section is to promote the triangulated equivalences
from Theorems~\ref{contraderived-cotorsion-pair-category-theorem}(b), \ref{projective-and-flat-contraderived-theorem}
and~\ref{graded-flat-coderived-contraderived-theorem} to a Quillen
equivalence of suitable abelian model structures.
This is in principle a qualitatively better result (see~\cite[Remark~2.5]{DS-K-th}) as Quillen equivalences preserve homotopical structure such as homotopy limits and colimits or simplicial mapping spaces, as well as invariants such as the Waldhausen $K$-theory~\cite[\S3]{DS-K-th}. Plain equivalences of homotopy categories may not enjoy such favorable properties by~\cite[Remarks~2.5 and~6.8]{DS-K-th}. We also refer to~\cite{Schli,DS-expl} for a concrete counterexample of strikingly simple equivalent homotopy categories whose natural model structures differ in the Waldhausen $K$-theory and so cannot be connected by Quillen equivalences.

\subsection{Graded-flat cotorsion pair}
 As explained in \S\ref{abelian-model-structures-subsecn}, Hovey observed in~\cite{Hov} that a class of model structures, which are now called abelian, are based on cotorsion pairs. With this in mind,
 the aim of this section is to prove the following proposition.

\begin{prop} \label{graded-flat-cotorsion-pair-prop}
 Let $B^\cu=(B^*,d,h)$ be a CDG\+ring.
 Then there exists a hereditary complete cotorsion pair $(\sL,\sR)$
generated by a set of objects in the abelian category\/
$\sA=\sZ^0(B^\cu\bModl)$ such that\/ $\sL=\sZ^0(B^\cu\bModl_\bflat)$
is the class of all graded-flat left CDG\+modules over~$B^\cu$.
\end{prop}

\begin{proof}
 Let $\sS\subset\sZ^0(\bModr B^\cu)$ be a class of right CDG\+modules
over $B^\cu$ generating the coderived cotorsion pair in
$\sZ^0(\bModr B^\cu)$ (as defined in Section~\ref{coderived-subsecn}).
 For example, one can take $\sS$ to be the class of all coacyclic
right CDG\+modules over $B^\cu$, or $\sS$ can be the set of
contractible CDG\+modules defined in~\cite[proof of Theorem~7.11]{PS5}
(where the functor denoted by $\Psi^+$ in~\cite{PS5} corresponds to
the functor $G^+$ from Section~\ref{prelim-delta-and-G-subsecn} above
in the case of CDG\+modules; see~\cite[diagram~(6)
in Section~2.6]{PS5}).

 Let $(\sL,\sR)$ be the cotorsion pair in $\sZ^0(B^\cu\bModl)$
cogenerated by the class $\sT=\{\Hom_\boZ^\cu(S^\cu,\boQ/\boZ)\mid
S^\cu\in\sS\}$ of the character CDG\+modules of the CDG\+modules
from the class~$\sS$.
 We claim that $\sL=\sZ^0(B^\cu\bModl_\bflat)$ is the class of all
graded-flat left CDG\+modules over~$B^\cu$.

 Indeed, a graded left module $F^*$ over a graded ring $B^*$ is flat
if and only the graded right $B^*$\+module $\Hom_\boZ^*(F^*,\boQ/\boZ)$
is injective.
 Hence a graded left CDG\+module $F^\cu$ over $B^\cu$ is graded-flat
if and only if the graded right CDG\+module
$\Hom_\boZ^\cu(F^\cu,\boQ/\boZ)$ is graded-injective.
 Since the class $\sS$ generates the coderived cotorsion pair in
$\sZ^0(\bModr B^\cu)$, the latter condition holds if and only if
the abelian group
$$
 \Ext^1_{\sZ^0(\bModr B^\cu)}(S^\cu,\Hom_\boZ^\cu(F^\cu,\boQ/\boZ))
$$
vanishes for all $S^\cu\in\sS$.
 In view of
the formula~\eqref{character-module-Ext-Tor-homological-formula}
for the graded ring $B^*[\delta]^\rop$ and $n=1$, this $\Ext^1$ group
vanishes if and only if so does the group
$$
 \Tor_1^{B^*[\delta]}(S^\cu,F^\cu)^0.
$$
 Applying
the same formula~\eqref{character-module-Ext-Tor-homological-formula}
for the graded ring $B^*[\delta]$ and $n=1$, we see that the latter
$\Tor_1$ group vanishes if and only if so does the group
$$
 \Ext^1_{\sZ^0(B^\cu\bModl)}(F^\cu,\Hom_\boZ^\cu(S^\cu,\boQ/\boZ)),
$$
as desired.

 Finally, the cotorsion pair $(\sL,\sR)$ in $\sZ^0(B^\cu\bModl)$
is complete and generated by a set of CDG\+modules by
Theorem~\ref{cogenerated-by-pure-injectives}.
 This cotorsion pair is hereditary, since the class of all graded-flat
CDG\+modules is obviously closed under kernels of epimorphisms in
$\sZ^0(B^\cu\bModl)$.
\end{proof}

\subsection{Flat contraderived abelian model structure}
 Now we obtain both the abelian model structures needed for our Quillen equivalence.
 We first recall the abelian model structure constructed
in~\cite[Section~6]{PS5} (in the particular case of CDG\+modules).

\begin{thm} \label{projective-contraderived-model-structure}
 Let $B^\cu=(B^*,d,h)$ be a CDG\+ring.
 Then the triple of classes of objects\/
$\sL=\sZ^0(B^\cu\bModl_\bproj)$, \ $\sW=\sZ^0(B^\cu\bModl)_\ac^\bctr$,
and\/ $\sR=\sZ^0(B^\cu\bModl)$ is a cofibrantly generated hereditary
abelian model structure on the abelian category of CDG\+modules\/
$\sZ^0(B^\cu\bModl)$.
\end{thm}

\begin{proof}
 This is~\cite[Proposition~1.3.6(1)]{Bec} and a particular case
of~\cite[Theorem~6.17]{PS5}.
\end{proof}

 The abelian model structure of
Theorem~\ref{projective-contraderived-model-structure} was called
the ``contraderived model structure'' in~\cite{Bec,PS4,PS5}.
 We call it the \emph{projective contraderived model structure}, to
distinguish it from the abelian model structure provided by the next
theorem.

\begin{thm} \label{flat-contraderived-model-structure}
 Let $B^\cu=(B^*,d,h)$ be a CDG\+ring.
 Then the triple of classes of objects\/
$\sL=\sZ^0(B^\cu\bModl_\bflat)$, \ $\sW=\sZ^0(B^\cu\bModl)_\ac^\bctr$,
and\/ $\sR=\sZ^0(B^\cu\bModl^\bcot)$ is a cofibrantly generated
hereditary abelian model structure on the abelian category of
CDG\+modules\/ $\sZ^0(B^\cu\bModl)$.
\end{thm}

\begin{proof}
 The argument is based on Gillespie's theorem
(Theorem~\ref{gillespie-theorem} above).
 Consider two nested hereditary complete cotorsion pairs in
the abelian category of CDG\+modules $\sZ^0(B^\cu\bModl)$:
\begin{enumerate}
\renewcommand{\theenumi}{\roman{enumi}}
\item the flat cotorsion pair ($\widetilde\sL=\sZ^0(B^\cu\bModl)_\flat$,
$\sZ^0(B^\cu\bModl)^\cot=\sR$) in the graded module category
$\sZ^0(B^\cu\bModl)=B^*[\delta]\sModl$, as per
Section~\ref{cotorsion-graded-modules-subsecn};
\item the graded-flat cotorsion pair ($\sL=\sZ^0(B^\cu\bModl_\bflat)$,
$\widetilde\sR$) from Proposition~\ref{graded-flat-cotorsion-pair-prop}.
\end{enumerate}

 Recall that $\sZ^0(B^\cu\bModl)^\cot=\sZ^0(B^\cu\bModl^\bcot)$ by
Theorem~\ref{cotorsion=graded-cotorsion-theorem}; so there is
no ambiguity in our definitions of the class~$\sR$.
 The inclusion $\widetilde\sL=\sZ^0(B^\cu\bModl)_\flat\subset
\sZ^0(B^\cu\bModl_\bflat)=\sL$ explained in
Section~\ref{prelim-proj-inj-flatness-subsecn} implies the inclusion
$\widetilde\sR\subset\sR$.

 Next we need to show that the cores of the two cotorsion pairs agree,
that is $\widetilde\sL\cap\sR=\sL\cap\widetilde\sR$.
 Let $F^\cu$ be a CDG\+module belonging to $\sL\cap\widetilde\sR=
\sL\cap\sL^{\perp_1}\subset\sZ^0(B^\cu\bModl)$.
 So, first of all, $F^\cu$ is a graded-flat CDG\+module over~$B^\cu$.
 Furthermore, we have $\Ext^1_{\sZ^0(B^\cu\bModl)}(P^\cu,F^\cu)=0$
for any graded-flat CDG\+module $P^\cu$, and in particular, for any
graded-projective CDG\+module $P^\cu$ over~$B^\cu$.
 In other words, $F^\cu$ is a graded-flat contraacyclic CDG\+module.
 By Corollary~\ref{flat=graded-flat-and-contraacyclic-cor}, it follows
that $F^\cu$ is a flat CDG\+module over~$B^\cu$;
so $F^\cu\in\widetilde\sL$.

 Conversely, let $F^\cu$ be a CDG\+module belonging to
$\widetilde\sL\cap\sR$.
 So, $F^\cu$ is a flat cotorsion CDG\+module over~$B^\cu$.
 Let us show that $F^\cu\in\widetilde\sR=\sL^{\perp_1}\subset
\sZ^0(B^\cu\bModl)$.
 Given a graded-flat CDG\+module $G^\cu$ over $B^\cu$, we need
to prove that the abelian group
$\Ext^1_{\sZ^0(B^\cu\bModl)}(G^\cu,F^\cu)$ vanishes.

 For this purpose, consider the restricted contraderived cotorsion pair
from Section~\ref{restricted-contraderived-cotorsion-pair-subsecn}.
 This is the complete cotorsion pair formed by the classes of
graded-projective CDG\+modules and graded-flat contraacyclic
CDG\+modules in the exact category of graded-flat CDG\+modules
$\sZ^0(B^\cu\bModl_\bflat)$.
 By Corollary~\ref{flat=graded-flat-and-contraacyclic-cor},
the class of graded-flat contraacyclic CDG\+modules coincides with
the class of flat CDG\+modules.

 Pick a special preenvelope exact sequence $0\rarrow G^\cu\rarrow H^\cu
\rarrow P^\cu\rarrow0$ \,\eqref{special-preenvelope-sequence}
in $\sZ^0(B^\cu\bModl_\bflat)$ with a flat CDG\+module $H^\cu$ and
a graded-projective CDG\+module $P^\cu$ over~$B^\cu$.
 Now we have $\Ext^1_{\sZ^0(B^\cu\bModl)}(H^\cu,F^\cu)=0$ since
$H^\cu$ is a flat CDG\+module and $F^\cu$ is a cotorsion CDG\+module.
 We also have $\Ext^2_{\sZ^0(B^\cu\bModl)}(P^\cu,F^\cu)=0$ since
$P^\cu$ is a graded-projective CDG\+module, $F^\cu$ is a flat
CDG\+module (hence contraacyclic, by
Corollary~\ref{flats-are-contraacyclic-cor}), and the projective
contraderived cotorsion pair in $\sZ^0(B^\cu\bModl)$ is hereditary.
 Thus $\Ext^1_{\sZ^0(B^\cu\bModl)}(G^\cu,F^\cu)=0$.

 Now we can apply Theorem~\ref{gillespie-theorem} to produce
a hereditary abelian model structure on $\sZ^0(B^\cu\bModl)$ with
the desired class of cofibrant objects $\sL$ and the desired class
of fibrant objects~$\sR$.
 It remains to prove that the resulting class of weakly trivial
objects $\sW$ coincides with the class of all contraacyclic
CDG\+modules.

 Indeed, all objects in the class
$\widetilde\sL=\sZ^0(B^\cu\bModl)_\flat$ are contraacyclic by
Corollary~\ref{flats-are-contraacyclic-cor}.
 All objects in the class $\widetilde\sR=\sL^{\perp_1}$ are
contraacyclic by the definition, since the class
$\sL=\sZ^0(B^\cu\bModl_\bflat)$ contains all the graded-projective
CDG\+modules.
 Using any one of the two descriptions of the class $\sW$ in
Gillespie's theorem (Theorem~\ref{gillespie-theorem}) and
the fact that the class of all contraacyclic CDG\+modules is thick
(by~\cite[Lemma~6.13(a)]{PS5} or
Theorem~\ref{projective-contraderived-model-structure}), it follows
that all the CDG\+modules belonging to $\sW$ are contraacyclic.

 Conversely, let $W^\cu$ be a contraacyclic left CDG\+module
over~$B^\cu$.
 Consider either a special precover or a special preenvelope short
exact sequence for $W^\cu$ in the complete cotorsion pair
$(\sL,\widetilde\sR)$ in $\sZ^0(B^\cu\bModl)$.
 We have already mentioned that all the CDG\+modules from
$\widetilde\sR$ are contraacyclic and the class $\sW$ is thick
in $\sZ^0(B^\cu\bModl)$.
 Hence all the three terms of our approximation sequence are
contraacyclic CDG\+modules.
 As all graded-flat contraacyclic CDG\+modules are flat by
Corollary~\ref{flat=graded-flat-and-contraacyclic-cor}, we have
obtained a short exact sequence showing that $W\in\sW$ according to
the description of the class $\sW$ in Theorem~\ref{gillespie-theorem}.

 Finally, the abelian model structure $(\sL,\sW,\sR)$ on
$\sZ^0(B^\cu\bModl)$ is cofibrantly generated
by~\cite[Corollary~4.3]{PS4}, as both the cotorsion pairs
$(\widetilde\sL,\sR)$ and $(\sL,\widetilde\sR)$ are generated by some
sets of objects (see Section~\ref{cotorsion-graded-modules-subsecn}
and Proposition~\ref{graded-flat-cotorsion-pair-prop}).
\end{proof}

 We will call the abelian model structure $(\sL,\sW,\sR)$ defined in
Theorem~\ref{flat-contraderived-model-structure} the \emph{flat
contraderived model structure} on the abelian category of CDG\+modules
$\sZ^0(B^\cu\bModl)$.

%
%
%
%
%

\subsection{The Quillen equivalence}
 Now we can see that the projective and flat contraderived model
structures on the abelian category of CDG\+modules $\sZ^0(B^\cu\bModl)$
are Quillen equivalent.

\begin{thm}\label{the-Quillen-equivalence-thm}
 Let $B^\cu=(B^*,d,h)$ be a CDG\+ring.
 The identity adjunction (of the two identity functors) between
two copies of the abelian category\/ $\sZ^0(B^\cu\bModl)$ is
a Quillen adjunction, and in fact a Quillen equilalence, between
the projective contraderived model structure and the flat contraderived
model structure on the abelian category of CDG\+modules over~$B^\cu$.
 The identity functor acting from\/ $\sZ^0(B^\cu\bModl)$ endowed with
the projective contraderived model structure to $\sZ^0(B^\cu\bModl)$
endowed with the flat contraderived model structure is the left
Quillen equivalence, and the identity functor acting from the flat
contraderived model structure to the projective one is the right
Quillen equivalence.
\end{thm}

\begin{proof}
 The pair of identity functors is a Quillen adjunction because
all graded-projective CDG\+modules are graded-flat, or if one wishes,
because the class of graded-cotorsion CDG\+modules is contained in
the class of all CDG\+modules.
 The Quillen adjunction is a Quillen equivalence (as per the discussion
in Section~\ref{quillen-adjunctions-equivalences-subsecn}) because
the classes of weakly trivial objects and hence also the classes of 
weak equivalences (recall Theorem~\ref{abelian-model-structures-thm})
in the two model structures agree.
\end{proof}

\begin{rem}
A more practical description of the class of weak equivalences
than that provided by the above proof is as follows.
A morphism in the abelian category of CDG\+modules\/
$\sZ^0(B^\cu\bModl)$ is a weak equivalence in either of the two
contraderived model structures if and only if it has a contraacyclic
cone.
This is for example~\cite[Lemma~6.18]{PS5} specialized to the DG\+categories
of CDG\+modules.
Consequently, a morphism between graded-flat CDG\+modules (i.~e.\ cofibrant objects of the flat contraderived module structure) is a weak equivalence by Corollary~\ref{flat=graded-flat-and-contraacyclic-cor} if and only if its cone is a flat CDG\+module.
\end{rem}

\begin{rem}\label{rem:how-the-Quillen-equivalence-connects}
Let us make it clear how Theorems~\ref{projective-contraderived-model-structure}, \ref{flat-contraderived-model-structure} and~\ref{the-Quillen-equivalence-thm} encapsulate the main results of Section~\ref{projective-and-flat-contraderived-secn}. First of all, Becker proved in~\cite[Proposition~1.1.14]{Bec} that, in the terminology of~\S\ref{abelian-model-structures-subsecn},
\begin{enumerate}
\item the class of cofibrant and fibrant objects of a hereditary abelian model category $\sA$ is a Frobenius exact category, and
\item the stable category of this Frobenius exact category is canonically equivalent to the homotopy category $\sA[\cW^{-1}]$.
\end{enumerate}
If we apply this to the projective contraderived model structure from Theorem~\ref{projective-contraderived-model-structure}, we recover the equivalence from Theorem~\ref{contraderived-cotorsion-pair-category-theorem}(b). Next, it is a standard fact that for any model category $\sC$ with the class of weak equivalences $\cW$ and the full subcategory of cofibrant objects $\sL$, the canonical functor $\sL[(\sL\cap\cW)^{-1}]\to\sC[\cW^{-1}]$ is an equivalence \cite[Proposition 1.2.3]{Hov-book}. Since Quillen equivalences induce equivalences between homotopy categories by Lemma~\ref{quillen-equivalence-lemma}, the combination of the latter observation with Theorem~\ref{the-Quillen-equivalence-thm} recovers Theorem~\ref{projective-and-flat-contraderived-theorem}. Finally, Theorem~\ref{graded-flat-coderived-contraderived-theorem} again follows by an application of \cite[Proposition~1.1.14]{Bec} to the flat contraderived model structure of Theorem~\ref{flat-contraderived-model-structure}.
\end{rem}

\Section{Compact Generators of the Contraderived Category}
\label{compact-generators-of-contraderived-secn}

 In this section we are interested in CDG\+rings $B^\cu=(B^*,d,h)$
such that the graded ring $B^*$ is graded right coherent.
 Under this assumption, \cite[Corollary~8.20]{PS5} tells that
the coderived category of right CDG\+modules
$\sD^\bco(\bModr B^\cu)$ is compactly generated and describes its
full subcategory of compact objects (up to idempotent completion).
 The aim of this section is to show that, under the same assumption,
the contraderived category of left CDG\+modules
$\sD^\bctr(B^\cu\bModl)$ is compactly generated as well, and describe
its full subcategory of compact objects.

\subsection{The tensor product pairing functor}      
\label{tensor-product-pairing-subsecn}
 Let $B^\cu=(B^*,d,h)$ be a CDG\+ring.
 Let $N^\cu=(N^*,d_N)$ be a right CDG\+module and $M^\cu=(M^*,d_M)$
be a left CDG\+module over~$B^\cu$.
 Then, following the discussion in
Sections~\ref{prelim-cdg-bimodules-subsecn}
and~\ref{prelim-dg-categories-of-cdg-modules-subsecn}, the tensor
product $N^*\ot_{B^*}M^*$ of the graded right $B^*$\+module $N^*$
and the graded left $B^*$\+module $M^*$ is endowed with a natural
differential making it a complex of abelian groups.
 The DG\+functor of two arguments
$$
 \ot_{B^*}\:\bModr B^\cu\times B^\cu\bModl\lrarrow\bC(\Ab)
$$
induces a triangulated functor of two arguments
\begin{equation} \label{tensor-product-on-homotopy-categories}
 \ot_{B^*}\:\sH^0(\bModr B^\cu)\times\sH^0(B^\cu\bModl)
 \lrarrow\sK(\Ab)
\end{equation}
taking values in the homotopy category of complexes of abelian
groups $\sK(\Ab)$.

 The aim of this section is to construct a left derived functor of
the functor~\eqref{tensor-product-on-homotopy-categories} acting
from the Cartesian product of the coderived and the contraderived
category to the derived category of abelian groups,
\begin{equation} \label{derived-tensor-product}
 \ot_{B^*}^\boL\:\sD^\bco(\bModr B^\cu)\times\sD^\bctr(B^\cu\bModl)
 \lrarrow\sD(\Ab).
\end{equation}
 It is easy to construct
a derived functor~\eqref{derived-tensor-product} using
graded-projective resolutions of the second argument, but a more
substantial approach consists in using graded-flat resolutions and
Theorem~\ref{projective-and-flat-contraderived-theorem}.
In the jargon of model categories, this is the same as proving that
$\ot_{B^*}\:\sZ^0(\bModr B^\cu)\times\sZ^0(B^\cu\bModl)
\lrarrow\sZ^0\bC(\Ab)$ is a Quillen bifunctor with respect to
the coderived model structure on $\sZ^0(\bModr B^\cu)$, the flat
contraderived model structure on $\sZ^0(B^\cu\bModl)$, and the injective
model structure on the abelian category of complexes of abelian groups
$\sC(\Ab)=\sZ^0\bC(\Ab)$ in the sense of~\cite[Theorem~2.3.13]{Hov-book}
(we refer to~\cite[Section~4.2]{Hov-book}
or~\cite[Section~8.3]{Sto-ICRA} for details).

 Specifically, let $N^\cu$ be a right CDG\+module and $M^\cu$ be
a left CDG\+module over~$B^\cu$.
 Following Theorems~\ref{contraderived-cotorsion-pair-category-theorem}
and~\ref{projective-and-flat-contraderived-theorem}, we have
a natural triangulated equivalence
\begin{equation} \label{flat-interpretation-of-contraderived}
 \sD^\bctr(B^\cu\bModl)\simeq
 \sH^0(B^\cu\bModl_\bflat)/\sH^0(B^\cu\bModl)_\flat.
\end{equation}
 The construction of the category
equivalence~\eqref{flat-interpretation-of-contraderived}, as per
Theorem~\ref{contraderived-cotorsion-pair-category-theorem}, involves
the possibility to find, for any left CDG\+module $M^\cu$ over $B^\cu$,
a graded-flat left CDG\+module $F^\cu$ over $B^\cu$ together with
a closed morphism of CDG\+modules $F^\cu\rarrow M^\cu$ with
contraacyclic cone.
 We put
\begin{equation} \label{derived-tensor-product-constructed}
 N^\cu\ot_{B^*}^\boL M^\cu=N^\cu\ot_{B^*} F^\cu.
\end{equation}

 In order to show that the derived tensor product
functor~$\ot_{B^*}^\boL$ \,\eqref{derived-tensor-product} is
well-defined by the formula~\eqref{derived-tensor-product-constructed}, 
it remains to check that
\begin{enumerate}
\item the complex of abelian groups $N^\cu\ot_{B^*}G^\cu$ is acyclic
for any right CDG\+module $N^\cu$ and any flat left CDG\+module $G^\cu$
over~$B^\cu$;
\item the complex of abelian groups $Y^\cu\ot_{B^*}F^\cu$ is acyclic
for any coacyclic right CDG\+module $Y^\cu$ and any graded-flat left
CDG\+module $F^\cu$ over~$B^\cu$.
\end{enumerate}

 Concerning~(1), we observe that, by the graded version of
the Govorov--Lazard characterization of flat modules (over the graded
ring~$B^*[\delta]$), the CDG\+module $G^\cu$ is a direct limit of
projective left CDG\+modules over~$B^\cu$.
 Following the discussion in
Section~\ref{prelim-proj-inj-flatness-subsecn}, all projective
CDG\+modules are contractible.
 Since the tensor product of CDG\+modules preserves direct limits,
it follows that the complex $N^\cu\ot_{B^*}G^\cu$ is a direct limit of
contractible complexes of abelian groups; hence it is a (pure) acyclic
complex.
 An alternative proof of~(1) can be obtained by reversing the argument
in the proof of
Lemma~\ref{flat-cdg-modules-acyclic-tensor-product-lemma} below.

 In order to prove~(2), we use the description of coacyclic
CDG\+modules obtained in the paper~\cite{PS5}.
 According to~\cite[Corollary~7.17]{PS5}, the class of coacyclic right
CDG\+modules over $B^\cu$ is precisely the closure of the class of
all contractible right CDG\+modules under extensions and direct limits.
 The tensor product functor ${-}\ot_{B^*}\nobreak F^\cu$ with
a graded-flat left CDG\+module $F^\cu$ over $B^\cu$ preserves
extensions; and the tensor product preserves direct limits
quite generally.
 So the complex $Y^\cu\ot_{B^*}F^\cu$ can be obtained from contractible
complexes of abelian groups using the operations of extensions and
direct limits; hence this complex is acyclic.

\subsection{Compact generators of the coderived category}
\label{compact-generators-of-coderived-subsecn}
 Let $\sT$ be a triangulated category and $\sS\subset\sT$ be
a class of objects.
 One says that the class $\sS$ \emph{weakly generates} the triangulated
category $\sT$ if for every nonzero object $T\in\sT$ there exist
an object $S\in\sS$ and an integer $n\in\boZ$ for which
$\Hom_\sT(S,T[n])\ne0$.

 Let $\sT$ be a triangulated category with coproducts and
$\sS\subset\sT$ be a class of objects.
 One says that the class $\sS$ \emph{generates} $\sT$ (as a triangulated
category with coproducts) if the minimal full triangulated subcategory
of $\sT$ containing $\sS$ and closed under coproducts coincides
with~$\sT$.
 One can easily see that if $\sS$ generates $\sT$ then $\sS$ also
weakly generates $\sT$ (as the terminology suggests).

 Let $\sT$ be a triangulated category with coproducts.
 An object $S\in\sT$ is said to be \emph{compact} if the functor
$\Hom_\sT(S,{-})\:\sT\rarrow\Ab$ preserves coproducts.
 It is well-known that if $\sS$ is a set of compact objects in $\sT$,
then $\sS$ generates $\sT$ if and only if $\sS$ weakly generates~$\sT$
\,\cite[Theorem~2.1(2) or~4.1]{Neem-bb}.
 If this is the case, then the set $\sS$ is said to be a set of
\emph{compact generators} of $\sT$ and the triangulated category $\sT$
is said to be \emph{compactly generated}.

 For any triangulated category $\sT$ with coproducts, the class of
all compact objects in $\sT$ is a full triangulated subcategory.
 If $\sS$ is a set of compact generators of $\sT$, then the class
of all compact objects of $\sT$ is the closure of $\sS$ under shifts,
cones, and direct summands in~$\sT$ \,\cite[Theorem~2.1(3)]{Neem-bb}.

 A graded ring $B^*$ is said to be \emph{graded right coherent} if
every finitely generated homogeneous right ideal in $B^*$ is finitely
presented (as a graded right $B^*$\+module).
 Given a CDG\+ring $B^\cu=(B^*,d,h)$, the graded ring $B^*$ is graded
right coherent if and only if the graded ring $B^*[\delta]$
is~\cite[Corollary~8.9(b)]{PS5}.

 Let $B^\cu=(B^*,d,h)$ be a CDG\+ring such that the graded ring $B^*$
is graded right coherent.
 Then the additive category $\sZ^0(\bmodr B^\cu)$ of finitely presented
right CDG\+modules over $B^\cu$ (as defined in
Section~\ref{prelim-finite-countable-modules-subsecn}) is abelian.

 Any short exact sequence in $\sZ^0(\bModr B^\cu)$, and in particular,
any short exact sequence in $\sZ^0(\bmodr B^\cu)$, can be viewed as
a finite complex of CDG\+modules.
 So one can construct its totalization (the total CDG\+module), as
explained in Section~\ref{totalizations-subsecn} below.
 Clearly, the totalization of a finite complex of finitely presented
CDG\+modules is a finitely presented CDG\+module.

 A finitely presented CDG\+module over $B^\cu$ is said to be
\emph{absolutely acyclic}~\cite[Sections~3.3 and~3.11]{Pkoszul},
\cite[Section~8.4]{PS5} if it belongs to the minimal full triangulated
subcategory closed under direct summands in the homotopy category
$\sH^0(\bmodr B^\cu)$ containing the totalizations of all short exact
sequences in $\sZ^0(\bmodr B^\cu)$.
 The full subcategory of absolutely acyclic finitely presented
CDG\+modules is denoted by $\sH^0(\bmodr B^\cu)_\ac^\abs\subset
\sH^0(\bmodr B^\cu)$ or $\sZ^0(\bmodr B^\cu)_\ac^\abs\subset
\sZ^0(\bmodr B^\cu)$.
 The full subcategory $\sZ^0(\bmodr B^\cu)_\ac^\abs$ is the closure
of the full subcategory of contractible finitely presented
CDG\+modules under extensions and direct summands in the abelian
category $\sZ^0(\bmodr B^\cu)$ \,\cite[Proposition~8.12]{PS5}.

 The triangulated Verdier quotient category
$$
 \sD^\abs(\bmodr B^\cu)=\sH^0(\bmodr B^\cu)/\sH^0(\bmodr B^\cu)_\ac^\abs
$$
is called the \emph{absolute derived category} of finitely presented
right CDG\+modules over~$B^\cu$.
 The following theorem is one of the main results of
the paper~\cite{PS5} (specialized from abelian DG\+categories
to CDG\+modules).

\begin{thm} \label{compact-generators-of-coderived-theorem}
 Let $B^\cu=(B^*,d,h)$ be a CDG\+ring such that the graded ring $B^*$
is graded right coherent.
 Then the inclusion of DG\+categories\/ $\bmodr B^\cu\rarrow
\bModr B^\cu$ induces a fully faithful triangulated functor
\begin{equation} \label{absolute-derived-into-coderived}
 \sD^\abs(\bmodr B^\cu)\lrarrow\sD^\bco(\bModr B^\cu).
\end{equation}
 The coderived category\/ $\sD^\bco(\bModr B^\cu)$ is compactly 
generated, and the set (of representatives of the isomorphism classes)
of all objects in the image of the triangulated
functor~\eqref{absolute-derived-into-coderived} is a set of compact
generators of\/ $\sD^\bco(\bModr B^\cu)$.
\end{thm}

\begin{proof}
 This is~\cite[Corollary~8.20]{PS5} (with the left CDG\+modules
replaced by the right ones).
\end{proof}

\subsection{Contraacyclic and coacyclic total complexes}
\label{totalizations-subsecn}
 Let $B^\cu=(B^*,d,h)$ be a CDG\+ring and $C^{\bu,\cu}$ be a complex
of left CDG\+modules over~$B^\cu$, i.~e., a complex in the abelian
category $\sZ^0(B^\cu\bModl)$
$$
 \dotsb\lrarrow C^{-1,\cu}\lrarrow C^{0,\cu}\lrarrow C^{1,\cu}
 \lrarrow C^{2,\cu}\lrarrow\dotsb
$$
 So, for every $n\in\boZ$, a left CDG\+module
$C^{n,\cu}=(C^{n,*},d_{C^n})$ over $B^\cu$ is given.
 The differential $d_n\:C^{n,\cu}\rarrow C^{n+1,\cu}$ is a closed
morphism of CDG\+modules; and the composition $d_{n+1}d_n\:
C^{n,\cu}\rarrow C^{n+2,\cu}$ vanishes.

 Then the \emph{direct sum totalization}
$S^\cu=\Tot^\sqcup(C^{\bu,\cu})$ of the complex $C^{\bu,\cu}$ is
the CDG\+module $S^\cu=(S^*,d_S)$ with the grading components
$S^i=\bigoplus_{n+j=i}C^{n,j}$ for all $i\in\boZ$ and the differential
given by the formula $d_S(c_{n,j})=d_n(c_{n,j})+(-1)^nd_{C^n}(c_{n,j})$
for all $c_{n,j}\in C^{n,j}$.
 The usual sign rule for the shift of a left graded module mentioned
in Section~\ref{prelim-graded-modules} is involved with the definition
of the action of $B^*$ on~$S^*$.

 Dually, the \emph{direct product totalization}
$T^\cu=\Tot^\sqcap(C^{\bu,\cu})$ of the complex $C^{\bu,\cu}$ is
the CDG\+module $T^\cu=(T^*,d_T)$ with the grading components
$T^i=\prod_{n+j=i}C^{n,j}$ and the differential given by the same
formula as in the previous paragraph.
 The same sign rule is involved with the definition of the action of
$B^*$ on~$T^*$.

 Of course, for a \emph{finite} complex of CDG\+modules $C^{\bu,\cu}$,
one has $\Tot^\sqcup(C^{\bu,\cu})=\Tot^\sqcap(C^{\bu,\cu})$.
 In this context, we will drop the superindex and write simply
$\Tot(C^{\bu,\cu})$.
 We refer to~\cite[Section~1.2]{Pkoszul}, \cite[Section~1.3]{Pedg},
or~\cite[Section~1.6]{PS5} for a more general discussion of
totalizations of complexes in DG\+categories.

 In particular, let $R$ be a ring and $C^{\bu,\bu}$ be a bicomplex
of $R$\+modules.
 Then the direct sum totalization $S^\bu=\Tot^\sqcup(C^{\bu,\bu})$
and the direct product totalization $T^\bu=\Tot^\sqcap(C^{\bu,\bu})$
are complexes of $R$\+modules with the terms
$S^i=\bigoplus_{n+j=i}C^{n,j}$ and $T^i=\prod_{n+j=i}C^{n,j}$.

\begin{lem} \label{co-contra-acyclic-total-complex-lemma}
 Let $R$ be a ring and $C^{\bu,\bu}$ be a bicomplex of $R$\+modules.
\par
\textup{(a)} Assume that $C^{n,\bu}=0$ for all $n<0$ and, for every
$j\in\boZ$, the complex of $R$\+modules
$$
 0\lrarrow C^{0,j}\lrarrow C^{1,j}\lrarrow C^{2,j}\lrarrow\dotsb
$$
is acyclic.
 Then the complex of $R$\+modules\/ $\Tot^\sqcup(C^{\bu,\bu})$ is
coacyclic (and therefore, acyclic). \par
\textup{(b)} Assume that $C^{n,\bu}=0$ for all $n>0$ and, for every
$j\in\boZ$, the complex of $R$\+modules
$$
 \dotsb\lrarrow C^{-2,j}\lrarrow C^{-1,j}\lrarrow C^{0,j}\lrarrow0
$$
is acyclic.
 Then the complex of $R$\+modules\/ $\Tot^\sqcap(C^{\bu,\bu})$ is
contraacyclic (and therefore, acyclic).
\end{lem}

\begin{proof}
 In fact, for any acyclic bounded below complex of CDG\+modules
$C^{\bu,\cu}$, the CDG\+module $\Tot^\sqcup(C^{\bu,\cu})$ is
coacyclic in the sense of Positselski~\cite[Lemma~2.1]{Psemi},
\cite[Lemma~9.12(a)]{Pksurv}.
 Any CDG\+module coacyclic in the sense of Positselski is also
coacyclic in the sense of Becker~\cite[Lemma~7.9]{PS5}.
 One can easily see that any DG\+module coacyclic in the sense of
Becker is acyclic.

 Dually, for any acyclic bounded above complex of CDG\+modules
$C^{\bu,\cu}$, the CDG\+module $\Tot^\sqcap(C^{\bu,\cu})$ is
contraacyclic in the sense of Positselski~\cite[Lemma~4.1]{Psemi},
\cite[Lemma~9.12(b)]{Pksurv}.
 Any CDG\+module contraacyclic in the sense of Positselski is also
contraacyclic in the sense of Becker~\cite[Lemma~6.13]{PS5}.
 One can easily see that any DG\+module contraacyclic in the sense
of Becker is acyclic.
\end{proof}

\begin{lem} \label{contractible-totalization-lemma}
 Let $B^\cu=(B^*,d,h)$ be a CDG\+ring. \par
\textup{(a)} Let $C^{\bu,\cu}$ be a bounded below complex of
CDG\+modules over $B^\cu$, i.~e., $C^{n,\cu}=0$ for $n<0$,
$$
 0\lrarrow C^{0,\cu}\lrarrow C^{1,\cu}\lrarrow C^{2,\cu}\lrarrow\dotsb
$$
 Assume that the complex $C^{\bu,\cu}$ is graded-split exact, that is,
its underlying complex of graded $B^*$\+modules $C^{\bu,*}$
is split exact as a complex in the additive category of graded
$B^*$\+modules.
 Then the total CDG\+module\/ $\Tot^\sqcup(C^{\bu,\cu})$ over $B^\cu$
is contractible. \par
\textup{(b)} Let $C^{\bu,\cu}$ be a bounded above complex of
CDG\+modules over $B^\cu$, i.~e., $C^{n,\cu}=0$ for $n>0$,
$$
 \dotsb\lrarrow C^{-2,\cu}\lrarrow C^{-1,\cu}\lrarrow C^{0,\cu}
 \lrarrow0.
$$
 Assume that the complex $C^{\bu,\cu}$ is graded-split exact.
 Then the total CDG\+module\/ $\Tot^\sqcap(C^{\bu,\cu})$ over $B^\cu$
is contractible. 
\end{lem}

\begin{proof}
 We will only prove and use part~(a).
 This is a version of (the proof of)
Lemma~\ref{co-contra-acyclic-total-complex-lemma}(a).
 One can argue similarly to the proof of~\cite[Lemma~2.1]{Psemi}
or~\cite[Lemma~9.12(a)]{Pksurv} using the fact that the totalization
of a graded-split short exact sequence of CDG\+modules is contractible.
 Alternatively, here is a direct argument.
 Let $h\:C^{\bu,*}\rarrow C^{\bu,*}$ be a cochain homotopy contracting
the split exact complex $C^{\bu,*}$ in the additive category of graded
$B^*$\+modules.
 So, for every $n\ge0$, we have a morphism of graded $B^*$\+modules
$h_n\:C^{n,*}\rarrow C^{n-1,*}$.
 Put $S^\cu=\Tot^\sqcup(C^{\bu,\cu})$, and consider~$h$ as
a homogeneous $B^*$\+module map $S^*\rarrow S^*$ of degree~$-1$.
 Then $t=\id-(d_Sh+hd_S)\:S^\cu\rarrow S^\cu$ is a closed endomorphism
of degree~$0$ of the CDG\+module $S^\cu$ mapping $C^{n,*}\subset S^*$
to $C^{n-1,*}\subset S^*$ for every $n\ge0$.
 For this reason, the endomorphism $t$ is locally nilpotent, hence
the endomorphism $d_Sh+hd_S=\id-t\:S^\cu\rarrow S^\cu$ is
an automorphism.
 Indeed, the inverse automorphism $(\id-t)^{-1}\:S^\cu\rarrow S^\cu$
can be constructed as $(\id-t)^{-1}=\sum_{m=0}^\infty t^m$, with
the infinite sum becoming finite when applied to any specific
element of~$S^*$.
 A CDG\+module (more generally, a DG\+category object) admitting
a closed automorphism homotopic to zero is contractible.
 The proof of part~(b) is dual-analogous.
\end{proof}

\subsection{Absolute derived category produced via complexes
of CDG-modules}
\label{complexes-of-cdg-modules-subsecn}
 The aim of this section is to extend to (finitely presented)
CDG\+modules a certain technique invented by
Efimov~\cite[Remark~2.7]{EP}, \cite[Appendix~A]{Ef} in the context of
matrix factorizations.

 Let $B^\cu=(B^*,d,h)$ be a CDG\+ring.
 When the graded ring $B^*$ is graded right coherent (as defined in
Section~\ref{compact-generators-of-coderived-subsecn}), the additive
category $\sZ^0(\bmodr B^\cu)$ of finitely presented right CDG\+modules
over $B^\cu$ is abelian.
 We start with defining the \emph{graded-split exact structure}
on the additive category $\sZ^0(\bmodr B^\cu)$ for an arbitrary
CDG\+ring~$B^\cu$.

 A short sequence of CDG\+modules $0\rarrow K^\cu\rarrow L^\cu
\rarrow M^\cu\rarrow0$ over $B^\cu$ is said to be \emph{graded-split
exact} if the underlying sequence of graded $B^*$\+modules
$0\rarrow K^*\rarrow L^*\rarrow M^*\rarrow0$ is split exact in
the additive category $\sModr B^*$.
 One can easily observe that pullbacks and pushouts of graded-split
exact sequences are graded-split exact.
 Furthermore, the class of all finitely presented CDG\+modules is
closed under extensions in $\sZ^0(\bModr B^\cu)$; so in particular
it is closed under graded-split extensions.
 Consequently, the class of all graded-split exact sequences of
finitely presented right CDG\+modules defines an exact structure on
the additive category $\sZ^0(\bmodr B^\cu)$, which we call
the \emph{graded-split exact structure}.

 An exact category is said to be \emph{Frobenius} if it has enough
projective objects, enough injective objects, and the classes of
projective and injective objects coincide.

\begin{lem} \label{graded-split-exact-structure-lemma}
 Let $B^\cu=(B^*,d,h)$ be a CDG\+ring.
 Then the exact category\/ $\sZ^0(\bmodr B^\cu)$ with the graded-split
exact structure is Frobenius.
 The projective-injective objects of the graded-split exact structure
on\/ $\sZ^0(\bmodr B^\cu)$ are precisely the contractible CDG\+modules.
\end{lem}

\begin{proof}
 Let $E^\cu$ be a contractible CDG\+module over~$B^\cu$.
 Then Lemma~\ref{Ext-1-homotopy-hom-lemma} implies that any graded-split
short exact sequence of CDG\+modules with the leftmost or rightmost term
$E^\cu$ splits.
 So contractible CDG\+modules are both projective and injective in
the graded-split exact structure on $\sZ^0(\bmodr B^\cu)$.
 For any finitely presented right CDG\+module $N^\cu$ over $B^\cu$,
the natural graded-split short exact sequences of CDG\+modules $0\rarrow
N^\cu\rarrow\cone(\id_{N^\cu})\rarrow N^\cu[1]\rarrow0$ and $0\rarrow
N^\cu[-1]\rarrow\cone(\id_{N^\cu})[-1]\rarrow N^\cu\rarrow0$ represent
$N^\cu$ as the domain of an admissible monomorphism into a contractible
finitely presented CDG\+module and the codomain of an admissible
epimorphism from a contractible finitely presented CDG\+module.
 So there are enough contractible CDG\+modules both as projective and as
injective objects in $\sZ^0(\bmodr B^\cu)$.
 As the class of all contractible CDG\+modules is closed under direct
summands in $\sZ^0(\bmodr B^\cu)$, it follows that all projective
objects of the graded-split exact structure on $\sZ^0(\bmodr B^\cu)$
are contractible, and all injective objects of the same exact structure
are contractible.
\end{proof}

 The following lemma is quite well-known and standard,
see for example~\cite[Theorem~4.4.1]{Buch86} or~\cite[Exemple~2.3]{KeVo87}.

\begin{lem} \label{stable-category-lemma}
 Let\/ $\sE$ be a Frobenius exact category.
 Then the following three triangulated categories are naturally
equivalent:
\begin{enumerate}
\item the stable category\/ $\sE^\st$ of\/ $\sE$, i.~e.,
the additive quotient category of\/ $\sE$ by the ideal of morphisms
factorizable through projective-injective objects;
\item the full subcategory in the unbounded homotopy category\/
$\sK(\sE)$ of complexes in\/ $\sE$ consisting of all the acyclic
complexes of projective-injective objects;
\item the triangulated Verdier quotient category of the bounded derived
category\/ $\sD^\bb(\sE)$ by the thick subcategory of bounded complexes
of projective-injective objects in\/~$\sE$.
\end{enumerate}
\end{lem}

\begin{proof}
 The natural triangulated functor from~(1) to~(2) assigns to
an object of $\sE$ its acyclic two-sided projective-injective
resolution.
 The inverse functor from~(2) to~(1) assigns to an acyclic complex
of projective-injective objects in $\sE$ its object of degree~$0$
cocycles.
 The natural triangulated functor from~(1) to~(3) is induced by
the inclusion $\sE\rarrow\sD^\bb(\sE)$.
 The natural triangulated functor from~(3) to~(2) assigns to a bounded
complex $C^\bu$ in~$\sE$ the cone of the composition of
quasi-isomorphisms $P^\bu\rarrow C^\bu\rarrow J^\bu$, where $P^\bu$
is a bounded above complex of projectives and $J^\bu$ is a bounded
below complex of injectives in~$\sE$.
\end{proof}

 In the next proposition, we consider (bounded) \emph{complexes of}
(finitely presented) \emph{CDG\+modules} and denote the category of
such complexes by $\sC^\bb(\sZ^0(\bmodr B^\cu))$.
 In this context, the category of finitely presented CDG\+modules
$\sZ^0(\bmodr B^\cu)$ is viewed just as an additive category endowed
with various exact structures.

 We denote by $\sZ^0(\bmodr B^\cu)_\gs$ the additive category of
finitely presented CDG\+mod\-ules endowed with the graded-split
exact structure.
 Assuming the graded ring $B^*$ to be graded right coherent,
the undecorated notation $\sZ^0(\bmodr B^\cu)$ stands for
the abelian category $\sZ^0(\bmodr B^\cu)$ endowed with the abelian
exact structure.
 Accordingly, we speak about the bounded derived category of
the abelian/exact category of finitely presented CDG\+modules.

\begin{prop} \label{bounded-derived-and-absolute-derived-prop}
\textup{(a)} Let $B^\cu=(B^*,d,h)$ be a CDG\+ring.
 Then the totalization functor
$$
 \Tot\:\sC^\bb(\sZ^0(\bmodr B^\cu))\lrarrow\sZ^0(\bmodr B^\cu)
$$
from Section~\ref{totalizations-subsecn} induces a triangulated functor
$$
 \sD^\bb(\sZ^0(\bmodr B^\cu)_\gs)\lrarrow
 \sH^0(\bmodr B^\cu),
$$
which is a Verdier quotient functor identifying the homotopy
category of finitely presented CDG\+modules\/ $\sH^0(\bmodr B^\cu)$
with the quotient
category of the bounded derived category of complexes of CDG\+modules\/
$\sD^\bb(\sZ^0(\bmodr B^\cu)_\gs)$ by the thick subcategory of
bounded complexes of contractible CDG\+modules. \par
\textup{(b)} Let $B^\cu=(B^*,d,h)$ be a CDG\+ring such that
the graded ring $B^*$ is graded right coherent.
 Then the same totalization functor\/ $\Tot$ induces a triangulated
functor
\begin{equation} \label{Verdier-quotient-bounded-to-absolute-derived}
 \sD^\bb(\sZ^0(\bmodr B^\cu))\lrarrow
 \sD^\abs(\bmodr B^\cu),
\end{equation}
which is a Verdier quotient functor identifying the absolute
derived category of CDG\+modules\/ $\sD^\abs(\bmodr B^\cu)$ with
the quotient category of the bounded derived category of complexes
of CDG\+modules\/ $\sD^\bb(\sZ^0(\bmodr B^\cu))$ by the thick
subcategory spanned by contractible CDG\+modules viewed as objects
of\/ $\sZ^0(\bmodr B^\cu)$.
\end{prop}

\begin{proof}
 This is a finitely presented CDG\+module version
of~\cite[Proposition~A.3\,(1\<2)]{Ef}.
 Part~(a) is the triangulated equivalence of~(1) and~(3) in
Lemma~\ref{stable-category-lemma} applied to the Frobenius exact
category $\sE=\sZ^0(\bmodr B^\cu)_\gs$ from
Lemma~\ref{graded-split-exact-structure-lemma}.
 The triangulated functor induced by the totalization functor is
quasi-inverse to the functor from (1) to~(3) mentioned in
the proof of Lemma~\ref{stable-category-lemma}.
 Part~(b) follows from part~(a) by a passage to the Verdier
quotients on the two sides of a triangulated equivalence.
\end{proof}

\subsection{Totalized dual finitely generated graded-projective
resolutions} \label{totalized-dual-construction-subsecn}
 The aim of this section is to construct, for a CDG\+ring $B^\cu$
with a graded right coherent underlying graded ring $B^*$,
a contravariant triangulated functor
$$
 \widetilde\Xi_{B^\cu}\:\sD^\bb(\sZ^0(\bmodr B^\cu))^\sop\lrarrow
 \sD^\bctr(B^\cu\bModl).
$$
 Then in the rest of
Section~\ref{compact-generators-of-contraderived-secn} we will
show that the functor $\widetilde\Xi_{B^\cu}$ factorizes through
the absolute derived category $\sD^\abs(\bmodr B^\cu)$, the resulting
triangulated functor $\Xi_{B^\cu}\:\sD^\abs(\bmodr B^\cu)^\sop\rarrow
\sD^\bctr(B^\cu\bModl)$ is fully faithful, and the essential image
of $\Xi_{B^\cu}$ is a set of compact generators of
$\sD^\bctr(B^\cu\bModl)$.

 In order to be able to prove in a natural way that the functor
$\Xi_{B^\cu}$ is triangulated, we make the functor
$\widetilde\Xi_{B^\cu}$ act from the triangulated category
$\sD^\bb(\sZ^0(\bmodr B^\cu))^\sop$ rather than just from
the abelian category $\sZ^0(\bmodr B^\cu)^\sop$ (which would be
otherwise sufficient).
 Then we will use the technology of
Section~\ref{complexes-of-cdg-modules-subsecn}.

 Let $B^\cu=(B^*,d,h)$ be a CDG\+ring.
 Consider the abelian category of left CDG\+mod\-ules
$\sZ^0(B^\cu\bModl)\simeq B^*[\delta]\sModl$, and form
the abelian category $\sC(\sZ^0(B^\cu\bModl))$ of complexes
in $\sZ^0(B^\cu\bModl)$.
 Here $\sZ^0(B^\cu\bModl)$ is viewed as just an abelian category.
{\emergencystretch=0em\hfuzz=2pt\par}

 Then the direct sum totalization, as constructed in
Section~\ref{totalizations-subsecn}, is an exact functor
$$
 \Tot^\sqcup\:\sC(\sZ^0(B^\cu\bModl))\lrarrow\sZ^0(B^\cu\bModl)
$$
preserving infinite direct sums.
 Furthermore, one easily observes that the functor $\Tot^\sqcup$
takes shifts and cones in the category of complexes 
$\sC(\sZ^0(B^\cu\bModl))$ to shifts and cones in the category of
CDG\+modules $\sZ^0(B^\cu\bModl)$.

 In particular, the functor $\Tot^\sqcup$ takes the cones of identity
endomorphisms of complexes in $\sZ^0(B^\cu\bModl)$ to the cones of
identity endomorphisms of CDG\+modules.
 It follows that the functor $\Tot^\sqcup$ takes homotopic morphisms
of complexes to homotopic morphisms of CDG\+modules, and induces
a triangulated functor
\begin{equation} \label{totalization-between-homotopy-categories-eqn}
 \Tot^\sqcup\:\sK(\sZ^0(B^\cu\bModl))\lrarrow\sH^0(B^\cu\bModl).
\end{equation}
 Here $\sK(\sZ^0(B^\cu\bModl))$ is the homotopy category of (unbounded) 
complexes in the abelian category $\sZ^0(B^\cu\bModl)$.

 Let $\sA$ be an exact category.
 A full subcategory $\sF\subset\sA$ is said to be \emph{resolving}
if $\sF$ is closed under extensions and kernels of admissible 
epimorphisms in $\sA$, and for every object $A\in\sA$
there exists an object $F\in\sF$
together with an admissible epimorphism $F\rarrow A$ in~$\sA$.
 We refer to the paper~\cite{Neem-de} for the definition of
the (bounded and unbounded) \emph{derived categories} of an exact
category.
 The following lemma is well-known.

\begin{lem} \label{resolving-subcategory-lemma}
 Let\/ $\sA$ be an exact category and\/ $\sF\subset\sA$ be a resolving
subcategory.
 Endow\/ $\sF$ with the inherited exact category structure.
 Then the triangulated functor between the bounded above derived
categories
$$
 \sD^-(\sF)\lrarrow\sD^-(\sA)
$$
induced by the inclusion of exact categories\/ $\sF\rarrow\sA$ is
an equivalence of triangulated categories.
\end{lem}

\begin{proof}
 See~\cite[Proposition~13.2.2(i)]{KS}
or~\cite[Proposition~A.3.1(a)]{Pcosh}.
\end{proof}

 We will apply Lemma~\ref{resolving-subcategory-lemma} in
the following context.
Let $B^\cu=(B^*,d,h)$ be a CDG\+ring such that the graded
ring $B^*$ is graded right coherent.
Let $\bmodr B^\cu\subset\bModr B^\cu$ denote
the full DG\+subcategory of finitely presented right CDG\+modules
(as in Section~\ref{prelim-finite-countable-modules-subsecn}).
 Let $\sA=\sZ^0(\bmodr B^\cu)$ be the abelian category of finitely
presented right CDG\+modules over~$B^\cu$.
 Consider the full DG\+subcategory of graded-projective CDG\+modules
$\bModrproj B^\cu$ in the DG\+category of right CDG\+modules
$\bModr B^\cu$, as per the discussion in
Section~\ref{prelim-proj-inj-flatness-subsecn}.
 Let
$$
 \bmodrproj B^\cu=\bmodr B^\cu\cap\bModrproj B^\cu\subset\bModr B^\cu
$$ 
be the full DG\+subcategory of finitely generated graded-projective
right CDG\+mod\-ules, and let $\sF\subset\sA$ be the full subcategory
of finitely generated graded-projective right CDG\+modules in
the abelian category $\sZ^0(\bmodr B^\cu)$,
$$
 \sF=\sZ^0(\bmodrproj B^\cu)=\sZ^0(\bmodr B^\cu\cap\bModrproj B^\cu).
$$
 Recall that all finitely generated graded-projective CDG\+modules
are finitely presented by
Corollary~\ref{graded-projective-generated-presented-cor}(a).

\begin{lem} \label{graded-projective-resolving-lemma}
\textup{(a)} For any CDG\+ring $B^\cu=(B^*,d,h)$, the full
subcategory of graded-projective CDG\+modules\/
$\sZ^0(B^\cu\bModl_\bproj)$ is resolving in the abelian category
of left CDG\+modules\/ $\sZ^0(B^\cu\bModl)$. \par
\textup{(b)} For any CDG\+ring $B^\cu=(B^*,d,h)$ with a graded
right coherent graded ring $B^*$, the full subcategory of finitely
generated graded-projective CDG\+modules\/
$\sF=\sZ^0(\bmodrproj B^\cu)$ is resolving in the abelian category
of finitely presented right CDG\+modules\/ $\sA=\sZ^0(\bmodr B^\cu)$.
\end{lem}

\begin{proof}
 It is obvious that the property of a CDG\+module to be
graded-projective is preserved by extensions and kernels of
surjective morphisms.
 So it remains to explain why any CDG\+module is a quotient of
a graded-projective one.
 In both the contexts of~(a) and~(b), this holds because there are
enough projective objects in the respective abelian category of
CDG\+modules, and all projective CDG\+modules are graded-projective.
 For part~(a), this was explained in
Section~\ref{prelim-proj-inj-flatness-subsecn}.

 For part~(b), the situation it similar.
 It suffices to say that the graded ring $B^*[\delta]$ is graded right
coherent by~\cite[Corollary~8.9(b)]{PS5}, so there are enough
finitely generated projective graded modules in the abelian category
of finitely presented graded right $B^*[\delta]$\+modules
$\sZ^0(\bmodr B^\cu)=\smodr B^*[\delta]$ (see
Section~\ref{prelim-finite-countable-modules-subsecn}).
 Furthermore, any (finitely generated) projective graded right
$B^*[\delta]$\+module is also (finitely generated) projective
as a graded module over $B^*$, as the graded right $B^*$\+module
$B^*[\delta]$ is projective/free with two homogeneous generators.

 Explicitly, let $N^\cu$ be a finitely presented graded right
CDG\+module over~$B^\cu$.
 Let $P^*$ be a finitely generated projective graded right $B^*$\+module
endowed with a surjective graded $B^*$\+module map $P^*\rarrow N^*$.
 Then the morphism of CDG\+modules $G^+(P^*)\rarrow N^\cu$
corresponding by adjunction to the morphism of graded modules
$P^*\rarrow N^*$ represents $N^*$ as an epimorphic image of
a finitely generated projective CDG\+module $G^+(P^*)$
(see Section~\ref{prelim-delta-and-G-subsecn}).
\end{proof}

 By Lemma~\ref{graded-projective-resolving-lemma}(b),
\,$\sF$ is a resolving subcategory in $\sA$.
 So Lemma~\ref{resolving-subcategory-lemma} tells us that
\begin{equation} \label{fin-gen-graded-proj-derived-eqn}
 \sD^-(\sZ^0(\bmodrproj B^\cu))\simeq\sD^-(\sZ^0(\bmodr B^\cu)).
\end{equation}
 This is the conclusion we make from the two previous lemmas.

 As mentioned at the end of Section~\ref{prelim-cdg-modules-subsecn},
any CDG\+ring $B^\cu$ is (obviously) a CDG\+bi\-module over itself,
in the sense of the definition in
Section~\ref{prelim-cdg-bimodules-subsecn}.
 Therefore, for any right CDG\+module $N^\cu$ over $B^\cu$, the graded
Hom module $\Hom_{B^\rop{}^*}^*(N^*,B^*)$ has a natural structure of
a left CDG\+module $\Hom_{B^\rop{}^*}^\cu(N^\cu,B^\cu)$ over $B^\cu$,
as per the discussion in Section~\ref{prelim-cdg-bimodules-subsecn}.
 We will apply this construction to finitely generated
graded-projective right CDG\+modules $N^\cu$ over~$B^\cu$.

To that end, for any CDG\+ring $B^\cu$, consider the full
DG\+subcategory of graded-projective CDG\+modules $B^\cu\bModl_\bproj$
in the DG\+category of left CDG\+modules $B^\cu\bModl$ and the 
corresponding full subcategory $\sZ^0(B^\cu\bModl_\bproj)$ in the 
abelian category of left CDG\+modules $\sZ^0(B^\cu\bModl)$.
 As before, let $B^\cu\bmodl\subset B^\cu\bModl$ denote
the full DG\+subcategory of finitely presented left CDG\+modules, let
$$
 B^\cu\bmodl_\bproj=B^\cu\bmodl\cap B^\cu\bModl_\bproj
 \subset B^\cu\bModl
$$
denote the full DG\+subcategory of finitely generated
graded-projective left CDG\+mod\-ules in the DG\+category
$B^\cu\bModl$, and set
$$
 \sE=\sZ^0(B^\cu\bmodl_\bproj)=\sZ^0(B^\cu\bmodl\cap B^\cu\bModl_\bproj)
$$
to be the full subcategory of finitely generated graded-projective
CDG\+modules in the abelian category of left CDG\+modules
$\sZ^0(B^\cu\bModl)$.
 The full subcategory $\sE$ is closed under extensions in
$\sZ^0(B^\cu\bModl)$, so it inherits an exact category structure.
 Then the contravariant DG\+functor
$$
 \Hom_{B^\rop{}^*}^\cu({-},B^\cu)\:
 (\bmodrproj B^\cu)^\sop\lrarrow B^\cu\bmodl_\bproj
$$
is an anti-equivalence of DG\+categories, while the contravariant
additive functor
\begin{equation} \label{dualization-of-fin-gen-graded-proj-eqn}
 \Hom_{B^\rop{}^*}^\cu({-},B^\cu)\:
 (\sZ^0(\bmodrproj B^\cu))^\sop\lrarrow\sZ^0(B^\cu\bmodl_\bproj)
\end{equation}
is an anti-equivalence of exact categories $\sF^\sop\simeq\sE$.
 The anti-equivalence of exact
categories~\eqref{dualization-of-fin-gen-graded-proj-eqn} induces
an anti-equivalence of their unbounded derived categories
$$
 \Hom_{B^\rop{}^*}^\cu({-},B^\cu)\:\sD(\sZ^0(\bmodrproj B^\cu))^\sop
 \lrarrow\sD(\sZ^0(B^\cu\bmodl_\bproj)),
$$
which restricts to an anti-equivalence between the bounded above
and bounded below derived categories
\begin{equation} \label{triangulated-dualization-functor-eqn}
 \Hom_{B^\rop{}^*}^\cu({-},B^\cu)\:\sD^-(\sZ^0(\bmodrproj B^\cu))^\sop
 \lrarrow\sD^+(\sZ^0(B^\cu\bmodl_\bproj)).
\end{equation}

 Furthermore, we observe that any short exact sequence of
graded-projective CDG\+modules is graded-split.
 Therefore, the restriction of the totalization
functor~\eqref{totalization-between-homotopy-categories-eqn} to
the homotopy category of bounded below complexes of
graded-projective CDG\+modules
$$
 \Tot^\sqcup\:\sK^+(\sZ^0(B^\cu\bModl_\bproj))
 \lrarrow\sH^0(B^\cu\bModl_\bproj)
$$
factorizes through the bounded below derived category of
the exact category $\sZ^0(B^\cu\bModl_\bproj)$ and induces
a triangulated functor
$$
 \Tot^\sqcup\:\sD^+(\sZ^0(B^\cu\bModl_\bproj))
 \lrarrow\sH^0(B^\cu\bModl_\bproj),
$$
in view of Lemma~\ref{contractible-totalization-lemma}(a).
 Restricting further to bounded below complexes of finitely generated
graded-projective CDG\+modules, we obtain a triangulated functor
\begin{equation} \label{totalization-acting-from-derived}
 \Tot^\sqcup\:\sD^+(\sZ^0(B^\cu\bmodl_\bproj))
 \lrarrow\sH^0(B^\cu\bModl_\bproj).
\end{equation}

 Finally, we consider the composition of triangulated functors
\begin{multline} \label{long-composition-eqn}
 \sD^\bb(\sZ^0(\bmodr B^\cu))^\sop\,\rightarrowtail\,
 \sD^-(\sZ^0(\bmodr B^\cu))^\sop \\
 \,\simeq\,\sD^-(\sZ^0(\bmodrproj B^\cu))^\sop
 \,\simeq\,\sD^+(\sZ^0(B^\cu\bmodl_\bproj)) \\ \lrarrow
 \sH^0(B^\cu\bModl_\bproj)\,\simeq\,\sD^\bctr(B^\cu\bModl).
\end{multline}
 Here the first functor $\sD^\bb(\sZ^0(\bmodr B^\cu))^\sop
\rightarrowtail \sD^-(\sZ^0(\bmodr B^\cu))^\sop$ is opposite to
the natural triangulated inclusion of the bounded derived category
of the abelian category $\sZ^0(\bmodr B^\cu)$ into the bounded
above derived category.
 The second functor $\sD^-(\sZ^0(\bmodr B^\cu))^\sop\rarrow
\sD^-(\sZ^0(\bmodrproj B^\cu))^\sop$ is opposite to the triangulated
equivalence~\eqref{fin-gen-graded-proj-derived-eqn} provided by
Lemma~\ref{resolving-subcategory-lemma}.
 The third functor $\Hom_{B^\rop{}^*}^\cu({-},B^\cu)\:\allowbreak
\sD^-(\sZ^0(\bmodrproj B^\cu))^\sop\rarrow
\sD^+(\sZ^0(B^\cu\bmodl_\bproj))$ is the triangulated
anti-equiv\-a\-lence~\eqref{triangulated-dualization-functor-eqn}.
 The fourth functor $\Tot^\sqcup\:
\sD^+(\sZ^0(B^\cu\bmodl_\bproj))\rarrow\sH^0(B^\cu\bModl_\bproj)$ is
the totalization~\eqref{totalization-acting-from-derived}.
 The fifth functor $\sH^0(B^\cu\bModl_\bproj)\rarrow
\sD^\bctr(B^\cu\bModl)$ is the triangulated equivalence of
Theorem~\ref{contraderived-cotorsion-pair-category-theorem}(b).
{\emergencystretch=1em\par}

 The composition of functors~\eqref{long-composition-eqn} is
the desired triangulated functor
\begin{equation} \label{functor-tilde-Xi-on-bounded-derived-category}
 \widetilde\Xi_{B^\cu}\:\sD^\bb(\sZ^0(\bmodr B^\cu))^\sop
 \lrarrow\sD^\bctr(B^\cu\bModl).
\end{equation}
 Explicitly, the contravariant functor $\widetilde\Xi_{B^\cu}$ assigns
to a finite complex of finitely presented right CDG\+modules
$N^{\bu,\cu}$ over $B^\cu$ the left CDG\+module
$$
 \widetilde\Xi(N^{\bu,\cu})=
 \Tot^\sqcup(\Hom_{B^\rop{}^*}^\cu(P^{\bu,\cu},B^\cu)).
$$
 Here
$$
 P^{\bu,\cu}=(\dotsb\rarrow P^{-2,\cu}\rarrow P^{-1,\cu}
 \rarrow P^{0,\cu}\rarrow P^{1,\cu}\rarrow\dotsb)
$$
is a bounded above complex of finitely generated graded-projective
right CDG\+mod\-ules $P^{n,\cu}$, \,$n\in\boZ$, mapping
quasi-isomorphically into $N^{\bu,\cu}$ (in the sense of
a quasi-isomorphism of complexes in the abelian category
$\sZ^0(\bmodr B^\cu)$).
 The complex of finitely generated graded-projective left CDG\+modules
$$
 \Hom_{B^\rop{}^*}^\cu(P^{\bu,\cu},B^\cu)=
 (\dotsb\rarrow\Hom_{B^\rop{}^*}^\cu(P^{0,\cu},B^\cu)\rarrow
 \Hom_{B^\rop{}^*}^\cu(P^{-1,\cu},B^\cu)\rarrow\dotsb)
$$
is the result of termwise application of the dualization functor
$\Hom_{B^\rop{}^*}^\cu({-},B^\cu)$ to the complex of finitely generated 
graded-projective right CDG\+modules~$P^{\bu,\cu}$.

\subsection{Hom from the totalized dual resolution computed}
 The aim of this section is to prove the following proposition, and
deduce a corollary defining the contravariant triangulated functor
$$
 \Xi_{B^\cu}\:\sD^\abs(\bmodr B^\cu)^\sop
 \lrarrow\sD^\bctr(B^\cu\bModl).
$$

\begin{prop} \label{Hom-from-tilde-Xi-prop}
 Let $B^\cu=(B^*,d,h)$ be a CDG\+ring such that the graded ring $B^*$ is
graded right coherent.
 Let $N^\cu$ be a finitely presented right CDG\+module over $B^\cu$ and
$M^\cu$ be a left CDG\+module over~$B^\cu$.
 Then there is a natural isomorphism of abelian groups
$$
 \Hom_{\sD^\bctr(B^\cu\bModl)}(\widetilde\Xi_{B^\cu}(N^\cu),M^\cu)
 \simeq H^0(N^\cu\ot_{B^*}^\boL M^\cu).
$$
 Here\/ $\widetilde\Xi_{B^\cu}$ is the triangulated
functor~\eqref{functor-tilde-Xi-on-bounded-derived-category} from
Section~\ref{totalized-dual-construction-subsecn} (applied
to the one-term complex of CDG\+modules $N^\cu\in\sZ^0(\bmodr B^\cu)
\subset\sD^\bb(\sZ^0(\bmodr B^\cu))$) and\/ $\ot_{B^*}^\boL$ is
the derived tensor product functor~\eqref{derived-tensor-product} from
Section~\ref{tensor-product-pairing-subsecn}.
\end{prop}

\begin{proof}
 We start with two elementary observations of natural isomorphisms.
 For any finitely generated graded-projective right CDG\+module $P^\cu$
over a CDG\+ring $B^\cu$ and any left CDG\+module $G^\cu$ over $B^\cu$
there is a natural isomorphism of complexes of abelian groups
\begin{equation} \label{hom-from-hom-from-fin-gen-graded-proj-eqn}
 \Hom^\bu_{B^*}(\Hom^\cu_{B^\rop{}^*}(P^\cu,B^\cu),G^\cu)
 \simeq P^\cu\ot_{B^*}G^\cu.
\end{equation}
 For any complex of left CDG\+modules $Q^{\bu,\cu}$ and any left
CDG\+module $G^\cu$ over $B^\cu$, there is a natural isomorphism of
complexes of abelian groups
\begin{equation} \label{hom-from-direct-sum-totalization-eqn}
 \Hom^\bu_{B^*}(\Tot^\sqcup(Q^{\bu,\cu}),G^\cu)\simeq
 \Tot^\sqcap(\Hom^\bu_{B^*}(Q^{\bu,\cu},G^\cu)).
\end{equation}
 Here $\Hom^\bu_{B^*}(Q^{\bu,\cu},G^\cu)$ is the complex of complexes
of abelian groups whose terms are the complexes
$\Hom^\bu_{B^*}(Q^{n,\cu},G^\cu)$, \,$n\in\boZ$.

 Furthermore, we observe that, by the definition of the contraderived
category given in Section~\ref{contraderived-subsecn}, a natural
isomorphism of abelian groups
\begin{equation} \label{contraderived-hom-from-graded-projective-eqn}
 \Hom_{\sD^\bctr(B^\cu\bModl)}(Q^\cu,G^\cu)\simeq
 \Hom_{\sH^0(B^\cu\bModl)}(Q^\cu,G^\cu)=
 H^0\Hom_{B^*}^\bu(Q^\cu,G^\cu)
\end{equation}
holds for any graded-projective left CDG\+module $Q^\cu$ and any
left CDG\+module $G^\cu$ over~$B^\cu$.
 Moreover, for any left CDG\+module $M^\cu$ over $B^\cu$ there exists
a graded-flat (or even graded-projective) left CDG\+module $F^\cu$
over $B^\cu$ together with a closed morphism of left CDG\+modules
$F^\cu\rarrow M^\cu$ with a contraacyclic cone
(see Theorem~\ref{contraderived-cotorsion-pair-category-theorem}).
 Then
the isomorphism~\eqref{contraderived-hom-from-graded-projective-eqn}
for $G^\cu=F^\cu$ implies a natural isomorphism
\begin{equation} \label{contraderived-hom-graded-proj-flat-eqn}
 \Hom_{\sD^\bctr(B^\cu\bModl)}(Q^\cu,M^\cu)\simeq
 H^0\Hom_{B^*}^\bu(Q^\cu,F^\cu).
\end{equation}

 Returning to the situation at hand, choose a resolution $P_\bu^\cu$
of the finitely presented right CDG\+module $N^\cu$ over $B^\cu$ by
finitely generated graded-projective CDG\+modules $P_n^\cu$, \,$n\ge0$
in the abelian category $\sZ^0(\bmodr B^\cu)$,
\begin{equation} \label{fin-gen-graded-projective-resolution-eqn}
 \dotsb\lrarrow P_2^\cu\lrarrow P_1^\cu\lrarrow P_0^\cu\lrarrow N^\cu
 \lrarrow0.
\end{equation}
 Put $Q^{n,\cu}=\Hom_{B^\rop{}^*}^\cu(P_n^\cu,B^\cu)$ and
$Q^\cu=\Tot^\sqcup(Q^{\bu,\cu})$.
 So $Q^{\bu,\cu}$ is a bounded below complex of finitely generated
graded-projective left CDG\+modules over $B^\cu$ and $Q^\cu$ is
a (countably generated) graded-projective left CDG\+module over~$B^\cu$.
 By the definition of the functor $\widetilde\Xi_{B^\cu}$, we have
$\widetilde\Xi_{B^\cu}(N^\cu)=Q^\cu$.
 
 Now the formulas~\eqref{contraderived-hom-graded-proj-flat-eqn},
\eqref{hom-from-direct-sum-totalization-eqn},
and~\eqref{hom-from-hom-from-fin-gen-graded-proj-eqn} tell us that
\begin{multline*}
 \Hom_{\sD^\bctr(B^\cu\bModl)}(Q^\cu,M^\cu)\simeq
 H^0\Hom_{B^*}^\bu(Q^\cu,F^\cu) \\ \simeq
 H^0\Tot^\sqcap(\Hom^\bu_{B^*}(Q^{\bu,\cu},F^\cu))
 \simeq H^0\Tot^\sqcap(P_\bu^\cu\ot_{B^*}F^\cu).
\end{multline*}

 Finally, the bounded above complex of complexes of abelian groups
\begin{equation} \label{acyclic-complex-of-tensor-products-eqn}
 \dotsb\lrarrow P_2^\cu\ot_{B^*}F^\cu\lrarrow P_1^\cu\ot_{B^*}F^\cu
 \lrarrow P_0^\cu\ot_{B^*}F^\cu\lrarrow N^\cu\ot_{B^*}F^\cu
 \lrarrow0
\end{equation}
is acyclic, since the complex of right
CDG\+modules~\eqref{fin-gen-graded-projective-resolution-eqn} is
acyclic in $\sZ^0(\bModr B^\cu)$ and the left CDG\+module $F^\cu$
is graded-flat.
 By Lemma~\ref{co-contra-acyclic-total-complex-lemma}(b),
the direct product totalization of
the bicomplex~\eqref{acyclic-complex-of-tensor-products-eqn} is
an acyclic complex of abelian groups.
 Hence the induced map of abelian groups
$$
 H^0\Tot^\sqcap(P_\bu^\cu\ot_{B^*}F^\cu)\lrarrow
 H^0(N^\cu\ot_{B^*}F^\cu)
$$
is an isomorphism.
 It remains to recall that the complex $N^\cu\ot_{B^*}F^\cu$
represents the derived category object $N^\cu\ot_{B^*}^\boL M^\cu$
as per the construction of the derived functor~$\ot_{B^*}^\boL$
in the formula~\eqref{derived-tensor-product-constructed}.
\end{proof}

\begin{cor} \label{factorizes-through-absolute-derived-cor}
 The contravariant triangulated
functor~\eqref{functor-tilde-Xi-on-bounded-derived-category}
$$
 \widetilde\Xi_{B^\cu}\:\sD^\bb(\sZ^0(\bmodr B^\cu))^\sop
 \lrarrow\sD^\bctr(B^\cu\bModl)
$$
from Section~\ref{totalized-dual-construction-subsecn} factorizes
through the triangulated Verdier quotient
functor~\eqref{Verdier-quotient-bounded-to-absolute-derived}
$$
 \sD^\bb(\sZ^0(\bmodr B^\cu))\lrarrow
 \sD^\abs(\bmodr B^\cu)
$$
from Proposition~\ref{bounded-derived-and-absolute-derived-prop}(b) and,
consequently, induces a well-defined contravariant triangulated functor
\begin{equation} \label{functor-Xi-on-absolute-derived-category}
 \Xi_{B^\cu}\:\sD^\abs(\bmodr B^\cu)^\sop
 \lrarrow\sD^\bctr(B^\cu\bModl).
\end{equation}
\end{cor}

\begin{proof}
 According to
Proposition~\ref{bounded-derived-and-absolute-derived-prop}(b),
we only need to check that $\widetilde\Xi_{B^\cu}(N^\cu)=0$ in
$\sD^\bctr(B^\cu\bModl)$ for any contractible finitely presented
right CDG\+module $N^\cu$ over $B^\cu$ (viewed as an object of
$\sZ^0(\bmodr B^\cu)\subset\sD^\bb(\sZ^0(\bmodr B^\cu))$).
 This follows immediately from Proposition~\ref{Hom-from-tilde-Xi-prop}
(as any contractible CDG\+module vanishes in the coderived category).
 Alternatively, one can give a direct proof along the following lines.
 By~\cite[Lemma~3.3]{PS5}, \,$N^\cu$ is a direct summand of
the cone of the identity endomorphism of some finitely presented
right CDG\+module $M^\cu$ over~$B^\cu$.
 Choose a resolution $P_\bu^\cu$ of $M^\cu$ by finitely generated
graded-projective right CDG\+modules $P_n^\cu$, and apply
the functor of the cone of the identity endomorphism to every
CDG\+module $P_n^\cu$, producing a resolution of
$N^\cu=\cone(\id_{M^\cu})$ by finitely generated graded-projective
right CDG\+modules $\cone(\id_{P_n^\cu})$.
 Following the construction of the functor $\widetilde\Xi_{B^\cu}$,
show that the cones of the identity endomorphisms are transformed into
the cocones of the identity endomorphisms by this contravariant functor.
\end{proof}

\begin{rem} \label{RHom-from-tilde-Xi-prop}
The proof of Proposition~\ref{Hom-from-tilde-Xi-prop} (together with
the result of Corollary~\ref{factorizes-through-absolute-derived-cor})
in fact tells us that there is a natural isomorphism
$\boR\Hom_{B^*}^\bu(\Xi_{B^\cu}(N^\cu),M^\cu)\allowbreak
\simeq N^\cu\ot_{B^*}^\boL M^\cu$ of functors
$$
 \sD^\abs(\bmodr B^\cu)\times\sD^\bctr(B^\cu\bModl)
 \lrarrow\sD(\Ab).
$$
Here, $\boR\Hom_{B^*}^\bu$ is the right derived functor of $\Hom^\bu_{B^*}$
with respect to either of the contraderived model structures on
$\sZ^0(B^\cu\bModl)$.
\end{rem}

\subsection{The contraderived category is compactly generated}
 In this section we prove one of our main results,
Theorem~\ref{image-of-Xi-compact-generators-thm}.

\begin{lem} \label{derived-tensor-product-preserves-coproducts}
 For any CDG\+ring $B^\cu=(B^*,d,h)$, the derived tensor product
functor~\eqref{derived-tensor-product}
$$
 \ot_{B^*}^\boL\:\sD^\bco(\bModr B^\cu)\times\sD^\bctr(B^\cu\bModl)
 \lrarrow\sD(\Ab)
$$
preserves coproducts in both of its arguments.
\end{lem}

\begin{proof}
 Here it needs to be explained why the coproducts exist in the coderived
and contraderived categories, and how they are constructed.
 The infinite coproducts do exist in both the coderived and
the contraderived category, but their constructions are different.

 To construct coproducts in the coderived category
$\sD^\bco(\bModr B^\cu)$, one needs to represent it as
the \emph{Verdier quotient category} of the homotopy category,
$$
 \sD^\bco(\bModr B^\cu)=\sH^0(\bModr B^\cu)/\sH^0(\bModr B^\cu)_\ac^\bco
$$
(see Section~\ref{coderived-subsecn}).
 Then the claim is that the Verdier quotient functor
$\sH^0(\bModr B^\cu)\allowbreak\rarrow\sD^\bco(\bModr B^\cu)$ preserves
coproducts~\cite[paragraph before Corollary~7.15]{PS5}.
 The coproducts in the homotopy category $\sH^0(\bModr B^\cu)$, of
course, agree with those in the abelian category $\sZ^0(\bModr B^\cu)
=\sModr B^*[\delta]$.

 To construct coproducts in the contraderived category
$\sD^\bctr(\bModr B^\cu)$, one needs to represent it as
the \emph{full triangulated subcategory} of the homotopy category,
$$
 \sD^\bctr(B^\cu\bModl)\simeq\sH^0(B^\cu\bModl_\bproj)
$$
(see Theorem~\ref{contraderived-cotorsion-pair-category-theorem}(b)).
 Then it suffices to observe that the full subcategory
$\sH^0(B^\cu\bModl_\bproj)$ is closed under coproducts in
$\sH^0(B^\cu\bModl)$.
 So the claim is that the embedding $\sD^\bctr(B^\cu\bModl)\rarrow
\sH^0(B^\cu\bModl)$ preserves coproducts.

 Alternatively, one can use the construction of the flat contraderived
category,
$$
 \sD^\bctr(B^\cu\bModl)\simeq
 \sH^0(B^\cu\bModl_\bflat)/\sH^0(B^\cu\bModl)_\flat.
$$
(see Theorem~\ref{projective-and-flat-contraderived-theorem}).
 Then the observation is that both the full subcategories
of flat and graded-flat CDG\+modules $\sH^0(B^\cu\bModl)_\flat$
and $\sH^0(B^\cu\bModl_\bflat)$ are closed under coproducts in
the homotopy category $\sH^0(B^\cu\bModl)$.
 Consequently, the Verdier quotient functor
$\sH^0(B^\cu\bModl_\bflat)\rarrow\sD^\bctr(B^\cu\bModl)$
preserves coproducts.

 With these observations in mind, the claim that the derived
tensor product functor~$\ot_{B^*}^\boL$ preserves coproducts in
both of the arguments follows immediately from its
definition~\eqref{derived-tensor-product-constructed} and
the fact that the underived functor of tensor product of CDG\+modules
preserves infinite direct sums.
\end{proof}

\begin{lem} \label{flat-cdg-modules-acyclic-tensor-product-lemma}
 Let $B^\cu=(B^*,d,h)$ be a CDG\+ring and
$F^\cu\in\sZ^0(B^\cu\bModl_\bflat)$ be a graded-flat left CDG\+module
over~$B^\cu$.
 Then the following three conditions are equivalent:
\begin{enumerate}
\item the complex of abelian groups $N^\cu\ot_{B^*}F^\cu$ is acyclic
for every right CDG\+module $N^\cu$ over~$B^\cu$;
\item the complex of abelian groups $N^\cu\ot_{B^*}F^\cu$ is acyclic
for every finitely presented right CDG\+module $N^\cu$ over~$B^\cu$;
\item the left CDG\+module $F^\cu$ over $B^\cu$ is flat.
\end{enumerate}
\end{lem}

\begin{proof}
 The equivalence (1)\,$\Longleftrightarrow$\,(2) holds because all
the right CDG\+modules over $B^\cu$ are direct limits of finitely
presented ones (in the abelian category $\sZ^0(\bModr B^\cu)=
\sModr B^*[\delta]$) and tensor products of CDG\+modules commute with
direct limits.
 The implication (3)\,$\Longrightarrow$\,(1) was already explained
as the assertion~(1) in Section~\ref{tensor-product-pairing-subsecn}.

 Let us prove the implication (1)\,$\Longrightarrow$\,(3).
 Given a right CDG\+module $N^\cu$ over $B^\cu$, we need to show that
$\Tor^1_{B^*[\delta]}(N^\cu,F^\cu)=0$.
 Equivalently, it suffices to show that
$\Ext^1_{B^*[\delta]\sModl}(F^\cu,\Hom_{\boZ}^\cu(N^\cu,\boQ/\boZ))=0$.
 Now the underlying graded $B^*$\+module of the left CDG\+module
$\Hom_{\boZ}^\cu(N^\cu,\boQ/\boZ)$ over $B^\cu$ is pure-injective,
hence it is cotorsion; while, by assumption, the underlying graded
$B^*$\+module of the left CDG\+module $F^\cu$ over $B^\cu$ is flat.
 By Lemma~\ref{Ext-1-homotopy-hom-lemma}, we have
\begin{multline*}
 \Ext^1_{B^*[\delta]\sModl}(F^\cu,\Hom_{\boZ}^\cu(N^\cu,\boQ/\boZ))
 \simeq
 \Hom_{\sH^0(B^\cu\bModl)}(F^\cu,\Hom_{\boZ}^\cu(N^\cu,\boQ/\boZ)[1])
 \\ = H^1\Hom_{B^*}^\bu(F^\cu,\Hom_{\boZ}^\cu(N^\cu,\boQ/\boZ))
 \simeq\Hom_{\boZ}(H^{-1}(N^\cu\ot_{B^*}F^\cu),\>\boQ/\boZ).
\end{multline*}
 Finally, $H^{-1}(N^\cu\ot_{B^*}F^\cu)=0$ by~(1), and we are done.
\end{proof}

 Now we can prove the following theorem, which is the first main result
of Section~\ref{compact-generators-of-contraderived-secn}.

\begin{thm} \label{image-of-Xi-compact-generators-thm}
 Let $B^\cu=(B^*,d,h)$ be a CDG\+ring such that the graded ring $B^*$
is graded right coherent.
 Then the contraderived category of left CDG\+modules\/
$\sD^\bctr(B^\cu\bModl)$ is compactly generated.
 The image of the contravariant triangulated functor\/~$\Xi_{B^\cu}$
\,\eqref{functor-Xi-on-absolute-derived-category} from
Corollary~\ref{factorizes-through-absolute-derived-cor} consists of
compact objects in\/ $\sD^\bctr(B^\cu\bModl)$ and generates this
triangulated category.
\end{thm}

\begin{proof}
 First of all, it needs to be mentioned that infinite coproducts exist
in the contraderived category $\sD^\bctr(B^\cu\bModl)$.
 This was explained in the proof of
Lemma~\ref{derived-tensor-product-preserves-coproducts}.
 Now the object $\Xi_{B^\cu}(N^\cu)\in\sD^\bctr(B^\cu\bModl)$ is
compact for any $N^\cu\in\sD^\abs(\bmodr B^\cu)$ in view of
the computation of the functor
$\Hom_{\sD^\bctr(B^\cu\bModl)}(\widetilde\Xi_{B^\cu}(N^\cu),{-})$ in
Proposition~\ref{Hom-from-tilde-Xi-prop} and the assertion that
the derived tensor product functor~$\ot_{B^*}^\boL$ preserves
coproducts in its second argument
(Lemma~\ref{derived-tensor-product-preserves-coproducts}).

 Finally, it is clear that there is only a set of objects, up to
isomorphism, in the image of the functor~$\Xi_{B^\cu}$.
 So, in view of the discussion in
Section~\ref{compact-generators-of-coderived-subsecn}, it suffices
to show that the image of $\Xi_{B^\cu}$ weakly generates
the contraderived category.
 Now let $M^\cu\in\sD^\bctr(B^\cu\bModl)$ be an object such that
$\Hom_{\sD^\bctr(B^\cu\bModl)}(\Xi_{B^\cu}(N^\cu),M^\cu)=0$
for all $N^\cu\in\sD^\abs(\bmodr B^\cu)$.
 By Proposition~\ref{Hom-from-tilde-Xi-prop}, this means that
the $N^\cu\ot_{B^*}^\boL M^\cu=0$ in the derived category of
abelian groups $\sD(\Ab)$ for all finitely presented right
CDG\+modules $N^\cu$ over~$B^\cu$.
 Let $F^\cu$ be a graded-flat left CDG\+module over $B^\cu$ together
with a closed morphism of CDG\+modules $F^\cu\rarrow M^\cu$ with
contraacyclic cone.
 Then the complex $N^\cu\ot_{B^*}F^\cu$ is acyclic for all
$N^\cu\in\sH^0(\bmodr B^\cu)$.
 By Lemma~\ref{flat-cdg-modules-acyclic-tensor-product-lemma}, it
follows that the CDG\+module $F^\cu$ is flat.
 Applying Corollary~\ref{flats-are-contraacyclic-cor}, we conclude
that the CDG\+module $F^\cu$ is contraacyclic; hence $M^\cu$ is
a zero object of $\sD^\bctr(B^\cu\bModl)$.
\end{proof}

\subsection{Hom between totalized dual resolutions computed}
 The aim of this section is to prove that the contravariant
triangulated functor $\Xi_{B^\cu}$
\,\eqref{functor-Xi-on-absolute-derived-category} from
Corollary~\ref{factorizes-through-absolute-derived-cor}
is fully faithful.

\begin{prop} \label{Hom-from-Xi-to-Xi-prop}
 Let $B^\cu=(B^*,d,h)$ be a CDG\+ring such that the graded ring $B^*$
is graded right coherent.
 Then, for any finitely presented right CDG\+modules $N^\cu$ and $K^\cu$
over $B^\cu$, there is a natural isomorphism of abelian groups
$$
 \Hom_{\sD^\bctr(B^\cu\bModl)}(\Xi_{B^\cu}(N^\cu),\Xi_{B^\cu}(K^\cu))
 \simeq
 \Hom_{\sD^\bco(\bModr B^\cu)}(K^\cu,N^\cu).
$$
\end{prop}

\begin{proof}
 Let
$$
 \dotsb\lrarrow P_2^\cu\lrarrow P_1^\cu\lrarrow P_0^\cu\lrarrow N^\cu
 \lrarrow0
$$
(as in formula~\eqref{fin-gen-graded-projective-resolution-eqn}
in the proof of Proposition~\ref{Hom-from-tilde-Xi-prop}) and
\begin{equation} \label{fin-gen-graded-projective-resolution-2nd}
 \dotsb\lrarrow R_2^\cu\lrarrow R_1^\cu\lrarrow R_0^\cu\lrarrow K^\cu
 \lrarrow0
\end{equation}
be some chosen resolutions of the CDG\+modules $N^\cu$ and $K^\cu$
by finitely generated graded-projective CDG\+modules $P_n^\cu$ and
$R_n^\cu$, \,$n\ge0$ in the abelian category of finitely presented
right CDG\+modules $\sZ^0(\bmodr B^\cu)$.
 Consider the finitely generated graded-projective left CDG\+modules
$Q^{n,\cu}=\Hom_{B^\rop{}^*}(P_n^\cu,B^\cu)$ and
$S^{n,\cu}=\Hom_{B^\rop{}^*}(R_n^\cu,B^\cu)$; and denote the respective
totalizations of bounded below complexes of left CDG\+mod\-ules by
$Q^\cu=\Tot^\sqcup(Q^{\bu,\cu})$ and $S^\cu=\Tot^\sqcup(S^{\bu,\cu})$.
 So $Q^\cu$ and $S^\cu$ are (countably generated) graded-projective
left CDG\+modules over $B^\cu$ representing the contraderived category
objects $\Xi_{B^\cu}(N^\cu)=Q^\cu$ and $\Xi_{B^\cu}(K^\cu)=S^\cu$.

 Following the computations in the proof of
Proposition~\ref{Hom-from-tilde-Xi-prop} (for $F^\cu=M^\cu=S^\cu$),
we have a natural isomorphism of abelian groups
\begin{multline} \label{hom-of-Xi-as-tensor-product-eqn}
 \Hom_{\sD^\bctr(B^\cu\bModl)}(\Xi_{B^\cu}(N^\cu),\Xi_{B^\cu}(K^\cu))
 \\ = \Hom_{\sD^\bctr(B^\cu\bModl)}(Q^\cu,S^\cu)\simeq
 H^0(N^\cu\ot_{B^*}S^\cu).
\end{multline}

 Quite generally, for any finitely generated graded-projective
right CDG\+module $E^\cu$ over a CDG\+ring $B^\cu$ and any right
CDG\+module $G^\cu$ over $B^\cu$, there is a natural isomorphism of
complexes of abelian groups
\begin{equation} \label{tensor-with-hom-from-fin-gen-graded-proj-eqn}
 G^\cu\ot_{B^*}\Hom^\cu_{B^\rop{}^*}(E^\cu,B^\cu)
 \simeq\Hom^\bu_{B^\rop{}^*}(E^\cu,G^\cu).
\end{equation}
 For any complex of left CDG\+modules $D^{\bu,\cu}$ and any right
CDG\+module $G^\cu$ over $B^\cu$, there is a natural isomorphism of
complexes of abelian groups
\begin{equation} \label{tensor-with-direct-sum-totalization-eqn}
 G^\cu\ot_{B^*}\Tot^\sqcup(D^{\bu,\cu})\simeq
 \Tot^\sqcup(G^\cu\ot_{B^*}D^{\bu,\cu}).
\end{equation}
 Here $G^\cu\ot_{B^*}D^{\bu,\cu}$ is the complex of complexes of
abelian groups whose terms are the complexes
$G^\cu\ot_{B^*}D^{n,\cu}$, \,$n\in\boZ$.

 Returning to the situation at hand,
by Theorem~\ref{coderived-cotorsion-pair-category-theorem}
there exists a graded-injective right CDG\+module $J^\cu$ over $B^\cu$
together with a closed morphism of right CDG\+modules
$N^\cu\rarrow J^\cu$ with a coacyclic cone.
 Following the discussion of the derived functor of tensor
product~$\ot_{B^*}^\boL$ in
Section~\ref{tensor-product-pairing-subsecn}, and specifically
according to the assertion~(2), the induced morphism of complexes
of abelian groups
$$
 N^\cu\ot_{B^*}S^\cu\lrarrow J^\cu\ot_{B^*}S^\cu
$$
is a quasi-isomorphism.
 Now we have natural isomorphisms of abelian groups
\begin{multline} \label{long-isomorphism-of-hom-and-tensor-cohomology}
 \Hom_{\sD^\bctr(B^\cu\bModl)}(Q^\cu,S^\cu)\simeq
 H^0(N^\cu\ot_{B^*}S^\cu)\simeq H^0(J^\cu\ot_{B^*}S^\cu) \\
 \simeq H^0\Tot^\sqcup(J^\cu\ot_{B^*}S^{\bu,\cu})\simeq
 H^0\Tot^\sqcup(\Hom^\bu_{B^\rop{}^*}(R_\bu^\cu,J^\cu))
\end{multline}
by the formulas~\eqref{hom-of-Xi-as-tensor-product-eqn},
\eqref{tensor-with-direct-sum-totalization-eqn},
and~\eqref{tensor-with-hom-from-fin-gen-graded-proj-eqn}
for $D^{\bu,\cu}=S^{\bu,\cu}$ and $E^\cu=R_n^\cu$, \,$n\ge0$.

 Finally, the bounded below complex of complexes of abelian groups
\begin{equation} \label{acyclic-complex-of-homs-eqn}
 0\rarrow\Hom^\bu_{B^\rop{}^*}(K^\cu,J^\cu)\rarrow 
 \Hom^\bu_{B^\rop{}^*}(R_0^\cu,J^\cu)\rarrow
 \Hom^\bu_{B^\rop{}^*}(R_1^\cu,J^\cu)\rarrow\dotsb
\end{equation}
is acyclic, since the complex of right
CDG\+modules~\eqref{fin-gen-graded-projective-resolution-2nd} is
acyclic in $\sZ^0(\bModr B^\cu)$ and the right CDG\+module $J^\cu$
is graded-injective.
 By Lemma~\ref{co-contra-acyclic-total-complex-lemma}(a),
the direct sum totalization of
the bicomplex~\eqref{acyclic-complex-of-homs-eqn} is
an acyclic complex of abelian groups.
 Hence the induced map of abelian groups
\begin{equation} \label{totalized-bicomplex-induced-map}
 H^0\Hom^\bu_{B^\rop{}^*}(K^\cu,J^\cu)\lrarrow
 H^0\Tot^\sqcup(\Hom^\bu_{B^\rop{}^*}(R_\bu^\cu,J^\cu))
\end{equation}
is an isomorphism.
 It remains to recall that the map
\begin{multline*}
 H^0\Hom^\bu_{B^\rop{}^*}(K^\cu,J^\cu)=
 \Hom_{\sH^0(\bModr B^\cu)}(K^\cu,J^\cu) \\ \lrarrow
 \Hom_{\sD^\bco(\bModr B^\cu)}(K^\cu,J^\cu)\simeq
 \Hom_{\sD^\bco(\bModr B^\cu)}(K^\cu,N^\cu)
\end{multline*}
induced by the Verdier quotient functor
$\sH^0(\bModr B^\cu)\rarrow\sD^\bco(\bModr B^\cu)$ is
an isomorphism by the definition of the coderived category,
since the right CDG\+module $J^\cu$ over $B^\cu$ is
graded-injective.
\end{proof}

 Now we can deduce the next theorem, which is the second
main result of Section~\ref{compact-generators-of-contraderived-secn}.

\begin{thm} \label{Xi-fully-faithful-theorem}
 Let $B^\cu=(B^*,d,h)$ be a CDG\+ring such that the graded ring $B^*$
is graded right coherent.
 Then the contravariant triangulated
functor~\eqref{functor-Xi-on-absolute-derived-category}
$$
 \Xi_{B^\cu}\:\sD^\abs(\bmodr B^\cu)^\sop
 \lrarrow\sD^\bctr(B^\cu\bModl)
$$
from Corollary~\ref{factorizes-through-absolute-derived-cor}
is fully faithful.
\end{thm}

\begin{proof}
 Let $N^\cu$ and $K^\cu$ be two finitely presented right CDG\+modules
over~$B^\cu$.
 By Proposition~\ref{Hom-from-Xi-to-Xi-prop}, we have a natural
isomorphism of abelian groups
\begin{equation} \label{Hom-from-Xi-to-Xi-isomorphism-eqn}
 \Hom_{\sD^\bctr(B^\cu\bModl)}(\Xi_{B^\cu}(N^\cu),\Xi_{B^\cu}(K^\cu))
 \simeq
 \Hom_{\sD^\bco(\bModr B^\cu)}(K^\cu,N^\cu).
\end{equation}
 The construction in the proof of the proposition shows
that~\eqref{Hom-from-Xi-to-Xi-isomorphism-eqn} is an isomorphism of
functors on the category $\sZ^0(\bmodr B^\cu)\times
\sZ^0(\bmodr B^\cu)^\sop$, but both the left-hand side and
the right-hand side of~\eqref{Hom-from-Xi-to-Xi-isomorphism-eqn}
are actually functors on $\sD^\abs(\bmodr B^\cu)\times
\sD^\abs(\bmodr B^\cu)^\sop$.
 As the triangulated category $\sD^\abs(\bmodr B^\cu)$ can be obtained
from the abelian category $\sZ^0(\bmodr B^\cu)$ by inverting some
morphisms (specifically, those with absolutely acyclic cones), it
follows that~\eqref{Hom-from-Xi-to-Xi-isomorphism-eqn} is actually
an isomorphism of functors on $\sD^\abs(\bmodr B^\cu)\times
\sD^\abs(\bmodr B^\cu)^\sop$.

 Furthermore, by Theorem~\ref{compact-generators-of-coderived-theorem},
the triangulated functor $\sD^\abs(\bmodr B^\cu)\rarrow
\sD^\bco(\bModr B^\cu)$ \,\eqref{absolute-derived-into-coderived}
induced by the inclusion of DG\+categories $\bmodr B^\cu\rarrow
\bModr B^\cu$ is fully faithful.
 So the functor~\eqref{absolute-derived-into-coderived} induces
an isomorphism of  the abelian groups of morphisms
\begin{equation} \label{hom-abs-derived=hom-coderived}
 \Hom_{\sD^\abs(\bmodr B^\cu)}(K^\cu,N^\cu)\simeq
 \Hom_{\sD^\bco(\bModr B^\cu)}(K^\cu,N^\cu).
\end{equation}
 We need to show that the map of abelian groups
\begin{equation} \label{map-of-homs-induced-by-Xi-eqn}
 \Xi_{B^\cu}\:\Hom_{\sD^\abs(\bmodr B^\cu)}(K^\cu,N^\cu)\lrarrow
 \Hom_{\sD^\bctr(B^\cu\bModl)}(\Xi_{B^\cu}(N^\cu),\Xi_{B^\cu}(K^\cu))
\end{equation}
is an isomorphism.
 For this purpose, it suffices to check that the triangular diagram
of abelian group maps
\begin{equation} \label{triangular-diagram-of-homs}
\begin{gathered}
 \xymatrixcolsep{-3.5em}
 \xymatrix{
  & \Hom_{\sD^\abs(\bmodr B^\cu)}(K^\cu,N^\cu)
  \ar[ld]_-{\eqref{map-of-homs-induced-by-Xi-eqn}}
  \ar@{=}[rd]^-{\eqref{hom-abs-derived=hom-coderived}} \\
  \Hom_{\sD^\bctr(B^\cu\bModl)}(\Xi_{B^\cu}(N^\cu),\Xi_{B^\cu}(K^\cu))
  \ar@{=}[rr]_-{\eqref{Hom-from-Xi-to-Xi-isomorphism-eqn}} &&
  \Hom_{\sD^\bco(\bModr B^\cu)}(K^\cu,N^\cu)
 }
\end{gathered}
\end{equation}
is commutative.

 In fact, as we have established in the discussion above,
\eqref{triangular-diagram-of-homs}~is a diagram of
natural transformations of functors on the category
$\sD^\abs(\bmodr B^\cu)\times\sD^\abs(\bmodr B^\cu)^\sop$.
 As any morphism in the absolute derived category
$\sD^\abs(\bmodr B^\cu)$ is isomorphic to a morphism coming from
the abelian category $\sZ^0(\bmodr B^\cu)$ via the localization
functor $\sZ^0(\bmodr B^\cu)\rarrow\sD^\abs(\bmodr B^\cu)$,
it suffices to check that, for any closed morphism of CDG\+modules
$f\:K^\cu\rarrow N^\cu$, the two images of~$f$ in
$\Hom_{\sD^\bco(\bModr B^\cu)}(K^\cu,N^\cu)$ obtained via
the two routes in~\eqref{triangular-diagram-of-homs} agree.

 The latter property can be checked following the construction of
the isomorphism in Proposition~\ref{Hom-from-Xi-to-Xi-prop}
step by step.
 Choose resolutions $P_\bu^\cu$
\,\eqref{fin-gen-graded-projective-resolution-eqn} of
the CDG\+module $N^\cu$ and $R_\bu^\cu$
\,\eqref{fin-gen-graded-projective-resolution-2nd}
of the CDG\+module $K^\cu$ so that the morphism $f\:K^\cu\rarrow
N^\cu$ can be extended to a morphism of complexes of CDG\+modules
$\tilde f\:R_\bu^\cu\rarrow P_\bu^\cu$ forming a commutative
diagram with~$f$ in the abelian category $\sZ^0(\bmodr B^\cu)$.

 Then there is a naturally defined degree~$0$ cocycle~$c$ in
the complex $N^\cu\ot_{B^*}S^{0,\cu}$ corresponding to
the composition of closed morphisms of CDG\+modules
$g\:R_0^\cu\rarrow K^\cu\rarrow N^\cu$.
 The image of~$c$ vanishes in $N^\cu\ot_{B^*}S^{1,\cu}$ (since
the composition of morphisms of CDG\+modules $R_1^\cu\rarrow R_0^\cu
\rarrow N^\cu$ vanishes), so $c\in N^\cu\ot_{B^*}S^{0,\cu}\subset
N^\cu\ot_{B^*}S^\cu$ is actually a degree~$0$ cocycle
in the complex $N^\cu\ot_{B^*}S^\cu$.
 The cocycle~$c$ represents the cohomology class corresponding to
the morphism $\Xi_{B^\cu}(f)\:\Xi_{B^\cu}(N^\cu)\rarrow
\Xi_{B^\cu}(K^\cu)$ under
the isomorphism~\eqref{hom-of-Xi-as-tensor-product-eqn}.

 Continuing to follow the proof of
Proposition~\ref{Hom-from-Xi-to-Xi-prop}, denote by $j\:N^\cu
\rarrow J^\cu$ our chosen closed morphism of CDG\+modules with
coacyclic cone.
 Then $jg\:R_0^\cu\rarrow J^\cu$ is a degree~$0$ cocycle in
the complex $\Hom_{B^\rop{}^*}^\bu(R_0^\cu,J^\cu)$.
 Moreover, $jg\in\Hom_{B^\rop{}^*}^0(R_0^\cu,J^\cu)\subset
\Tot^\sqcup(\Hom^\bu_{B^\rop{}^*}(R_\bu^\cu,J^\cu))$ is
actually a degree~$0$ cocycle in the complex
$\Tot^\sqcup(\Hom^\bu_{B^\rop{}^*}(R_\bu^\cu,J^\cu))$ (for
the same reason that the composition of morphisms of CDG\+modules
$R_1^\cu\rarrow R_0^\cu\rarrow J^\cu$ vanishes).
 The cohomology class of the cocycle~$c$ corresponds to
the cohomology class of the cocycle~$jg$ under the isomorphism
in~\eqref{long-isomorphism-of-hom-and-tensor-cohomology}.

 Finally, the cocycle~$jg$ is the image of the cocycle
$jf\in\Hom_{B^\rop{}^*}^0(K^\cu,J^\cu)$ under the map of complexes
arising from~\eqref{acyclic-complex-of-homs-eqn} and inducing
the cohomology isomorphism~\eqref{totalized-bicomplex-induced-map}.
 We have shown that the cocycle~$jf$ represents the image of
the morphism $\Xi_{B^\cu}(f)$ under
the isomorphism~\eqref{Hom-from-Xi-to-Xi-isomorphism-eqn}
constructed in Proposition~\ref{Hom-from-Xi-to-Xi-prop}.
 The same cocycle $jf\in\Hom_{B^\rop{}^*}^0(K^\cu,J^\cu)$ represents
the image of~$f$ under
the isomorphism~\eqref{hom-abs-derived=hom-coderived}.
\end{proof}

\subsection{Compact generators of the coderived and
contraderived categories}
 The following corollary summarizes our results on compact generators
of the coderived and contraderived categories of CDG\+modules.

 Given a triangulated category $\sS$, we denote by $\overline\sS$
the idempotent completion of~$\sS$.
 Given a triangulated category $\sT$ with coproducts, we denote by
$\sT^\cmp\subset\sT$ the full triangulated subcategory of compact
objects in~$\sT$.
In view of Theorem~\ref{compact-generators-of-coderived-theorem}, 
we also denote by
\begin{equation} \label{RHom-on-Dabs-definition}
\boR\Hom_{B^*}^\bu\:\sD^\abs(\bmodr B^\cu)^\sop\times
\sD^\abs(\bmodr B^\cu)\lrarrow\sD(\Ab)
\end{equation}
the functor given by the formula
$\boR\Hom_{B^*}^\bu(X^\cu,Y^\cu):=\Hom_{B^*}^\bu(X^\cu,I^\cu(Y^\cu))$,
where $A^\cu(Y^\cu)\to Y^\cu\to I^\cu(Y^\cu)\to A^\cu(Y^\cu)[1]$ is the
(unique up to a unique isomorphism) triangle in $\sH^0(\bModr B^\cu)$
corresponding to the semiorthogonal decomposition to the classes
$\sH^0(B^\cu\bModl_\binj)$ and $\sH^0(B^\cu\bModl)_\ac^\bco$
(recall Section~\ref{coderived-subsecn}).
This is the best available candidate for the right derived functor of
$\Hom_{B^*}^\bu$ on $\sD^\abs(\bmodr B^\cu)$, as explained in
Remark~\ref{RHom-for-Dabs-remark} below.

\begin{cor} \label{summary-on-compact-generators-cor}
 Let $B^\cu=(B^*,d,h)$ be a CDG\+ring such that the graded ring $B^*$
is graded right coherent.
 Then both the coderived category of right CDG\+modules\/
$\sD^\bco(\bModr B^\cu)$ and the contraderived category of left
CDG\+modules\/ $\sD^\bctr(B^\cu\bModl)$ over $B^\cu$ are compactly
generated.

 The full triangulated subcategory of compact objects in
the coderived category\/ $\sD^\bco(\bModr B^\cu)$ is equivalent to
the idempotent completion of the absolute derived category of
finitely presented right CDG\+modules\/ $\sD^\abs(\bmodr B^\cu)$,
\begin{equation} \label{compact-objects-of-coderived}
 \sD^\bco(\bModr B^\cu)^\cmp\simeq
 \overline{\sD^\abs(\bmodr B^\cu)}.
\end{equation}
 The full triangulated subcategory of compact objects in
the contraderived category\/ $\sD^\bctr(B^\cu\bModl)$ is
anti-equivalent to the idempotent completion of the absolute
derived category\/ $\sD^\abs(\bmodr B^\cu)$,
\begin{equation} \label{compact-objects-of-contraderived}
 \sD^\bctr(B^\cu\bModl)^\cmp\simeq
 \overline{\sD^\abs(\bmodr B^\cu)}^\sop.
\end{equation}
 Thus the full triangulated subcategories of compact objects in\/
$\sD^\bco(\bModr B^\cu)$ and\/ $\sD^\bctr(B^\cu\bModl)$ are
opposite to each other,
$$
 \sD^\bctr(B^\cu\bModl)^\cmp\simeq
 (\sD^\bco(\bModr B^\cu)^\cmp)^\sop.
$$
 The resulting connection between the coderived category\/
$\sD^\bco(\bModr B^\cu)$ and the contraderived category\/
$\sD^\bctr(B^\cu\bModl)$ can be expressed by the tensor product
pairing functor
$$
 \ot_{B^*}^\boL\:\sD^\bco(\bModr B^\cu)\times
 \sD^\bctr(B^\cu\bModl)\lrarrow\sD(\Ab),
$$
as per Proposition~\ref{Hom-from-tilde-Xi-prop},
which, using the identifications \eqref{compact-objects-of-coderived} and~\eqref{compact-objects-of-contraderived}, restricts to
the functor
$$
 \boR\Hom_{B^*}^\bu(-,-)\:\sD^\abs(\bmodr B^\cu)^\sop\times
 \sD^\abs(\bmodr B^\cu)\lrarrow\sD(\Ab).
$$
\end{cor}

\begin{proof}
 By Theorem~\ref{compact-generators-of-coderived-theorem},
the covariant triangulated functor $\sD^\abs(\bmodr B^\cu)\rarrow
\sD^\bco(\bModr B^\cu)$ \,\eqref{absolute-derived-into-coderived}
induced by the inclusion of DG\+categories $\bmodr B^\cu\rarrow
\bModr B^\cu$ is fully faithful, and (a set of representatives
of the isomorphism classes of objects in) its image is a set
of compact generators of $\sD^\bco(\bModr B^\cu)$.
 In view of the discussion in the beginning of
Section~\ref{compact-generators-of-coderived-subsecn}, a triangulated
equivalence~\eqref{compact-objects-of-coderived} follows.

 By Theorem~\ref{Xi-fully-faithful-theorem}, the contravariant
triangulated functor $\Xi_{B^\cu}\:\sD^\abs(\bmodr B^\cu)^\sop
\allowbreak\rarrow\sD^\bctr(B^\cu\bModl)$
\,\eqref{functor-Xi-on-absolute-derived-category} from
Corollary~\ref{factorizes-through-absolute-derived-cor} is
fully faithful.
 By Theorem~\ref{image-of-Xi-compact-generators-thm},
(a~set of representatives of the isomorphism classes of objects in)
the image of $\Xi_{B^\cu}$ is a set of compact generators of
$\sD^\bctr(B^\cu\bModl)$.
 In view of the same discussion in the beginning of
Section~\ref{compact-generators-of-coderived-subsecn}, a triangulated
equivalence~\eqref{compact-objects-of-contraderived} follows.

 To sum up:
the equivalence~\eqref{compact-objects-of-coderived}
is induced by the inclusion of the DG\+category of finitely
presented right CDG\+modules $\bmodr B^\cu$ into the DG\+category
of all right CDG\+modules $\bModr B^\cu$.
 The equivalence~\eqref{compact-objects-of-contraderived} is induced
by the contravariant functor~$\Xi_{B^\cu}$ constructed in
Section~\ref{totalized-dual-construction-subsecn}
and Corollary~\ref{factorizes-through-absolute-derived-cor}.
 Finally, the identification of the pairing $-\ot_{B^*}^\boL-$
with $\boR\Hom_{B^*}^\bu(-,-)$ is obtained from
Remark~\ref{RHom-from-tilde-Xi-prop}.
\end{proof}

\begin{rem} \label{RHom-for-Dabs-remark}
The notation in~\eqref{RHom-on-Dabs-definition} suggests that
the functor $\boR\Hom_{B^*}^\bu$ satisfies a certain universal property.
Usually, if $\sA$, $\sB$ are categories such that $\sA$ has direct limits
and $\sB$ is essentially small, $\cW$ is a right multiplicative system
of morphisms in $\sB$ in the sense of~\cite[Definition~7.1.5]{KS},
and $F\:\sB\to\sA$ is a functor,
one denotes by $\boR F\:\sB[\cW^{-1}]\to\sA$ the functor given by
$$
\boR F(Y) = \varinjlim_{Y\to Z} F(Z),
$$
where the colimit runs over all homomorphims in $\cW$ originating
in $Y$. This is explained in~\cite[Proposition~7.3.3]{KS} and the
colimit is actually filtered (so exists in $\sA$) by
\cite[Proposition~7.1.10]{KS}.

We would wish to apply this recipe to a functor of the form
\begin{equation} \label{Hom-functor-on-Dabs-to-derive}
\Hom^\bu_{B^*}(X^\cu,-)\:\sH^0(\bmodr B^\cu)\lrarrow\sD(\Ab)
\end{equation}
with $X^\cu\in\sH^0(\bmodr B^\cu)$ and to the class $\cW$ of
morphisms with absolutely acyclic cones, so that
$\sD^\abs(\bmodr B^\cu)=\sH^0(\bmodr B^\cu)[\cW^{-1}]$.
However, the purported colimits exist only in the enhanced setting
(i.~e.\ as suitable homotopy colimits on the underlying
DG\+categories) in this setup.
More in detail, there is a representative of the
morphism $Y^\cu\to I^\cu(Y^\cu)$ in $\sZ^0(\bModr B^\cu)$
which is a direct limit of monomorphisms with absolutely acyclic cokernels there thanks to~\cite[Proposition~8.13]{PS5},
and $\Hom^\bu(X^\cu,-)$ preserves direct limits as a functor
from $\sZ^0(\bModr B^\cu)$ to $\sC(\Ab)$.
However, the actual universal property of the right derived functor
from \cite[Definition~7.3.1(i)]{KS} might not be provable in general
using only the framework of triangulated categories.
\end{rem}

\subsection{Examples}
\label{examples-compacts-subsecn}

Here we continue the examples from Section~\ref{examples-Dbctr-subsecn} to illustrate the results about the pairing between (Becker) coderived and contraderived categories and the compact generators.

\begin{ex}[Finite-dimensional CDG\+rings] \label{finite-dimensional-algebra-example-contd}
 Let $B^\cu=(B^*,d,h)$ be a finite-dimensional CDG\+algebra over
a field~$k$, as in Example~\ref{finite-dimensional-algebra-example}.
 Then we have the graded $k$\+vector space dual CDG\+coalgebra
$\cC^\cu=(\cC^*,d,h)=B^\cu{}\spcheck$ over~$k$.
 
 Quite generally, for any CDG\+coalgebra $\cC^\cu$ over~$k$,
the coderived category of CDG\+comodules $\sD^\bco(\cC^\cu\bComodl)$
is compactly generated, as explained in~\cite[Section~5.5]{Pkoszul},
\cite[Remark~7.14]{Pksurv}.
 The compact generators of $\sD^\bco(\cC^\cu\bComodl)$ are
the finite-dimensional CDG\+comodules; so the full triangulated
subcategory of compact objects in $\sD^\bco(\cC^\cu\bComodl)$ is
equivalent to the idempotent completion of the absolute derived
category of finite-dimensional CDG\+comodules
$\sD^\abs(\cC^\cu\bcomodl)$ \,\cite[Theorem~4.6]{Pkoszul}.
 Here a CDG\+comodule $\cM^\cu$ is said to be (totally)
finite-dimensional if $\cM^i=0$ for all but a finite set of degrees~$i$
and $\cM^i$ is a finite-dimensional $k$\+vector space for all degrees
$i\in\boZ$.

 In view of the triangulated equivalence of derived
comodule-contramodule
correspondence~\eqref{coalgebra-co-contra-corresp}, it follows that
the contraderived category $\sD^\bctr(\cC^\cu\bContra)$ is also
compactly generated, and its full triangulated subcategory of
compact objects is equivalent to the idempotent completion of
$\sD^\abs(\cC^\cu\bcomodl)$.
 However, the description of the compact objects in
the contraderived category $\sD^\bctr(\cC^\cu\bContra)$ through
the compact objects in the coderived category
$\sD^\bco(\cC^\cu\bComodl)$ and the triangulated
equivalence~\eqref{coalgebra-co-contra-corresp} is a bit indirect.
 Can one ``see'' the compact objects in the contraderived
category $\sD^\bctr(\cC^\cu\bContra)$ without passing to
the coderived category?

 The results of this paper answer this question in the case of
a finite-dimensional CDG\+coalgebra~$\cC^\cu$.
 Specifically, let $B^\cu$ be a finite-dimensional CDG\+algebra
over~$k$ and $\cC^\cu=B^\cu{}\spcheck$.
 Let us point out that a graded $B^*$\+module is finitely presented
if and only if it is (totally) finite-dimensional over~$k$.
 We claim that there is a commutative diagram of fully faithful
triangulated functors and triangulated equivalences
\begin{equation} \label{finite-dim-algebra-main-comm-diagram}
\begin{gathered}
 \xymatrix{
  \sD^\abs(B^\cu\bmodl) \ar@{=}^{\vee}[r] \ar@{>->}[d]
  & \sD^\abs(\bmodr B^\cu)^\sop \ar@{>->}[d]^{\Xi_{B^\subcu}} \\
  \sD^\bco(B^\cu\bModl) \ar@{=}[r] & \sD^\bctr(B^\cu\bModl)
 }
\end{gathered}
\end{equation}
 Here the triangulated anti-equivalence in the upper horizontal line
takes a finite-dimensional right CDG\+module $N^\cu$ over $B^\cu$ to
its graded dual $k$\+vector space $L^\cu=N^\cu{}\spcheck=
\Hom_k^*(N^\cu,k)$, which is a finite-dimensional left CDG\+module
over~$B^\cu$.
 The triangulated equivalence~\eqref{fin-dim-algebra-co-contra-corresp}
in the lower horizontal line is provided by the derived functors
$\boR\Hom^\cu_{B^*}(B^\cu{}\spcheck,{-})$ and
$B^\cu{}\spcheck\ot_{B^*}^\boL{-}$.
 The leftmost vertical covariant fully faithful functor is induced by
the inclusion of DG\+categories $B^\cu\bmodl\rarrow B^\bu\bModl$,
as in Theorem~\ref{compact-generators-of-coderived-theorem}.
 The rightmost vertical contravariant fully faithful functor
$\Xi_{B^\cu}$ was constructed in
Corollary~\ref{factorizes-through-absolute-derived-cor}.

 Before proving the commutativity
of~\eqref{finite-dim-algebra-main-comm-diagram}, let us first recall
the construction of the functor (triangulated equivalence)
$\boR\Hom^\cu_{B^*}(B^\cu{}\spcheck,{-})\:\sD^\bco(B^\cu\bModl)
\rarrow\sD^\bctr(B^\cu\bModl)$.
 Let $M^\cu$ be a left CDG\+module over $B^\cu$, viewed as an object
of $\sD^\bco(B^\cu\bModl)$.
 Following~\cite[proofs of Theorems~3.7, 4.4(c) and~5.2]{Pkoszul},
one starts with choosing a coresolution
$0\rarrow M^\cu\rarrow J^{0,\cu}\rarrow J^{1,\cu}\rarrow J^{2,\cu}
\rarrow\dotsb$ of the CDG\+module $M^\cu$ by graded-injective
CDG\+modules $J^{n,\cu}$ in the abelian category $\sZ^0(B^\cu\bModl)$.
 The direct sum totalization $J^\cu=\Tot^\sqcup(J^{\bu,\cu})$
is a left CDG\+module over $B^\cu$ naturally isomorphic to $M^\cu$
in the coderived category $\sD^\bco(B^\cu\bModl)$.
 The left CDG\+module $\Hom^\cu_{B^*}(B^\cu{}\spcheck,J^\cu)$ over
$B^\cu$ represents the object of $\sD^\bctr(B^\cu\bModl)$ corresponding
to the object $M^\cu\in\sD^\bco(B^\cu\bModl)$.

 Now we start with a finite-dimensional right CDG\+module $N^\cu$
over~$B^\cu$, and pick its resolution $\dotsb\rarrow P_2^\cu\rarrow
P_1^\cu\rarrow P_0^\cu\rarrow N^\cu\rarrow0$ by finite-dimensional
graded-projective right CDG\+modules $P_n^\cu$ in the abelian
category $\sZ^0(\bmodr B^\cu)$.
 Passing to the graded dual $k$\+vector spaces, we obtain
a coresolution $0\rarrow L^\cu\rarrow P_0^\cu{}\spcheck\rarrow
P_1^\cu{}\spcheck\rarrow P_2^\cu{}\spcheck\rarrow\dotsb$ of
the left CDG\+module $L^\cu=N^\cu{}\spcheck$ by graded-injective
(finite-dimensional) left CDG\+modules $P_n^\cu{}\spcheck$
in the abelian category $\sZ^0(B^\cu\bmodl)\subset\sZ^0(B^\cu\bModl)$.
 Setting $J^{n,\cu}=P_n^\cu{}\spcheck$ and
$J^\cu=\Tot^\sqcup(J^{\bu,\cu})$, we have
$\boR\Hom_{B^*}^\cu(B^\cu{}\spcheck,L^\cu)=
\Hom_{B^*}^\cu(B^\cu{}\spcheck,J^\cu)$.
 On the other hand, putting $Q^{n,\cu}=
\Hom_{B^\rop{}^*}^\cu(P_n^\cu,B^\cu)$, we have
$\Xi_{B^\cu}(N^\cu)=Q^\cu=\Tot^\sqcup(Q^{\bu,\cu})$
by the definition of the functor $\Xi_{B^\cu}$.

 We claim that there is a natural closed isomorphism
$\Hom_{B^*}^\cu(B^\cu{}\spcheck,J^\cu)\simeq Q^\cu$ of
left CDG\+modules over~$B^\cu$.
 Indeed, for every $n\ge0$, we have
$$
 \Hom_{B^*}^\cu(B^\cu{}\spcheck,J^{n,\cu})=
 \Hom_{B^*}^\cu(B^\cu{}\spcheck,P_n^\cu{}\spcheck)
 \simeq\Hom_{B^\rop{}^*}^\cu(P_n^\cu,B^\cu)=Q^{n,\cu},
$$
since $B^*$ (as well as~$P_n^*$) is a finite-dimensional graded
$k$\+vector space.
 Since $B^*{}\spcheck$ is a finite-dimensional graded $k$\+vector
space, too, the functor $\Hom_{B^*}^\cu(B^\cu{}\spcheck,{-})$
commutes with direct sum totalizations, so we have
$$
 \Hom_{B^*}^\cu(B^\cu{}\spcheck,\Tot^\sqcup(J^{\bu,\cu}))
 \simeq\Tot^\sqcup(\Hom_{B^*}^\cu(B^\cu{}\spcheck,J^{\bu,\cu}))
 =\Tot^\sqcup(Q^{\bu,\cu}).
$$
 This finishes the proof of commutativity of
the diagram~\eqref{finite-dim-algebra-main-comm-diagram}.

 The commutative diagram~\eqref{finite-dim-algebra-main-comm-diagram}
provides two equivalent descriptions of the compact objects in
the contraderived category $\sD^\bctr(B^\cu\bModl)$.
 The description via the derived comodule-contramodule correspondence
is more general, in that it works for infinite-dimensional
CDG\+coalgebras $\cC^\cu$, too.
 However, the description via the functor~$\Xi_{B^\cu}$, as per
Theorems~\ref{image-of-Xi-compact-generators-thm}
and~\ref{Xi-fully-faithful-theorem}, may be a bit more explicit.
\end{ex}

\begin{ex}[Algebraic de~Rham complexes] \label{de-Rham-complex-example-contd}
 Let $X$ be a smooth affine algebraic variety over a field~$k$ of
characteristic zero.
 Consider the de~Rham DG\+algebra of differential forms
$\Omega^\bu(X/k)$, as in Example~\ref{de-Rham-complex-example}.
 Then there are several indirect ways to describe the compact
generators of the contraderived category
$\sD^\bctr(\Omega^\bu(X/k)\bModl)$.
 In particular, according to~\eqref{D-Omega-duality}, there is
a triangulated equivalence $\sD^\bctr(\Omega^\bu(X/k)\bModl)
\simeq\sD^\bco(\Omega^\bu(X/k)\bModl)$, and the compact generators
of the coderived category $\sD^\bco(\Omega^\bu(X/k)\bModl)$ are
the finitely generated DG\+modules over $\Omega^\bu(X/k)$,
as per~\cite[Section~3.11]{Pkoszul} and
Theorem~\ref{compact-generators-of-coderived-theorem} above.

 Furthermore, according to~\eqref{D-Omega-duality}, there is
also a triangulated equivalence of $D$\+$\Omega$ duality
$\sD^\bctr(\Omega^\bu(X/k)\bModl)\simeq\sD(D(X/k)\rModl)$.
 The derived category of modules over any associative ring $A$ is
compactly generated by perfect complexes of $A$\+modules, i.~e.,
essentially, bounded complexes of finitely generated projective
$A$\+modules.
 In the case of $A=D(X/k)$, the associative algebra of differential
operators $D(X/k)$ is Noetherian and has finite homological dimension
by~\cite[Proposition~4.2 and Theorem~4.3]{Ch}
(cf.~\cite[Theorem~3.7]{Sm}); so all bounded complexes of finitely
generated $D(X/k)$\+modules are quasi-isomorphic to bounded complexes
of finitely generated projective $D(X/k)$\+modules.
 The derived category of such complexes,
$\sD^\bb(D(X/k)\rmodl_\proj)\simeq\sD^\bb(D(X/k)\rmodl)$, is
the full subcategory of compact objects in $\sD(D(X/k)\rModl)$.

 The description of the full subcategory of compact
objects/generators of the contraderived
category $\sD^\bctr(\Omega^\bu(X/k)\bModl)$ provided by
Theorems~\ref{image-of-Xi-compact-generators-thm}
and~\ref{Xi-fully-faithful-theorem} may be a bit more explicit.

 Much more generally, let $X\rarrow S$ be a morphism of finite
presentation between affine schemes.
 Assume that $X$ is a coherent affine scheme
(cf.~\cite[Section~9.1]{PS5}).
 Then the de~Rham DG\+algebra $\Omega^\bu(X/S)$ of (fiberwise)
differential forms on $X$ relative to $S$, as defined
in~\cite[Section~16.6]{EGAIV} or~\cite[Section Tag~0FKF]{SP}, has
a (left and right) coherent underlying graded ring $\Omega^*(X/S)$.
 Indeed, $\Omega^*(X/S)$ is a finitely presented graded module
over the ring of functions $O(X)=\Omega^0(X/S)$, which is coherent.
 So Theorem~\ref{compact-generators-of-coderived-theorem} provides
a description of the compact generators of the coderived category
$\sD^\bco(\Omega^\bu(X/S)\bModl)$, while
Theorems~\ref{image-of-Xi-compact-generators-thm}
and~\ref{Xi-fully-faithful-theorem} describe the compact generators
of the contraderived category $\sD^\bctr(\Omega^\bu(X/S)\bModl)$
(with Corollary~\ref{summary-on-compact-generators-cor} providing
a summary and additional details).
\end{ex}

\begin{ex}[Matrix factorizations] \label{matrix-factorizations-example-contd}
 Continuing the discussion in
Example~\ref{matrix-factorizations-example}, consider a right coherent
ring~$R$.
 If $R$ is commutative, then $X=\Spec R$ is a coherent affine scheme
(cf.~\cite[Section~9.1]{PS5}), and one can view finitely
presented $R$\+modules as coherent sheaves on $\Spec R$.

 Let $L$ be an invertible $R$\+$R$\+bimodule and $w\in L$ be a central element.
 Consider the CDG\+ring $B^\cu=(B^*,0,w)$ as in
Example~\ref{matrix-factorizations-example}; so the left CDG\+modules
over $B^\cu$ are interpreted as left $R$\+module matrix factorizations of the potential~$w$. Similarly, right CDG\+modules
over $B^\cu$ can be interpreted as right $R$\+module matrix factorizations of the potential~$-w$.
 Let us say that a matrix factorization $M^\cu=(M^*,d_M)$ of right modules is \emph{coherent} if the $R$\+modules $M^0$ and $M^1$ are finitely presented.
 For $R$ commutative, this is the affine particular case of the definition of a coherent (matrix) factorization in~\cite[Section~9.2]{PS5}.

 According to Theorem~\ref{image-of-Xi-compact-generators-thm}
and~\ref{Xi-fully-faithful-theorem} (or
Corollary~\ref{summary-on-compact-generators-cor}), the homotopy
category of projective matrix factorizations of left modules is compactly generated
by the opposite category to the absolute derived category of
coherent matrix factorizations of right modules.
 In the case of a Noetherian scheme $X$ with a dualizing complex,
an analogous result can be obtained by comparing~\cite[Corollary~2.3(l)
or Proposition~1.5(d)]{EP} with~\cite[Proposition~2.5 and
Theorem~2.5]{EP}.
 See~\cite[Corollary~2.5]{EP} for the details.
 Our result in this paper is much more general in that it does not
require a dualizing complex, it does not presume Noetherianity
but only coherence, and it works for non-commutative rings.
However, even in the case of commutative rings it is also less general in that it is only applicable to affine schemes.
\end{ex}

\end{document}